\documentclass[twoside,11pt,dvipsnames,preprint]{article}
\usepackage{jmlr2e}
\usepackage{amsmath}
\usepackage{amsfonts}
\usepackage[title]{appendix}
\usepackage[normalem]{ulem}
\usepackage{cleveref}
\usepackage{xcolor}
\usepackage{verbatim}
\usepackage{algorithm,algpseudocode}
\usepackage{longtable,tabularx,mathtools,leftidx,multirow}
\usepackage{amssymb}
\usepackage{stmaryrd}
\SetSymbolFont{stmry}{bold}{U}{stmry}{m}{n}
\usepackage{subcaption}
\usepackage{pgfplots}
\pgfplotsset{compat=1.16}
\usepgfplotslibrary{groupplots}
\usepackage{tikz}
\usepackage{adjustbox}
\usepackage{mathrsfs}
\usepackage{bm}
\usepackage{enumitem}
\algblock{Input}{EndInput}
\algnotext{EndInput}
\algblock{Output}{EndOutput}
\algnotext{EndOutput}

\newcommand{\Desc}[2]{\State \makebox[10em][l]{#1}#2}

\newcommand{\newdim}{N}
\newcommand{\cumuldim}{\mathbf{n}}
\newcommand{\reals}{\mathbb{R}}

\newcommand{\node}{\mathbf{N}}
\newcommand{\core}{\mathcal{G}}

\newcommand{\tree}{\mathbf{T}}
\newcommand{\indexset}{\mathbf{I}}

\newcommand{\parents}{\mathbf{P}}
\newcommand{\successors}{\mathbf{S}}
\newcommand{\ranks}{\mathbf{R}}
\newcommand{\rank}{r}
\newcommand{\hierarchical}{\mathbf{H}}

\newcommand{\bigo}{\mathcal{O}}
\newcommand{\reshape}{\texttt{reshape}}
\newcommand{\oom}{${}^{\dagger}$}
\newcommand{\wt}{${}^{\ddagger}$}
\newcommand{\ouralgorithm}{\texttt{HT-RISE}}
\newtheorem{problem}{Problem}
\newtheorem{claim}{Claim}

\newcommand{\mat}[1]{\begin{bmatrix}#1\end{bmatrix}}
\newcommand{\tenprod}[2]{{_{#1}{\times}_{#2}}}

\definecolor{rackhamgreen}{HTML}{75988d}
\definecolor{tappanred}{HTML}{9a3324}
\definecolor{wavefieldgreen}{HTML}{a5a508}
\definecolor{matthaeiviolet}{HTML}{577294}
\marginparwidth=\dimexpr \marginparwidth + 2cm\relax

\definecolor{grey}{rgb}{0.5, 0.5, 0.49}

\DeclareMathOperator*{\pad}{\oplus}

\usepackage{lastpage}
\jmlrheading{00}{0000}{1-\pageref{LastPage}}{0/00; Revised 0/00}{0/00}{00-0000}{Doruk Aksoy and Alex A. Gorodetsky. DISTRIBUTION A. Approved for public release; distribution unlimited. OPSEC\#9305}

\ShortHeadings{Incremental Hierarchical Tucker}{Aksoy and Gorodetsky}
\firstpageno{1}

\begin{document}

\title{Incremental Hierarchical Tucker Decomposition}

\author{
    \name Doruk Aksoy \email doruk@umich.edu\\
    \addr Department of Aerospace Engineering\\
    University of Michigan\\
    Ann Arbor, MI, USA
    \AND
    \name Alex A. Gorodetsky \email goroda@umich.edu\\
    \addr Department of Aerospace Engineering\\
    University of Michigan\\
    Ann Arbor, MI, USA
}

\editor{My editor}

    \maketitle

    \begin{abstract}
        We present two new algorithms for approximating and updating the hierarchical Tucker decomposition of tensor streams. The first algorithm, \textit{Batch Hierarchical Tucker - leaf to root} (\texttt{BHT-l2r}), proposes an alternative and more efficient way of approximating a batch of similar tensors in hierarchical Tucker format. The second algorithm, \textit{Hierarchical Tucker - Rapid Incremental Subspace Expansion} (\ouralgorithm), updates the batch hierarchical Tucker representation of an accumulated tensor as new batches of tensors become available.
        The \ouralgorithm~algorithm is suitable for the online setting and never requires full storage or reconstruction of all data while providing a solution to the incremental Tucker decomposition problem. We provide theoretical guarantees for both algorithms and demonstrate their effectiveness on physical and cyber-physical data.
        The proposed \texttt{BHT-l2r} algorithm and the batch hierarchical Tucker format offers up to $6.2\times$ compression and $3.7\times$ reduction in time over the hierarchical Tucker format. The proposed \ouralgorithm~algorithm also offers up to $3.1\times$ compression and $3.2\times$ reduction in time over a state of the art incremental tensor train decomposition algorithm.
    \end{abstract}

    \section{Introduction}\label{sec:introduction}

    Low-rank tensor decomposition formats~\citep{oseledets2011tensor,tucker1966some,grasedyck2010hierarchical,de2000multilinear} provide an efficient way of representing multidimensional data arising from a large number of applications that include videos~\citep{tian2023tensor,chen2024lowrank,panagakis2021tensor}, MRI scan images
    ~\citep{zhang2019tensor,lehmann2022multi,mai2023}, deep learning~\citep{kossaifi2023multi,luo2024trawl,yang2024loretta}, and solutions of scientific simulations~\citep{marks2024hall,pfaff2020learning}.
    There does not exist a single, univerally optimal, low-rank tensor format. Instead, some well-known low-rank tensor formats are the CP format~\citep{harshman1970foundations,carroll1970analysis}, the Tucker format~\citep{tucker1966some}, the Tensor Train (TT) format~\citep{oseledets2011tensor}, and the hierarchical Tucker (HT) format~\citep{oseledets2009breaking,grasedyck2010hierarchical}.
    These formats provide different forms of compression that exploit slightly different types of low-rank structure in data.
    \begin{figure}[htbp]
        \centering
        \includegraphics[width=0.8\textwidth]{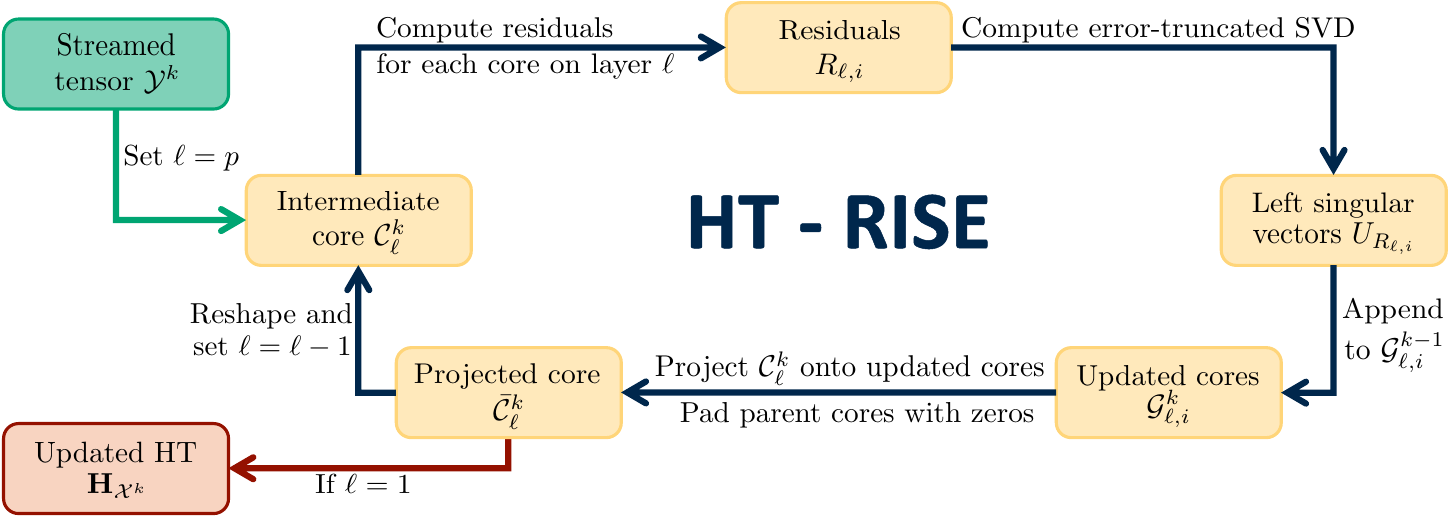}
        \caption{\small\emph{Flow of the proposed \ouralgorithm~algorithm. The algorithm updates the hierarchical Tucker representation of an accumulated tensor in batch hierarchical Tucker form as new batches of tensors become available. Green represents the input data and red represents the output data.}}
        \label{fig:algorithm_loop}
    \end{figure}

    In this paper, we consider the problem decomposing data in batches. This problem is motivated by the fact that data in many applications is not available at once, and/or the size of the data makes it infeasible or impossible to compute an approximation using a one-shot tensor decomposition algorithm due to computational issues (e.g., memory limits). For these scenarios, incremental algorithms can be effective by incrementally compress new data as it becomes available. They are particularly useful in the online setting where the data comes from a streaming process. While the literature has numerous incremental algorithms to compute the CP decomposition~\citep{zeng2021incremental,smith2018streaming}, the Tucker decomposition~\citep{de2023efficient,malik2018low}, and the TT decomposition~\citep{aksoy2024incremental,liu2018incremental,kressner2023streaming}, we are not aware of incremental approaches for the HT format.

    The main contribution of this paper is providing the first algorithm to incrementally construct a HT decomposition from streaming data.

    Furthermore, many applications that leverage compressed representations use it as a latent space for performing downstream tasks. In this setting, it is important to have both an encoding and decoding procedure between the full data and the latent space. While the HT format has some advantages over other formats, it does not immediately provide an efficient latent vector that represents an encoding of a batch of data. Straight-forward applications of existing HT decomposition algorithms~\citep{kressner2014algorithm} treat the batch dimension same as an additional dimension identical to existing ones --- so that when a batch of $d$-dimensional tensors $\mathcal{Y}\in\reals^{n_{1}\times\cdots\times n_{d}\times N}$ is compressed, the batch dimension is represented as an HT leaf node with orthonormal columns.

    We improve upon this approach by modifying the HT format to more efficiently represent the batch latent space by the HT root node.
    
    Formally, we describe the problem as follows. We consider a $d$-way tensor to be a multidimensional array $\mathcal{Y}\in\reals^{n_{1}\times\cdots\times n_{d}}$ with $d$ dimensions. Parallel to the definition in~\cite{aksoy2024incremental}, a \textit{tensor stream} (or alternatively, \textit{stream of tensors}), is  a sequence of $d+1$-way tensors $\mathcal{Y}^1,\mathcal{Y}^2$,\dots, where each element in the sequence $\mathcal{Y}^{k}\in \reals^{n_1\times\cdots\times n_d\times \newdim^{k}}$ is a $\newdim^{k}$ batch of $d$-dimensional tensors. A finite stream of tensors is called an \textit{accumulation} and can be viewed as a $(d+1)$-way tensor $\mathcal{X}^k\in\reals^{n_1\times\cdots\times n_d\times \cumuldim^{k}}$ by concatenating the tensors in the stream along the last dimension, with $\cumuldim^{k}=\sum_{i=1}^{k}\newdim^{i}$.

    Taken together, the above setting motivates the following problem as the first one that we solve.

        \begin{problem}\label{prob:batch_ht}
            (Approximating a batch of similar tensors in hierarchical Tucker format)
          Construct a scheme to compute an approximation $\hat{\mathcal{Y}}$ for a batch of $d$-dimensional tensors $\mathcal{Y}\in\reals^{n_{1}\times\cdots\times n_{d}\times N}$ in a (batch-modified) hierarchical Tucker format. Furthermore, the computed representation should approximate the original batch of tensors with error $\|\mathcal{Y}-\hat{\mathcal{Y}}\|_{F}\leq \varepsilon_{des}\|\mathcal{Y}\|_{F}$.

          The scheme should also create individual latent representations for all tensors in a batch, and provide a means to decode a latent representation into the full tensor.
        \end{problem}
     Next, we seek to perform the same task, but incrementally:
        \begin{problem}\label{prob:incremental_ht}
            (Updating an existing hierarchical Tucker representation as new batches of tensors become available)
            Construct a scheme to update the approximation $\hat{\mathcal{X}}^{k}$ of the accumulation tensor $\mathcal{X}^{k}$ after every increment $k$ in hierarchical Tucker format. The constructed scheme should maintain the guaranteed bounds on the error $\|\mathcal{X}^{k}-\hat{\mathcal{X}}^{k}\|_{F}$ for all $k$. Furthermore, this approximate accumulation should represent all previous tensor increments $\mathcal{Y}^{\ell}$ with error $\|\mathcal{Y}^{\ell}-\hat{\mathcal{Y}}^{\ell}\|_{F}\leq \varepsilon_{des}\|\mathcal{Y}^{\ell}\|_{F}$ for any $\ell\leq k$, where $\hat{\mathcal{Y}}^{\ell}$ can be extracted from $\hat{\mathcal{X}}^{k}$.
        \end{problem}
        Our contributions are algorithms to solve these two problems:
        \begin{enumerate}
        \item The \textit{Batch Hierarchical Tucker - leaf to root}, Algorithm~\ref{alg:batch_htucker},  computes an approximate  hierarchical Tucker representation of a batch of tensors $\mathcal{Y}$.
        \item The \textit{Hierarchical Tucker - Rapid Incremental Subspace Expansion}, Algorithm~\ref{alg:HIT}, updates an existing batch HT with new data.
        \end{enumerate}
        Moreover, \ouralgorithm~is suitable for the online setting and never requires full storage or reconstruction of all data while providing a solution to \Cref{prob:incremental_ht}.  These algorithms are both rank adaptive and have provable error bounds. A high level overview of the \ouralgorithm~algorithm is shown in \Cref{fig:algorithm_loop}. 

        These contributions are achieved by limiting the number of new basis vectors, representing directions orthogonal to the span of the existing ones, appended to the hierarchical Tucker cores of an existing accumulation tensor after each data increment. Moreover, our theoretical results are empirically justified on both scientific applications dealing with compression of numerical solutions of partial differential equations (PDEs) and on applications arising from image-based data such as video streams and multispectral satellite images.

        Our experiments show up to $1.5\times$ compression and $1.7\times$ reduction in time by the new Batch-HT format compared to the HT format on scientific data and up to $6.2\times$ compression and $3.7\times$ reduction in time for an image dataset. Furthermore, the proposed \ouralgorithm~algorithm provides up to $3.1\times$ compression and $4.5\times$ reduction in time at physical data, and $2\times$ compression and $5.3\times$ reduction in time at image-based data over a state of the art incremental TT decomposition algorithm. In addition to its computational efficiency, the proposed \ouralgorithm~algorithm discovers a latent representation that is more expressive and generalizable than the methods compared. This is evidenced by achieving the target approximation error on the test set with up to 15$\times$ less data than methods compared. Even for cases where the state-of-the-art TT decomposition algorithm fails to achieve the target error on the test set, \ouralgorithm~still achieves the target error.
        
        The  rest  of  this  paper  is  structured  as  follows. In \cref{sec:background},  we  present  the foundational  concepts  behind  both the  Tucker format and the HT format. In \cref{sec:methodology}, we present the \texttt{BHT-l2r} and \ouralgorithm~algorithms and prove their correctness. In \cref{sec:experiments}, we provide numerical experiments using our proposed approaches on physical and cyber-physical data. Finally, in \cref{sec:conclusion}, we conclude the paper and discuss possible future extensions.

    \section{Background}\label{sec:background}
        This section presents the relevant background on tensors, Tucker format, and hierarchical Tucker format.

        \subsection{Tensors}\label{sec:tensors}
            The mode-$i$ \textit{matricization} (or \textit{unfolding}) of a $d$-way tensor $\mathcal{A}\in\reals^{n_1\times\cdots\times n_d}$ reshapes the tensor into a matrix $A_{(i)}$ with size $n_{i}\times n_{1}\dots n_{i-1} n_{i+1}\dots n_d$.  In this unfolding, the $i$-th dimension is swapped to first index and fibers corresponding to the $i$-th mode are mapped into the rows of the matrix. The mode-$i$ unfolding operation will be represented as $\texttt{unfold}(\mathcal{A},i)$ in the rest of the text.

            The \textit{contraction} between two tensors $\mathcal{A}\in\reals^{n_{1}\times\cdots\times n_{d}}$ and $\mathcal{B}\in\reals^{n_d\times n_{d+1}\times\cdots\times n_{D}}$ along the $d$-th dimension of $\mathcal{A}$ and the first dimension of $\mathcal{B}$ is a binary operation represented as
            \begin{equation*}
            \resizebox{\textwidth}{!}{$
                \mathcal{C} = \mathcal{A}\ \tenprod{d}{1}\ \mathcal{B}, \quad \textrm{ where } \quad
                \mathcal{C}\left(i_{1},i_{2},\dots,i_{d-1},i_{d+1},\dots,i_{d}\right)=\sum_{j=1}^{n_d}\mathcal{A}\left(i_{1},\dots,i_{d-1},j\right)\mathcal{B}\left(j,i_{d+1},\dots,i_{D}\right)
                $},
            \end{equation*}
            and the output $\mathcal{C}\in\reals^{n_1\times\cdots\times n_{d-1}\times n_{d+1}\times\cdots\times n_{D}}$ becomes a $(D-2)$-way tensor. The subscripts on either side of the $\times$ sign indicate the contraction axes of the tensors on their respective sides.

            Let $\mathcal{A}\in\reals^{n_{1}\times\cdots\times n_{d}}$ be a $d$-dimensional tensor, and $B_{i}\in\reals^{m_{i}\times n_{i}}$ be matrices with $i=1,\dots,d$.
            Then \textit{multi-index contraction} between $\mathcal{A}$ and $B_{i}$ is defined as\footnote{Even though we have provided examples strictly contracting the second dimension of the matrices with the tensor, the multi-index contraction can happen between any appropriate axis of the matrix and the tensor in question.} 
            \begin{equation}\label{eq:multi_index_contraction}
                \mathcal{C}=\mathcal{A}\times\llbracket B_{1},\dots,B_{d}\rrbracket, \quad \textrm{where} \quad \mathcal{C}=\left(\cdots\left(\left(\left(\mathcal{A}~\tenprod{1}{2}~B_{1}\right)~\tenprod{2}{2}~B_{2}\right)~\tenprod{3}{2}~\cdots \right)~\tenprod{d}{2}~B_{d}\right),
            \end{equation}
            with $\mathcal{C}\in\reals^{m_{1}\times\cdots\times m_{d}}$ the resulting $d$-dimensional tensor. The idea in \eqref{eq:multi_index_contraction} can be generalized to a contraction with $d$ (or fewer) tensors $\mathcal{B}_i$ along their last dimensions in a similar manner. In case of contraction with $d$ three-way tensors $\mathcal{B}_i \in \reals^{m_{{1_i}} \times m_{2_i} \times n_{i}}$, \eqref{eq:multi_index_contraction} is denoted as
            \begin{equation} \label{eq:multi_index_contraction_for_three_way_tensors}
            \mathcal{C} =  \mathcal{A}\times\llbracket \mathcal{B}_{1},\dots, \mathcal{B}_{d}\rrbracket, \quad \textrm{where} \quad \mathcal{C}=\left(\cdots\left(\left(\left(\mathcal{A}~\tenprod{1}{3}~\mathcal{B}_{1}\right)~\tenprod{2}{3}~\mathcal{B}_{2}\right)~\tenprod{3}{3}~\cdots \right)~\tenprod{d}{3}~\mathcal{B}_{d}\right),
            \end{equation}
            with $\mathcal{C}\in\reals^{m_{1_1}\times m_{1_2} \times m_{2_1} \times m_{2_2} \times \cdots \times m_{d_1} \times m_{d_2}}$ being the resulting $2d$ dimensional tensor. Finally, it may be useful to only contract along a subset of dimensions. To this end, we can define a slightly modified version of \eqref{eq:multi_index_contraction_for_three_way_tensors} as
            \begin{equation} \label{eq:multi_index_contraction_partial}
            \resizebox{\textwidth}{!}{$
            \mathcal{C} =  \mathcal{A} \underset{\mathcal{I}}{\times} \llbracket \mathcal{B}_{i_1},\dots, \mathcal{B}_{i_m}\rrbracket, \quad \textrm{where} \quad \mathcal{C}=\left(\cdots\left(\left(\left(\mathcal{A}~\tenprod{i_{1}}{3}~\mathcal{B}_{i_{1}}\right)~\tenprod{i_{2}}{3}~\mathcal{B}_{i_{2}}\right)~\tenprod{i_{3}}{3}~\cdots \right)~\tenprod{i_{m}}{3}~\mathcal{B}_{i_{m}}\right)
            $}
            \end{equation}
            to represent a contraction along $m$ directions that are specified by the index set $\mathcal{I} =\{i_1,\ldots,i_{m}\}$, with $i_{j} \in {1,\ldots,d}$ denoting the dimension of $\mathcal{C}$ that is involved with the contraction.
            
             \textit{Concatenation} of two tensors along the $k$-th dimension is a binary operation. Concatenating two $d$-dimensional tensors $\mathcal{A}\in\reals^{n_{1}\times\cdots\times n_{k}\times \cdots\times n_{d}}$ and $\mathcal{B}\in\reals^{n_{1}\times\cdots\times n_{k-1} \times m_{k}\times n_{k+1}\times  \cdots\times n_{d}}$ along their $k$-th dimension is denoted by
            \begin{equation}\label{eq:concatenation}
                \mathcal{C}=\mathcal{A}\pad^{k}\mathcal{B},\quad\text{where} \quad\mathcal{C}\left( i_{1},\dots,i_{k},\dots,i_{d} \right)= \begin{cases}
                    \mathcal{A}\left( i_{1},\dots,i_{k},\dots,i_{k} \right)& \text{if } i_{k}\leq n_{k}\\
                \mathcal{B}\left( i_{1},\dots,(i_{k}-n_{k}),\dots,i_{k} \right) & \text{otherwise}
            \end{cases}.
            \end{equation}
    
            Concatenation of a tensor with $\mathbf{0}_{m\times n\times p}$, a tensor of zeros with size $m\times n \times p$ will be referred to as {\it padding}.

        \subsection{Tucker format}\label{sec:tucker_format}

            When tensors have additional structure, it may be possible to effectively represent them in a more compact format. In this section we describe the Tucker format, which will form the basis of the hierarchical Tucker format we consider later. L.R. Tucker proposed a way to decompose a third order tensor into three factor matrices and a core tensor, coining the name \textit{Tucker format} for this family of tensor representation format~\citep  {tucker1963implications,tucker1964extension,tucker1966some}. In Tucker format, a $d$ dimensional tensor $\mathcal{Y}\in\reals^{n_{1}\times\cdots\times n_{d}}$ is represented as a contraction between a $d$-dimensional core tensor $\mathcal{C}\in\reals^{r_{1}\times\cdots\times r_{d}}$ and $d$ factor matrices $U_{i}\in\reals^{n_{i}\times r_{i}}$,
            \begin{equation}\label{eq:tuckerformat}
                \mathcal{Y}=\mathcal{C}\times\llbracket U_{1},\dots,U_{d} \rrbracket.
            \end{equation}
            The Tucker representation reduces the storage from $\mathcal{O}\left(n^{d} \right)$ to $\mathcal{O} \left(dnr+r^{d}\right)$, where $r_{i}$ are potentially much smaller than $n_{i}$. One of the well established ways to compute Tucker representation of a tensor is through the \texttt{HOSVD} (\texttt{H}igher \texttt{O}rder \texttt{S}ingular \texttt{V}alue \texttt{D}ecomposition) algorithm~\citep{de2000multilinear,tucker1966some}. In literature, there are numerous efficient methods of computing the Tucker representation of a tensor including direct~\citep{vannieuwenhoven2012new,kressner2017recompression}, randomized~\citep{che2019randomized,kressner2017recompression,tsourakakis2010mach,zhang2018randomized,zhou2014decomposition,minster2022parallel,sun2020low}, and iterative~\citep{de2000best,kroonenberg1980principal,wen2015robust,chachlakis2020l1,kressner2017recompression,tsourakakis2010mach,elden2009newton} algorithms.

            \subsection{Hierarchical Tucker Format} \label{sec:hierarchical_tucker}
            The Tucker format scales exponentially with the number of dimensions $d$, which prohibits its usage for high-dimensional tensors. One way to alleviate this curse of dimensionality is to exploit structure of the Tucker core. Such structure can be exploited via a recursive Tucker decomposition. In \citet{oseledets2009breaking}, the resulting format is named as \textit{tree Tucker}, and in \citet{grasedyck2010hierarchical} it is mentioned as \textit{hierarchical Tucker} (HT). The HT format introduces a hierarchy among dimensions and splits them accordingly. In the simplest, and most commonly used, binary splitting case, the Tucker core $\mathcal{C}$ in Equation~\eqref{eq:tuckerformat} is reshaped into a tensor with $d/2$ dimensions and a decomposition of this core is created, recursively. The HT format has storage complexity $\bigo(dnR+dR^{3})$, cubic in the maximum HT rank $R$.

            The HT representation of a tensor is a couple $\hierarchical = (\tree, \bm{\core})$, where $\tree$ is a {\it dimension tree} and $\bm{\core}$ is a set of {\it HT cores}. The dimension tree is a connected graph of $2d-1$ nodes, $\tree = \{ \node_{\ell,i_{\ell}} \}$, that specifies a contraction ordering of the $2d-1$ HT cores, $\bm{\core} = \{\core_{\ell,i_{\ell}}\}$. The tree has depth $p$ and $|\tree_{\ell}|$ nodes in layer $\ell$. A layer is a set of nodes equidistant from the root. The indices of each node in the tree $\node_{\ell,i_{\ell}} \in \tree$ run from $\ell = 0,\ldots,p$ and $i_{\ell} = 1,\ldots, |\tree_{\ell}|,$ and each node describes how the corresponding core $\core_{\ell, i_{\ell}}$ contracts with its neighbors. Contracting all cores according to the dimension tree yields a reconstruction of the represented tensor. The root, interior, and leaves of the tree structure correspond to so-called \textit{root}, \textit{transfer}, and \textit{leaf} cores.  \Cref{fig:dimension_trees} illustrates various configurations of dimension trees for a five-dimensional tensor.

            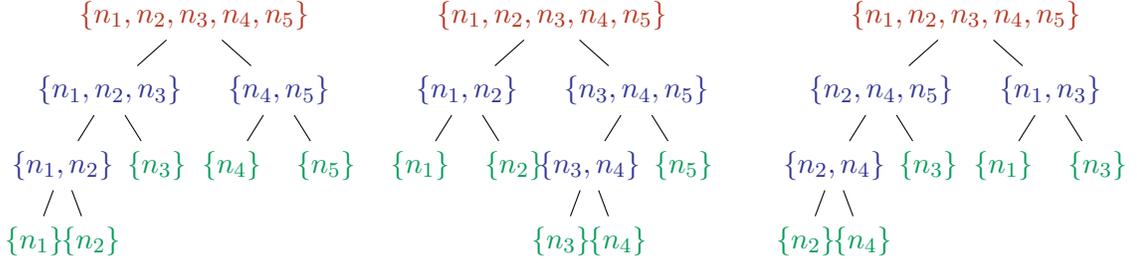
\begin{figure}[h!]
                \centering
                \begin{subfigure}[h]{0.3\textwidth}
                    \begin{tikzpicture}[level distance=1cm, 
    level 1/.style={sibling distance=2.25cm},
    level 2/.style={sibling distance=1.25cm},
    level 3/.style={sibling distance=0.75cm}]
  
    \node {\textcolor{BrickRed}{\{$n_{1},n_{2},n_{3},n_{4},n_{5}$\}}}
        child {node {\textcolor{Blue}{\{$n_{1},n_{2},n_{3}$\}}}
            child {node {\textcolor{Blue}{\{$n_{1},n_{2}$\}}}
                child {node {\textcolor{ForestGreen}{\{$n_{1}$\}}}}
                child {node {\textcolor{ForestGreen}{\{$n_{2}$\}}}}
                }
            child {node {\textcolor{ForestGreen}{\{$n_{3}$\}}}}
        }
        child {node {\textcolor{Blue}{\{$n_{4},n_{5}$\}}}
            child {node {\textcolor{ForestGreen}{\{$n_{4}$\}}}}
            child {node {\textcolor{ForestGreen}{\{$n_{5}$\}}}}
        };
  
\end{tikzpicture}
  
                \end{subfigure}
                \hspace{3mm}
                \begin{subfigure}[h]{0.3\textwidth}
                    \begin{tikzpicture}[level distance=1cm, 
    level 1/.style={sibling distance=2.25cm},
    level 2/.style={sibling distance=1.25cm},
    level 3/.style={sibling distance=0.75cm}]
  
    \node {\textcolor{BrickRed}{\{$n_{1},n_{2},n_{3},n_{4},n_{5}$\}}}
        child {node {\textcolor{Blue}{\{$n_{1},n_{2}$\}}}
            child {node {\textcolor{ForestGreen}{\{$n_{1}$\}}}
            }
            child {node {\textcolor{ForestGreen}{\{$n_{2}$\}}}}
            }
            child {node {\textcolor{Blue}{\{$n_{3},n_{4},n_{5}$\}}}
            child {node {\textcolor{Blue}{\{$n_{3},n_{4}$\}}}
                child {node {\textcolor{ForestGreen}{\{$n_{3}$\}}}}
                child {node {\textcolor{ForestGreen}{\{$n_{4}$\}}}}
            }
            child {node {\textcolor{ForestGreen}{\{$n_{5}$\}}}}
        };
  
\end{tikzpicture}
  
                \end{subfigure}
                \hspace{3mm}
                \begin{subfigure}[h]{0.3\textwidth}
                    \begin{tikzpicture}[level distance=1cm, 
    level 1/.style={sibling distance=2.25cm},
    level 2/.style={sibling distance=1.25cm},
    level 3/.style={sibling distance=0.75cm}]
  
    \node {\textcolor{BrickRed}{\{$n_{1},n_{2},n_{3},n_{4},n_{5}$\}}}
        child {node {\textcolor{Blue}{\{$n_{2},n_{4},n_{5}$\}}}
            child {node {\textcolor{Blue}{\{$n_{2},n_{4}$\}}}
                child {node {\textcolor{ForestGreen}{\{$n_{2}$\}}}}
                child {node {\textcolor{ForestGreen}{\{$n_{4}$\}}}}
                }
            child {node {\textcolor{ForestGreen}{\{$n_{3}$\}}}}
        }
        child {node {\textcolor{Blue}{\{$n_{1},n_{3}$\}}}
            child {node {\textcolor{ForestGreen}{\{$n_{1}$\}}}}
            child {node {\textcolor{ForestGreen}{\{$n_{3}$\}}}}
        };
  
\end{tikzpicture}
  
                \end{subfigure}
                \caption{\small\emph{Three possible configurations of the dimensions for a five dimensional tensor with shape $n_{1}\times n_{2}\times n_{3}\times n_{4}\times n_{5}$. Each dimension tree describes a different interaction between dimensions and therefore yields different compression performance. For a more detailed study on the effect of the axis reordering, please refer to \Cref{app:axis_ordering}. Root, transfer, and leaf nodes are colored in red, blue, and green, respectively.}}
                \label{fig:dimension_trees}
            \end{figure}
            For balanced binary trees, the depth is determined by $p=\lceil\log_{2}\left( d \right)\rceil$, with $\lceil\cdot\rceil$ defined as $\min\left\{ i\in \mathbb{Z} | i \geq \cdot \right\}$. In this study, we focus exclusively on balanced dimension trees, which means that the leaf nodes appear on both the last level ($p$) and the penultimate level ($p-1$). The order in which dimensions are arranged significantly impacts the interaction among them, and consequently influences the result of the decomposition, as empirically demonstrated in \Cref{app:axis_ordering}. \Cref{fig:dimension_trees} also shows that there are always a total of $d$ leaf cores, however these leaf cores are not always in the same layer. It will be useful to keep track of the dimension to which a leaf core refers. To this end, we will denote $d_{\ell, i_{\ell}} \in \{1,\ldots,d\}$ to be the dimension corresponding to leaf core $\core_{\ell,i_{\ell}}$. The subscript is only valid if the core is a leaf node.

            Formally,  a node is a couple $\node_{\ell, i_{\ell}} = (\successors_{\ell, i_{\ell}}, \parents_{\ell,i_{\ell}})$  that specifies the indices of the successors and parents with which $\core_{\ell,i_{\ell}}$ contracts. For each type of core we have:
            \begin{itemize}[wide, leftmargin=*]
            \item[Leaf core]: $\core_{\ell, i_{\ell}} \in \reals^{n_{i} \times \rank_{\ell, i_{\ell}}};$ $\successors_{\ell,i_{\ell}}=\emptyset$, $\parents_{\ell,i_{\ell}}= (\ell-1,\lceil i_{\ell}/2 \rceil)$
            \item[Transfer core]: $\core_{\ell, i_{\ell}} \in \reals^{\rank_{\ell+1,\alpha_{\ell+1}} \times  \rank_{\ell+1, \alpha_{\ell+1}+1} \times \rank_{\ell, i_{\ell}}}$,
            $\successors_{\ell,i_{\ell}} = \{(\ell+1,{\alpha}_{\ell+1}),(\ell+1,(\alpha_{\ell+1}+1))\}$, $\parents_{\ell,i_{\ell}} = (\ell-1,\lceil {i_{\ell}}/{2}\rceil)$
            \item[Root core]: $\core_{0,1}\in\reals^{\rank_{1,1}\times \rank_{1,2}}$, 
            $\successors_{0,1} = \{(1,1),(1,2)\}$  $\parents_{0,1}=\emptyset$,
            \end{itemize}
            where $\alpha_{\ell+1}= 2(i_{\ell}-\beta_{\ell,i_{\ell}})-1$ indexes the appropriate successor node, and $\beta_{\ell,i_{\ell}}$ is the number of leaf nodes in layer $\ell$ up to $i_{\ell}$ node of that layer.
            For a node $\node_{\ell-1,t}$, the order of successors $\successors_{\ell-1,t}=\{\node_{\ell,i},\node_{\ell,j}\}$ is determined by comparing their position within the $\ell$-th layer, i.e., by ranking $i$ and $j$.
            
            The collection of ranks, called the HT ranks, $\ranks = [\rank_{1,1}, \rank_{1,2}, \rank_{2,1},\ldots, \rank_{p, |\tree_p|}]$ correspond to the sizes of the edges along which the leaves, transfer, and root core nodes contract. Contraction with a parent occurs along the last dimension of a node, while contraction with successors occurs with the first two.

            To obtain the full tensor we can therefore perform contractions of each core with parents and successors according to the map provided by $\tree$. A particularly illuminating contraction ordering is root-to-leaves. Root-to-leaves contraction begins contraction with the root node and moves down the tree, layer by layer. Overall, this contraction represents the reconstruction of a tensor $\mathcal{Y}$ from its HT format as
            \begin{equation}\label{eq:htucker_reconstruction}
            \mathcal{Y} = \left(\cdots \left(\left(\mathcal{G}_{0,1} \times \llbracket \mathcal{G}_{1,1}, \mathcal{G}_{1,1} \rrbracket \right) \times \llbracket \mathcal{G}_{2,1},\mathcal{G}_{2,2}, \mathcal{G}_{2,3},\mathcal{G}_{2,4} \rrbracket \right) \times \cdots  \right)\times \llbracket \mathcal{G}_{p,1},\ldots,\mathcal{G}_{p, |\tree_p|} \rrbracket.
            \end{equation}
            This representation shows why the HT format is interpreted as a {\it hierarchical decomposition of the Tucker core tensor}. Indeed the last contraction is over the leaves, which can be interpreted as the Tucker leaves.
            
            \subsubsection{Leaves-to-root error truncated hierarchical Tucker decomposition}\label{sec:error_truncated_htucker}

                In this section we describe an algorithm to represent a given tensor $\mathcal{Y}\in\reals^{n_{1}\times\cdots\times n_{d}}$ in low-rank format. In practice, $\mathcal{Y}$ is rarely exactly low rank, and hence decomposition algorithms generate a representation with a prescribed accuracy.
                Two strategies to compute a hierarchical Tucker representation include: 1) leaves-to-root decomposition~\cite[Alg.~2]{grasedyck2010hierarchical}, and 2) root-to-leaves decomposition~\cite[Alg.~5]{kressner2014algorithm}. Our proposed incremental method updates the hierarchical Tucker representation from leaves to root, and we limit discussion to this case.
                
                Leaves-to-root decomposition essentially applies the HOSVD (or a variant) to each layer of the tree $\tree$, beginning at the $p$-th layer. First, a truncated singular value decomposition is performed for $|\tree_p|$ unfoldings --- each unfolding corresponding to the dimensions of the $i$-th node in the layer. For each, an error truncated \texttt{SVD} with a Frobenius norm bound $\varepsilon_{nw}$ on the residual $E_{p,i}$ is computed,
                \begin{equation}\label{eq:ht_leaf_svd}
                Y_{(d_{p,i})} = U_{{p},{i}} \Sigma_{p,{i}} V_{p,i}^{T}+E_{p,{i}} \quad \| E_{p,{i}} \|_{F} \leq \varepsilon_{nw}, \quad i=1,\ldots, |\tree_{p}| 
                \end{equation}
                where $U_{p,{i}}\in\reals^{n_{i}\times r_{p,i}}$ is the matrix of left singular vectors with the truncation rank $r_{p,i}$, $\Sigma_{p,i}\in\reals^{r_{p,i}\times r_{p,i}}$ is the matrix of singular values, and $V_{p,{i}}^{T}\in\reals^{r_{p,i}\times m}$ is the matrix of right singular vectors. All nodes at the $p$-th layer correspond to leaf cores, yielding  $\core_{p,i}=U_{p,i}$.
                The terms $\Sigma_{p,i}$ and $V_{p,i}$ are discarded in the leaves-to-root decomposition\footnote{If a variant of the \texttt{ST-HOSVD} algorithm~\citep{vannieuwenhoven2012new} is used to perform the computations, then $\Sigma$ and $V$ might be used in the subsequent steps within the layer.}. Once \eqref{eq:ht_leaf_svd} is repeated for all leaves on layer $p$, the orthogonal $U_{p,{i}}$ matrices are contracted with $\mathcal{Y}$ along the dimensions refered to by the leaves
                \begin{equation}\label{eq:ht_leaf_contraction}
                    \bar{\mathcal{C}}_{p}=\mathcal{Y} \underset{\mathcal{I}}{\times} \llbracket U_{p,1}^{T},\dots,U_{p,|\tree_{p}|}^{T} \rrbracket, \textrm{ where } \quad \mathcal{I} = \{d_{p,i} ; i =1, \ldots, |\tree_{p}| \}
                \end{equation}
                to obtain an intermediate core $\bar{\mathcal{C}}_{p}$ that will be decomposed recursively in subsequent steps. Prior to such a decomposition, this intermediate core must be reshaped. In a perfect binary tree, there are $d$ leaves (one for each dimension), and reshaping transforms $\bar{\mathcal{C}}_p$ into a tensor $\mathcal{C}_{p-1}$ of dimensions $\rank_{p,1}\rank_{p,2} \times \rank_{p,3} \rank_{p,4} \times \cdots \times \rank_{p,d-1}\rank_{p,d}$. More generally, the dimension tree defines the groupings by defining the successors of each node. To this end, we can define the reshaped dimensions using an \textit{index set} according to
                \begin{equation}\label{eq:ht_index_set_construction} 
                    \resizebox{\textwidth}{!}{$
                  \indexset_{\tree_{p-1}}=
                  \bigcup_{j=1}^{|\tree_{p-1}|}
                    \begin{cases}
                        n_{d_{p-1,j}} & \text{if $\node_{p-1,j}$ is a leaf}\\
                        \displaystyle{\prod_{(p,m) \in \successors_{p-1,j}}} \rank_{p,m} & \text{else.}
                    \end{cases},
                    \quad \textrm{ so that} \quad
                    \mathcal{C}_{p-1}=\reshape\left(\bar{\mathcal{C}}_{p},\ \indexset_{\tree_{p-1}}\right).
                    $}
                \end{equation}

                Once we obtain $\mathcal{C}_{p-1}$, the error truncated \texttt{SVD} analogous to \eqref{eq:ht_leaf_svd} and reshaping steps are repeatedly used to decompose each level to obtain $\mathcal{C}_{\ell}$ for $\ell=p-2,\ldots,1$. For these subsequent levels, the left-singular vectors $U_{\ell,i}$ of the truncated SVD of the unfoldings $C_{\ell,(i)}$ become the transfer cores. Transfer cores in the HT format are three dimensional and thus the left singular vectors are reshaped according to  
                \begin{equation}\label{eq:ht_tucker_core_reshape} 
                    \core_{\ell,i}=\reshape\left( U_{\ell,i}, \indexset_{\node_{\ell,i}} \right), \quad \textrm{ where }                     \indexset_{\node_{\ell,i}}=\left\{\bigcup_{(\ell+1,j) \in \successors_{\ell,i}}\rank_{\ell+1,j}\right\}\cup {\rank_{\ell, i}}. \quad i = 1,\ldots, |\tree_{\ell}|,
                \end{equation}
                whereas any leaf cores are reshaped analogously to the approach at the $p$-th layer.

                The decomposition is completed after a final error truncated \texttt{SVD} of $\mathcal{C}_{1}\in\reals^{\rank_{1,1}\times \rank_{1,2}}$ 
                \begin{equation}\label{eq:ht_last_svd}
                    \mathcal{C}_{1} = U_{1}\Sigma_{1}V_{1}^{T}+E_{1},                    
                \end{equation}
                with $U_{1}\in\reals^{\rank_{1,1}\times \rank_{0}}$, $\Sigma_{1}\in\reals^{\rank_{0}\times \rank_{0}}$, $V_{1}^{T}\in\reals^{\rank_{0}\times \rank_{1,2}}$ and $E_{1}\in\reals^{\rank_{1,1}\times \rank_{1,2}}$. Unlike the previous levels, all singular vectors and values are kept: $U_{1}$ is reshaped into $\core_{1,1}$, $V_{1}^{T}$ is reshaped into $\core_{1,2}$, and the singular values $\Sigma_{1}$ become the root core $\core_{0,1}.$ This procedure comes with the following guarantee
                \begin{theorem}[Adapted from~\cite{kressner2014algorithm} Lemma B.2]\label{thm:ht_nodewise_error}
                    For a \\$d$-dimensional tensor $\mathcal{Y}\in\reals^{n_{1}\times\cdots\times n_{d}}$, the best HT approximation $\tilde{\mathcal{Y}}$ with an absolute approximation error $\|\mathcal{Y}-\tilde{\mathcal{Y}}\|_{F} \leq \varepsilon_{abs}$ can be obtained by prescribing a node-wise truncation error $\varepsilon_{nw} = \frac{\varepsilon_{abs}}{\sqrt{2d-3}}$ such that 
                \begin{equation}\label{eq:ht_truncation_error}
                    \|E_{i}\|_{F} \leq \varepsilon_{nw}
                \end{equation} 
                for all truncated \texttt{SVD} computations.
                \end{theorem}

                In the subsequent section we describe modifications to the HT format and corresponding approximation algorithms that are needed to adapt to the context of streaming batches of tensors.

    \section{Methodology}\label{sec:methodology}
            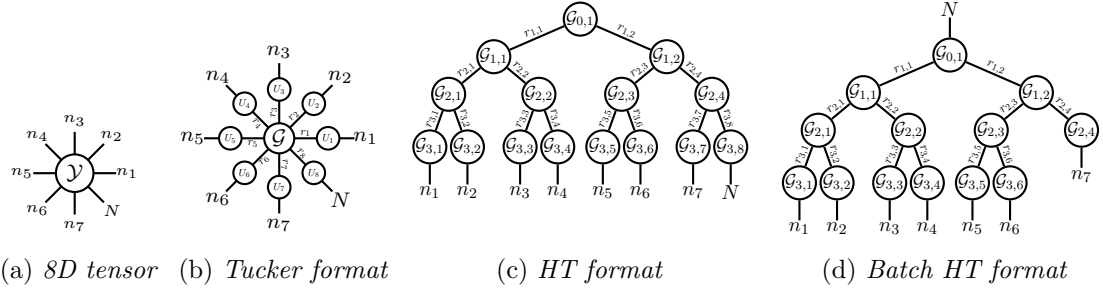
\begin{figure}[h!]
                \centering
                \begin{subfigure}{0.14\textwidth}\centering
                    \begin{tikzpicture}
    \def\k{8} 
    \def\radius{0.25cm} 
    \def\spikelength{0.3cm} 

    \draw[thick] (0,0) circle (\radius);
    \node at (0,0) {\scriptsize $\mathcal{Y}$};

    \foreach \n in {1,...,\k}{
        \pgfmathsetmacro\angle{360/\k * (\n - 1)}
        \draw[thick] (\angle:\radius) -- (\angle:\radius+\spikelength);
        \ifnum\n=\k{
            \node at (\angle:\spikelength + \radius + 0.15cm) {\tiny$N$};
        }
        \else{
            \node at (\angle:\spikelength + \radius + 0.15cm) {\tiny$n_{\n}$};
        }
        \fi
    }
\end{tikzpicture}
                    \caption{\small\emph{8D tensor}}
                \end{subfigure}
                \begin{subfigure}{0.2\textwidth}\centering






\begin{tikzpicture}
    \def\k{8} 
    \def\radius{0.2cm} 
    \def\spikelength{0.5cm} 
    \def\secradius{0.15cm} 
    \def\secspikelength{0.2cm} 

    \draw[thick] (0,0) circle (\radius);
    \node at (0,0) (core) {\scriptsize$\mathcal{G}$};

    \foreach \n in {1,...,\k}{
        \pgfmathsetmacro\angle{360/\k * (\n - 1)}

        \coordinate (sec_center) at (\angle:\spikelength + 1*\secradius);
        \node at (sec_center) (mat_\n) {\scalebox{0.35}{$U_{\n}$}};

        \draw[thick] (sec_center) circle (\secradius);
        
        \draw[thick] (\angle:\spikelength + 2*\secradius) -- (\angle:\spikelength + 2*\secradius + \secspikelength);
    
        \ifnum\n=\k
            \node at (\angle:\spikelength + 2*\secradius + \secspikelength + 0.15cm) {\scriptsize $N$};
        \else
            \node at (\angle:\spikelength + 2*\secradius + \secspikelength + 0.15cm) {\scriptsize $n_{\n}$};
        \fi
    }
    \path[-,thick]
        (0:\radius) edge node[yshift=2pt, sloped, minimum size=0.2cm] {\scalebox{0.4}{$r_{1}$}} (0:\spikelength)
        (45:\radius) edge node[yshift=2pt, sloped, minimum size=0.2cm] {\scalebox{0.4}{$r_{2}$}} (45:\spikelength)
        (90:\radius) edge node[yshift=2pt, sloped, minimum size=0.2cm] {\scalebox{0.4}{$r_{3}$}} (90:\spikelength)
        (135:\radius) edge node[yshift=-2pt, sloped, minimum size=0.2cm] {\scalebox{0.4}{$r_{4}$}} (135:\spikelength)
        (180:\radius) edge node[yshift=-2pt, sloped, minimum size=0.2cm] {\scalebox{0.4}{$r_{5}$}} (180:\spikelength)
        (225:\radius) edge node[yshift=-2pt, sloped, minimum size=0.2cm] {\scalebox{0.4}{$r_{6}$}} (225:\spikelength)
        (270:\radius) edge node[yshift=2pt, sloped, minimum size=0.2cm] {\scalebox{0.4}{$r_{7}$}} (270:\spikelength)
        (315:\radius) edge node[yshift=2pt, sloped, minimum size=0.2cm] {\scalebox{0.4}{$r_{8}$}} (315:\spikelength)
    ;
\end{tikzpicture}
                    \caption{\small\emph{Tucker format}}
                \end{subfigure}
                \begin{subfigure}{0.3\textwidth}\centering
                    \begin{tikzpicture}[
    level 1/.style={sibling distance=2.3cm, level distance=0.5cm},
    level 2/.style={sibling distance=1.2cm, level distance=0.5cm},
    level 3/.style={sibling distance=0.5cm, level distance=0.7cm},
    every node/.style={circle, draw, minimum size=0.2cm, inner sep=0pt},
    edge from parent/.style={draw= none, thick, sloped, anchor=south},
    edge from parent node/.style={midway, sloped},
    every path/.style={thick},
    ]
    \def\core{\mathcal{G}}
    \def\spikelength{0.5cm}

    \node (root) {\scalebox{0.6}{$\core_{0,1}$}}
        child {node (g_1_1) {\scalebox{0.6}{$\core_{1,1}$}}
            child {node (g_2_1) {\scalebox{0.6}{$\core_{2,1}$}}
                child{
                    node (g_3_1) {\scalebox{0.6}{$\core_{3,1}$}}
                }
                child{
                    node (g_3_2) {\scalebox{0.6}{$\core_{3,2}$}}
                }
            }
            child {node (g_2_2) {\scalebox{0.6}{$\core_{2,2}$}}
                child{
                    node (g_3_3) {\scalebox{0.6}{$\core_{3,3}$}}
                }
                child{
                    node (g_3_4) {\scalebox{0.6}{$\core_{3,4}$}}
                }
            }
        }
        child {node (g_1_2) {\scalebox{0.6}{$\core_{1,2}$}}
            child {node (g_2_3) {\scalebox{0.6}{$\core_{2,3}$}}
                child{
                    node (g_3_5) {\scalebox{0.6}{$\core_{3,5}$}}
                }
                child{
                    node (g_3_6) {\scalebox{0.6}{$\core_{3,6}$}}
                }
            }
            child {node (g_2_4) {\scalebox{0.6}{$\core_{2,4}$}}
                child{
                    node (g_3_7) {\scalebox{0.6}{$\core_{3,7}$}}
                }
                child{
                    node (g_3_8) {\scalebox{0.6}{$\core_{3,8}$}}
                }
            }
        };
    \path[-] 
        (root) edge node[yshift=-2pt,draw=none,sloped,anchor=south,minimum size=0.2cm] {\scalebox{0.45}{$r_{1,1}$}} (g_1_1)
        (root) edge node[yshift=-2pt,draw=none,sloped,anchor=south,minimum size=0.2cm] {\scalebox{0.45}{$r_{1,2}$}} (g_1_2)

        (g_1_1) edge node[yshift=-2pt,draw=none,sloped,anchor=south,minimum size=0.2cm] {\scalebox{0.45}{$r_{2,1}$}} (g_2_1)
        (g_1_1) edge node[yshift=-2pt,draw=none,sloped,anchor=south,minimum size=0.2cm] {\scalebox{0.45}{$r_{2,2}$}} (g_2_2)

        (g_1_2) edge node[yshift=-2pt,draw=none,sloped,anchor=south,minimum size=0.2cm] {\scalebox{0.45}{$r_{2,3}$}} (g_2_3)
        (g_1_2) edge node[yshift=-2pt,draw=none,sloped,anchor=south,minimum size=0.2cm] {\scalebox{0.45}{$r_{2,4}$}} (g_2_4)

        (g_2_1) edge node[yshift=-2pt,draw=none,sloped,anchor=south,minimum size=0.2cm] {\scalebox{0.45}{$r_{3,1}$}} (g_3_1)
        (g_2_1) edge node[yshift=-2pt,draw=none,sloped,anchor=south,minimum size=0.2cm] {\scalebox{0.45}{$r_{3,2}$}} (g_3_2)

        (g_2_2) edge node[yshift=-2pt,draw=none,sloped,anchor=south,minimum size=0.2cm] {\scalebox{0.45}{$r_{3,3}$}} (g_3_3)
        (g_2_2) edge node[yshift=-2pt,draw=none,sloped,anchor=south,minimum size=0.2cm] {\scalebox{0.45}{$r_{3,4}$}} (g_3_4)

        (g_2_3) edge node[yshift=-2pt,draw=none,sloped,anchor=south,minimum size=0.2cm] {\scalebox{0.45}{$r_{3,5}$}} (g_3_5)
        (g_2_3) edge node[yshift=-2pt,draw=none,sloped,anchor=south,minimum size=0.2cm] {\scalebox{0.45}{$r_{3,6}$}} (g_3_6)

        (g_2_4) edge node[yshift=-2pt,draw=none,sloped,anchor=south,minimum size=0.2cm] {\scalebox{0.45}{$r_{3,7}$}} (g_3_7)
        (g_2_4) edge node[yshift=-2pt,draw=none,sloped,anchor=south,minimum size=0.2cm] {\scalebox{0.45}{$r_{3,8}$}} (g_3_8)
        ;
  \newcounter{k}
  \setcounter{k}{1}

  \foreach \pos in {(g_3_1),(g_3_2),(g_3_3),(g_3_4),(g_3_5),(g_3_6),(g_3_7)} {
      \draw[thick] \pos -- ++(270:\spikelength) node[pos=1.4, draw=none] {\scalebox{0.7}{$n_{\thek}$}};
      \stepcounter{k};
  }
  \draw[thick] (g_3_8) -- ++(270:\spikelength) node[pos=1.4, draw=none] {\scalebox{0.7}{$N$}};
  
\end{tikzpicture}
                    \caption{\small\emph{HT format}}
                    \label{fig:htucker_as_node}
                \end{subfigure}
                \begin{subfigure}{0.32\textwidth}\centering
                    \begin{tikzpicture}[
    level 1/.style={sibling distance=2.3cm, level distance=0.5cm},
    level 2/.style={sibling distance=1.2cm, level distance=0.5cm},
    level 3/.style={sibling distance=0.5cm, level distance=0.7cm},
    every node/.style={circle, draw, minimum size=0.2cm, inner sep=0pt},
    edge from parent/.style={draw= none, thick, sloped, anchor=south},
    edge from parent node/.style={midway, sloped},
    every path/.style={thick},
    ]
    \def\core{\mathcal{G}}
    \def\spikelength{0.5cm}

    \node (root) {\scalebox{0.6}{$\core_{0,1}$}}
        child {node (g_1_1) {\scalebox{0.6}{$\core_{1,1}$}}
            child {node (g_2_1) {\scalebox{0.6}{$\core_{2,1}$}}
                child{
                    node (g_3_1) {\scalebox{0.6}{$\core_{3,1}$}}
                }
                child{
                    node (g_3_2) {\scalebox{0.6}{$\core_{3,2}$}}
                }
            }
            child {node (g_2_2) {\scalebox{0.6}{$\core_{2,2}$}}
                child{
                    node (g_3_3) {\scalebox{0.6}{$\core_{3,3}$}}
                }
                child{
                    node (g_3_4) {\scalebox{0.6}{$\core_{3,4}$}}
                }
            }
        }
        child {node (g_1_2) {\scalebox{0.6}{$\core_{1,2}$}}
            child {node (g_2_3) {\scalebox{0.6}{$\core_{2,3}$}}
                child{
                    node (g_3_5) {\scalebox{0.6}{$\core_{3,5}$}}
                }
                child{
                    node (g_3_6) {\scalebox{0.6}{$\core_{3,6}$}}
                }
            }
            child {node (g_2_4) {\scalebox{0.6}{$\core_{2,4}$}}
            }
        };
    \path[-] 
        (root) edge node[yshift=-2pt,draw=none,sloped,anchor=south,minimum size=0.2cm] {\scalebox{0.45}{$r_{1,1}$}} (g_1_1)
        (root) edge node[yshift=-2pt,draw=none,sloped,anchor=south,minimum size=0.2cm] {\scalebox{0.45}{$r_{1,2}$}} (g_1_2)

        (g_1_1) edge node[yshift=-2pt,draw=none,sloped,anchor=south,minimum size=0.2cm] {\scalebox{0.45}{$r_{2,1}$}} (g_2_1)
        (g_1_1) edge node[yshift=-2pt,draw=none,sloped,anchor=south,minimum size=0.2cm] {\scalebox{0.45}{$r_{2,2}$}} (g_2_2)

        (g_1_2) edge node[yshift=-2pt,draw=none,sloped,anchor=south,minimum size=0.2cm] {\scalebox{0.45}{$r_{2,3}$}} (g_2_3)
        (g_1_2) edge node[yshift=-2pt,draw=none,sloped,anchor=south,minimum size=0.2cm] {\scalebox{0.45}{$r_{2,4}$}} (g_2_4)

        (g_2_1) edge node[yshift=-2pt,draw=none,sloped,anchor=south,minimum size=0.2cm] {\scalebox{0.45}{$r_{3,1}$}} (g_3_1)
        (g_2_1) edge node[yshift=-2pt,draw=none,sloped,anchor=south,minimum size=0.2cm] {\scalebox{0.45}{$r_{3,2}$}} (g_3_2)

        (g_2_2) edge node[yshift=-2pt,draw=none,sloped,anchor=south,minimum size=0.2cm] {\scalebox{0.45}{$r_{3,3}$}} (g_3_3)
        (g_2_2) edge node[yshift=-2pt,draw=none,sloped,anchor=south,minimum size=0.2cm] {\scalebox{0.45}{$r_{3,4}$}} (g_3_4)

        (g_2_3) edge node[yshift=-2pt,draw=none,sloped,anchor=south,minimum size=0.2cm] {\scalebox{0.45}{$r_{3,5}$}} (g_3_5)
        (g_2_3) edge node[yshift=-2pt,draw=none,sloped,anchor=south,minimum size=0.2cm] {\scalebox{0.45}{$r_{3,6}$}} (g_3_6)

        ;
  \newcounter{m}
  \setcounter{m}{1}

  \foreach \pos in {(g_3_1),(g_3_2),(g_3_3),(g_3_4),(g_3_5),(g_3_6),(g_2_4)} {
      \draw[thick] \pos -- ++(270:\spikelength) node[pos=1.4, draw=none] {\scalebox{0.7}{$n_{\them}$}};
      \stepcounter{m};
  }
  \draw[thick] (root) -- ++(90:\spikelength) node[pos=1.4, draw=none] {\scalebox{0.7}{$N$}};
\end{tikzpicture}
                    \caption{\small\emph{Batch HT format}}
                    \label{fig:htucker_batch_new}
                \end{subfigure}
                
                \caption{\small\emph{Tensor network diagrams of an 8D tensor and its representations in Tucker format, hierarchical Tucker format, and batch hierarchical Tucker format. The streaming/batch dimension is labeled $N$}}
                \label{fig:batch_hierarchical_tucker}
            \end{figure}

            This section first presents the idea of computing an approximation for a batch of similar tensors in hierarchical Tucker format, then proposes a method to update an existing approximation in batch hierarchical Tucker format when new batches of tensors become available.

    \subsection{Batch hierarchical Tucker}\label{sec:batch_ht}
        In this section, we aim to provide a solution to \Cref{prob:batch_ht} by presenting a hierarchical Tucker decomposition algorithm for batches of tensors.
        One approach to compressing a batch of tensors in HT format would be to treat the batch dimension as an additional dimension of the tensor, resulting in the  batch dimension represented as another Tucker leaf on the $p$-th layer, as shown in Figure~\ref{fig:htucker_as_node}. Assuming that the individual tensors in a batch are similar, i.e., low rank within the batch, such an approach can lead to an inefficient/inaccurate representation. 
        As an alternative to adding a new leaf, we propose grouping the batch dimension into the root, as shown in Figure~\ref{fig:htucker_batch_new}. To do this, we exclude the batch dimension from the dimension tree during construction as if we are decomposing a single tensor from the batch. 

        Specifically, for an $N$-batch of $d$-dimensional tensors $\mathcal{Y}\in\reals^{n_{1}\times\cdots\times n_{d}\times N}$, we use a dimension tree $\tree$ corresponding to a tensor with only the first $d$ dimensions of size $n_{1},\dots,n_{d}$.  Throughout the decomposition process, the batch dimension is absorbed in the \texttt{SVD} computation and leaf/node contraction.
        Practically, this is performed by appending the batch size $N$ to each layer's index set $\indexset_{\tree_{\ell}}$ for $\ell=0,\dots,p$. 
        The batch dimension remains intact throughout the decomposition process, and no SVD is performed for a reshaping relevant to this last dimension.
        As an example, after performing \texttt{HOSVD} on the $p$-th layer (except the batch dimension $N$) and contracting the obtained leaves as shown in \Cref{eq:ht_leaf_contraction}, we end up with the intermediate core tensor $\bar{\mathcal{C}}_{p}\in\reals^{\rank_{1}\times\cdots\times \rank_{|\tree_p|} \times n_{|\tree_p|+1} \times \cdots \times n_{d} \times N}$. We then reshape $\bar{\mathcal{C}}_{p}$ according to the modified index set $\indexset_{\tree_{p-1}}$ and carry on with the decomposition. Since the batch dimension has been excluded from any \texttt{HOSVD} computation throughout the entire process, the root core $\mathcal{G}_{0,1}$ ends up being 3 dimensional where the third/new dimension has size $N$. The difference between an HT and a batch-HT is shown in \Cref{fig:batch_hierarchical_tucker}. \Cref{fig:batch_htucker_steps} depicts this procedure step by step for a tensor with $d=5$.\footnote{For simplicity of presentation, the above discussion assumed that the batch dimension is ordered as the last dimension of the tensor. However, as long as the batch dimension is added to the dimension tree properly (i.e. inserted to the correct order), the position of the batch index does not matter.}
        
        In this format, only the dimensionality of the root node is increased by one, all other transfer transfer nodes are still three dimensional and leaf-nodes remain two-dimensional. The pseudocode of the corresponding algorithm \texttt{BHT-l2r} is presented in \Cref{alg:batch_htucker}. Theoretical guarantees for the approximation error upper bound for the \texttt{BHT-l2r} algorithm essentially follows from Theorem~\ref{thm:ht_nodewise_error} and is provided in \Cref{app:proofs} as \Cref{thm:bht_leaves_to_root} for completion. \Cref{app:ht_vs_bht} presents a detailed comparison between HT and BHT formats on both scientific and image data.
        \begin{algorithm}[!h]
            \caption{\texttt{BHT-l2r}: Error truncated leaves-to-root batch hierarchical Tucker decomposition}
            \label{alg:batch_htucker}
            \begin{algorithmic}[1]
                \Input
                    \Desc{$\mathcal{Y}\in\reals^{n_{1}\times\cdots\times n_{d}\times \newdim}$}{Input tensor}
                    \Desc{$\tree$}{Dimension tree of the decomposition with depth $p$}
                    \Desc{$\varepsilon_{rel}$}{Relative error tolerance}
                \EndInput
                \Output
                    \Desc{$\hierarchical_{\mathcal{Y}}$}{Hierarchical Tucker representation with leaves and cores $\core_{\ell,i_{\ell}}$}
                \EndOutput
                \State $\varepsilon_{nw}\gets \varepsilon_{rel}\|\mathcal{Y}\|_{F}/\sqrt{2d-2}$ \Comment{Node-wise error tolerance for SVD}
                \State $\mathcal{C}\gets \mathcal{Y}$
                \For{$i=1,\dots,|\tree_{p}|$} \Comment{Compute the leaves on layer $p$}
                    \State $C \gets \texttt{unfold}\left( \mathcal{C}, d_{p,i} \right)$ \Comment{$d_{p,i}\in\{1,\dots,d\}$ is the dimension corresponding to the leaf node $\node_{p,i}$}
                    \State $\core_{p,i}\gets \texttt{SVD}\left(C ,\varepsilon_{nw} \right)$ \Comment{Only the left singular vectors are kept with $\ranks_{p,i}=r_{p,i}$} \label{alg:line:svd_leaf}
                    \State $\indexset_{\node_{p,i}} \gets \{n_{i},r_{p,i}\}$ \Comment{Create index set for the leaf nodes}
                \EndFor
                \State $\mathcal{C} \gets \mathcal{C} \underset{\mathcal{I}}{\times} \llbracket \core_{p,1},\dots, \core_{p,|\tree_{p}|} \rrbracket $ \Comment{$\mathcal{I}=\{d_{p,i}; i=1,\dots,|\tree_{p}|\}$ as in \eqref{eq:ht_leaf_contraction}}
                \For{$\ell = p-1$ to $1$} 
                    \State $\mathcal{C}\gets \reshape\left( \mathcal{C}, \indexset_{\tree_{\ell}} \right)$ \Comment{$\indexset_{\tree_{\ell}} \cup \{\newdim\}$ with $\indexset_{\tree_{\ell}}$ constructed using \eqref{eq:ht_index_set_construction}}
                    \State $\core_{\ell,1},\dots,\core_{\ell,|\tree_{\ell}|} \gets \texttt{HOSVD}(\mathcal{C},\varepsilon_{nw})$ \Comment{$\indexset_{\node_{\ell,j}}$ are created using \eqref{eq:ht_tucker_core_reshape} for $j=1,\dots,|\tree_{\ell}|$} 
                    \For{$j=1,\dots,|\tree_{\ell}|$}
                        \State $\core_{\ell,j}\gets \reshape\left( \core_{\ell,j},\indexset_{\node_{\ell,j}}\right)$ \Comment{Folds $\core_{\ell,j}$ into 3D if $\node_{\ell,i_{\ell}}$ is a transfer node}
                        \EndFor
                    \State $\mathcal{C} \gets \mathcal{C} \times \llbracket \core_{\ell,1},\dots, \core_{\ell,|\tree_{\ell}|} \rrbracket $
                \EndFor
                \State $\mathcal{G}_{0,1}\gets \reshape\left( \mathcal{C},\indexset_{\node_{0,1}} \right)$ \Comment{$\indexset_{\node_{0,1}}=\{r_{1,1},r_{1,2},\newdim\}$}
            \end{algorithmic}
        \end{algorithm}

        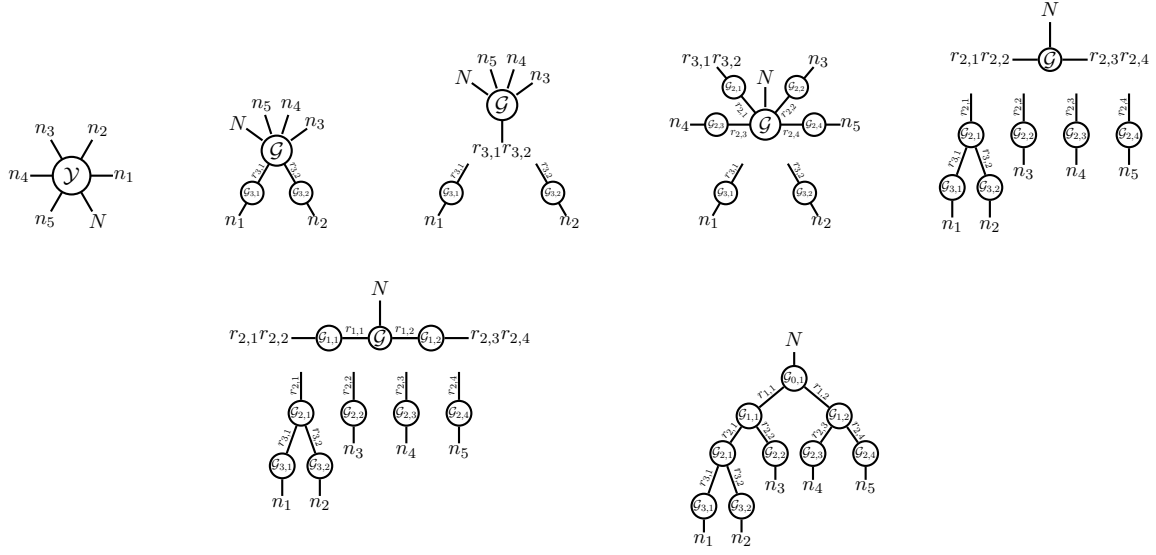
\begin{figure}[!h]
            \centering
            \begin{subfigure}{0.115\columnwidth}\centering
                \begin{tikzpicture}
    \def\k{6} 
    \def\radius{0.25cm} 
    \def\spikelength{0.3cm} 

    \draw[thick] (0,0) circle (\radius);
    \node at (0,0) {\scriptsize $\mathcal{Y}$};

    \foreach \n in {1,...,\k}{
        \pgfmathsetmacro\angle{360/\k * (\n - 1)}
        \draw[thick] (\angle:\radius) -- (\angle:\radius+\spikelength);
        \ifnum\n=\k{
            \node at (\angle:\spikelength + \radius + 0.15cm) {\scalebox{0.7}{$N$}};
        }
        \else{
            \node at (\angle:\spikelength + \radius + 0.15cm) {\scalebox{0.7}{$n_{\n}$}};
        }
        \fi
    }
\end{tikzpicture}
            \end{subfigure}\hfill
            \begin{subfigure}{0.1\columnwidth}\centering
                \begin{tikzpicture}
    \def\k{6} 
    \def\radius{0.2cm} 
    \def\spikelength{0.5cm} 
    \def\secradius{0.15cm} 
    \def\secspikelength{0.2cm} 

    \draw[thick] (0,0) circle (\radius);
    \node at (0,0) (core) {\scriptsize$\mathcal{G}$};

    \coordinate (sec_center) at (240:\spikelength + 1*\secradius);
    \coordinate (sec_edge) at (240:\spikelength + 2*\secradius);
    \node at (sec_center) (g_3_1) {\scalebox{0.35}{$\core_{3,1}$}};
    \draw[thick] (sec_center) circle (\secradius);
    \draw[thick] (sec_edge) -- ++(240:\secspikelength) node[pos=1.5, draw=none] {\scalebox{0.7}{$n_{1}$}};

    \coordinate (sec_center) at (300:\spikelength + 1*\secradius);
    \coordinate (sec_edge) at (300:\spikelength + 2*\secradius);
    \node at (sec_center) (g_3_2) {\scalebox{0.35}{$\core_{3,2}$}};
    \draw[thick] (sec_center) circle (\secradius);
    \draw[thick] (sec_edge) -- ++(300:\secspikelength) node[pos=1.5, draw=none] {\scalebox{0.7}{$n_{2}$}};

    \coordinate (sec_center) at (36:1*\spikelength + 0*\secradius);
    \draw[thick] (root) -- (sec_center) node[pos=1.5, draw=none] {\scalebox{0.7}{$n_{3}$}};
    \coordinate (sec_center) at (72:1*\spikelength + 0*\secradius);
    \draw[thick] (root) -- (sec_center) node[pos=1.5, draw=none] {\scalebox{0.7}{$n_{4}$}};
    \coordinate (sec_center) at (108:1*\spikelength + 0*\secradius);
    \draw[thick] (root) -- (sec_center) node[pos=1.5, draw=none] {\scalebox{0.7}{$n_{5}$}};
    \coordinate (sec_center) at (144:1*\spikelength + 0*\secradius);
    \draw[thick] (root) -- (sec_center) node[pos=1.5, draw=none] {\scalebox{0.7}{$N$}};

    \path[-,thick]
        (240:\radius) edge node[yshift=2pt, sloped, minimum size=0.2cm] {\scalebox{0.4}{$r_{3,1}$}} (240:\spikelength)
        (300:\radius) edge node[yshift=2pt, sloped, minimum size=0.2cm] {\scalebox{0.4}{$r_{3,2}$}} (300:\spikelength)
    ;
\end{tikzpicture}
            \end{subfigure}\hfill
            \begin{subfigure}{0.14\columnwidth}\centering
                \begin{tikzpicture}
    \def\k{6} 
    \def\radius{0.2cm} 
    \def\spikelength{0.5cm} 
    \def\secradius{0.15cm} 
    \def\secspikelength{0.2cm} 
    \def\offset{0.7cm}

    \draw[thick] (0,0) circle (\radius);
    \node at (0,0) (core) {\scriptsize$\mathcal{G}$};

    \coordinate (sec_center) at (240:\spikelength + 1*\secradius+\offset);
    \coordinate (sec_edge) at (240:\spikelength + 2*\secradius + \offset);
    \node at (sec_center) (g_3_1) {\scalebox{0.35}{$\core_{3,1}$}};
    \draw[thick] (sec_center) circle (\secradius);
    \draw[thick] (sec_edge) -- ++(240:\secspikelength) node[pos=1.5, draw=none] {\scalebox{0.7}{$n_{1}$}};

    \coordinate (sec_center) at (300:\spikelength + 1*\secradius+\offset);
    \coordinate (sec_edge) at (300:\spikelength + 2*\secradius + \offset);
    \node at (sec_center) (g_3_2) {\scalebox{0.35}{$\core_{3,2}$}};
    \draw[thick] (sec_center) circle (\secradius);
    \draw[thick] (sec_edge) -- ++(300:\secspikelength) node[pos=1.5, draw=none] {\scalebox{0.7}{$n_{2}$}};

    \coordinate (sec_center) at (36:1*\spikelength + 0*\secradius);
    \draw[thick] (root) -- (sec_center) node[pos=1.5, draw=none] {\scalebox{0.7}{$n_{3}$}};
    \coordinate (sec_center) at (72:1*\spikelength + 0*\secradius);
    \draw[thick] (root) -- (sec_center) node[pos=1.5, draw=none] {\scalebox{0.7}{$n_{4}$}};
    \coordinate (sec_center) at (108:1*\spikelength + 0*\secradius);
    \draw[thick] (root) -- (sec_center) node[pos=1.5, draw=none] {\scalebox{0.7}{$n_{5}$}};
    \coordinate (sec_center) at (144:1*\spikelength + 0*\secradius);
    \draw[thick] (root) -- (sec_center) node[pos=1.5, draw=none] {\scalebox{0.7}{$N$}};

    \path[-,thick]
        (240:\radius+\offset) edge node[yshift=2pt, sloped, minimum size=0.2cm] {\scalebox{0.4}{$r_{3,1}$}} (240:\spikelength+\offset)
        (300:\radius+\offset) edge node[yshift=2pt, sloped, minimum size=0.2cm] {\scalebox{0.4}{$r_{3,2}$}} (300:\spikelength+\offset)

        (270:\radius) edge node[yshift=-1pt, anchor=north, minimum size=0.2cm] {\scalebox{0.7}{$r_{3,1}r_{3,2}$}} (270:\spikelength)
    ;
\end{tikzpicture}
            \end{subfigure}\hfill
            \begin{subfigure}{0.17\columnwidth}\centering
                \begin{tikzpicture}
    \def\k{6} 
    \def\radius{0.2cm} 
    \def\spikelength{0.5cm} 
    \def\secradius{0.15cm} 
    \def\secspikelength{0.2cm} 
    \def\offset{0.4cm}
    \def\smalloffset{0.15cm}
    \def\minioffset{0.05cm}

    \draw[thick] (0,0) circle (\radius);
    \node at (0,0) (core) {\scalebox{0.7}{$\mathcal{G}$}};

    \coordinate (sec_center) at (240:\spikelength + 1*\secradius+\offset);
    \coordinate (sec_edge) at (240:\spikelength + 2*\secradius + \offset);
    \node at (sec_center) (g_3_1) {\scalebox{0.35}{$\core_{3,1}$}};
    \draw[thick] (sec_center) circle (\secradius);
    \draw[thick] (sec_edge) -- ++(240:\secspikelength) node[pos=1.5, draw=none] {\scalebox{0.7}{$n_{1}$}};

    \coordinate (sec_center) at (300:\spikelength + 1*\secradius+\offset);
    \coordinate (sec_edge) at (300:\spikelength + 2*\secradius + \offset);
    \node at (sec_center) (g_3_2) {\scalebox{0.35}{$\core_{3,2}$}};
    \draw[thick] (sec_center) circle (\secradius);
    \draw[thick] (sec_edge) -- ++(300:\secspikelength) node[pos=1.5, draw=none] {\scalebox{0.7}{$n_{2}$}};

    \coordinate (sec_center) at (0:\spikelength + 1*\secradius);
    \coordinate (sec_edge) at (0:\spikelength + 2*\secradius);
    \draw[thick] (sec_center) circle (\secradius);
    \node at (sec_center) (g_3_2) {\scalebox{0.35}{$\core_{2,4}$}};
    \draw[thick] (sec_edge) -- ++(0:\secspikelength) node[pos=1.7, draw=none] {\scalebox{0.7}{$n_{5}$}};

    \coordinate (sec_center) at (50:\spikelength + 1*\secradius);
    \coordinate (sec_edge) at (50:\spikelength + 2*\secradius);
    \draw[thick] (sec_center) circle (\secradius);
    \node at (sec_center) (g_3_2) {\scalebox{0.35}{$\core_{2,2}$}};
    \draw[thick] (sec_edge) -- ++(50:\secspikelength) node[pos=1.5, draw=none] {\scalebox{0.7}{$n_{3}$}};

    \coordinate (sec_center) at (130:\spikelength + 1*\secradius);
    \coordinate (sec_edge) at (130:\spikelength + 2*\secradius);
    \draw[thick] (sec_center) circle (\secradius);
    \node at (sec_center) (g_3_2) {\scalebox{0.35}{$\core_{2,1}$}};
    \draw[thick] (sec_edge) -- ++(130:\secspikelength) node[pos=1.5, draw=none] {\scalebox{0.7}{$r_{3,1}r_{3,2}$}};

    \coordinate (sec_center) at (180:\spikelength + 1*\secradius);
    \coordinate (sec_edge) at (180:\spikelength + 2*\secradius);
    \draw[thick] (sec_center) circle (\secradius);
    \node at (sec_center) (g_3_2) {\scalebox{0.35}{$\core_{2,3}$}};
    \draw[thick] (sec_edge) -- ++(180:\secspikelength) node[pos=1.7, draw=none] {\scalebox{0.7}{$n_{4}$}};

    \coordinate (sec_center) at (90:1*\spikelength + 0*\secradius);
    \draw[thick] (root) -- (sec_center) node[pos=1.5, draw=none] {\scalebox{0.7}{$N$}};

    \path[-,thick]
        (240:\radius+\offset) edge node[yshift=2pt, sloped, minimum size=0.2cm] {\scalebox{0.4}{$r_{3,1}$}} (240:\spikelength+\offset)
        (300:\radius+\offset) edge node[yshift=2pt, sloped, minimum size=0.2cm] {\scalebox{0.4}{$r_{3,2}$}} (300:\spikelength+\offset)

        (0:\radius) edge node[yshift=2pt, sloped, anchor=north, minimum size=0.2cm] {\scalebox{0.4}{$r_{2,4}$}} (0:\spikelength)
        (50:\radius) edge node[yshift=2pt, sloped, anchor=north, minimum size=0.2cm] {\scalebox{0.4}{$r_{2,2}$}} (50:\spikelength)
        (130:\radius) edge node[yshift=2pt, sloped, anchor=north, minimum size=0.2cm] {\scalebox{0.4}{$r_{2,1}$}} (130:\spikelength)
        (180:\radius) edge node[yshift=2pt, sloped, anchor=north, minimum size=0.2cm] {\scalebox{0.4}{$r_{2,3}$}} (180:\spikelength)
    ;
\end{tikzpicture}
            \end{subfigure}\hfill
            \begin{subfigure}{0.18\columnwidth}\centering
                \begin{tikzpicture}[
    level 1/.style={sibling distance=0.7cm, level distance=1cm},
    level 2/.style={sibling distance=0.5cm, level distance=0.7cm},
    level 3/.style={sibling distance=0.5cm, level distance=0.7cm},
    every node/.style={circle, draw, minimum size=0.3cm, inner sep=0pt},
    edge from parent/.style={draw= none, thick, sloped, anchor=south},
    edge from parent node/.style={midway, sloped},
    every path/.style={thick},
    ]
    \def\core{\mathcal{G}}
    \def\leafspikelength{0.4cm}
    \def\spikelength{0.5cm}

    \node (root) {\scalebox{0.7}{$\core$}}
            child {node (g_2_1) {\scalebox{0.45}{$\core_{2,1}$}}
                child{
                    node (g_3_1) {\scalebox{0.45}{$\core_{3,1}$}}
                }
                child{
                    node (g_3_2) {\scalebox{0.45}{$\core_{3,2}$}}
                }
            }
            child {node (g_2_2) {\scalebox{0.45}{$\core_{2,2}$}}
            }
            child {node (g_2_3) {\scalebox{0.45}{$\core_{2,3}$}}
            }
            child {node (g_2_4) {\scalebox{0.45}{$\core_{2,4}$}}
            }
        ;
    \path[-] 



        (g_2_1) edge node[yshift=-2pt,draw=none,sloped,anchor=south,minimum size=0.2cm] {\scalebox{0.45}{$r_{3,1}$}} (g_3_1)
        (g_2_1) edge node[yshift=-2pt,draw=none,sloped,anchor=south,minimum size=0.2cm] {\scalebox{0.45}{$r_{3,2}$}} (g_3_2)


    ;

  \newcounter{s}
  \setcounter{s}{1}

  \foreach \pos in {(g_3_1),(g_3_2),(g_2_2),(g_2_3),(g_2_4)} {
      \draw[thick] \pos -- ++(270:\leafspikelength) node[pos=1.5, draw=none] {\scalebox{0.7}{$n_{\thes}$}};
      \stepcounter{s};
  }
  \draw[thick] (root) -- ++(90:\spikelength) node[pos=1.5, draw=none] {\scalebox{0.7}{$N$}};
  \draw[thick] (root) -- ++(0:\spikelength) node[anchor= west, draw=none] {\scalebox{0.7}{$r_{2,3}r_{2,4}$}};
  \draw[thick] (root) -- ++(180:\spikelength) node[anchor=east, draw=none] {\scalebox{0.7}{$r_{2,1}r_{2,2}$}};
  \draw[thick] (g_2_1) -- ++(90:\spikelength+0.05cm) node[rotate=90, yshift=2pt, xshift=-5pt, draw=none] {\scalebox{0.45}{$r_{2,1}$}};
  \draw[thick] (g_2_2) -- ++(90:\spikelength+0.05cm) node[rotate=90, yshift=2pt, xshift=-5pt, draw=none] {\scalebox{0.45}{$r_{2,2}$}};
  \draw[thick] (g_2_3) -- ++(90:\spikelength+0.05cm) node[rotate=90, yshift=2pt, xshift=-5pt, draw=none] {\scalebox{0.45}{$r_{2,3}$}};
  \draw[thick] (g_2_4) -- ++(90:\spikelength+0.05cm) node[rotate=90, yshift=2pt, xshift=-5pt, draw=none] {\scalebox{0.45}{$r_{2,4}$}};
\end{tikzpicture}
            \end{subfigure}

            \begin{subfigure}{0.24\columnwidth}\centering
                \begin{tikzpicture}[
    level 1/.style={sibling distance=0.7cm, level distance=1cm},
    level 2/.style={sibling distance=0.5cm, level distance=0.7cm},
    level 3/.style={sibling distance=0.5cm, level distance=0.7cm},
    every node/.style={circle, draw, minimum size=0.3cm, inner sep=0pt},
    edge from parent/.style={draw= none, thick, sloped, anchor=south},
    edge from parent node/.style={midway, sloped},
    every path/.style={thick},
    ]
    \def\core{\mathcal{G}}
    \def\leafspikelength{0.4cm}
    \def\spikelength{0.5cm}

    \node (root) {\scalebox{0.7}{$\core$}}
            child {node (g_2_1) {\scalebox{0.45}{$\core_{2,1}$}}
                child{
                    node (g_3_1) {\scalebox{0.45}{$\core_{3,1}$}}
                }
                child{
                    node (g_3_2) {\scalebox{0.45}{$\core_{3,2}$}}
                }
            }
            child {node (g_2_2) {\scalebox{0.45}{$\core_{2,2}$}}
            }
            child {node (g_2_3) {\scalebox{0.45}{$\core_{2,3}$}}
            }
            child {node (g_2_4) {\scalebox{0.45}{$\core_{2,4}$}}
            }
        ;
    
  \newcounter{t}
  \setcounter{t}{1}

  \foreach \pos in {(g_3_1),(g_3_2),(g_2_2),(g_2_3),(g_2_4)} {
      \draw[thick] \pos -- ++(270:\leafspikelength) node[pos=1.5, draw=none] {\scalebox{0.7}{$n_{\thet}$}};
      \stepcounter{t};
  }
  \draw[thick] (root) -- ++(90:\spikelength) node[pos=1.5, draw=none] {\scalebox{0.7}{$N$}};
  \draw[thick] (root) -- ++(180:\spikelength) node[anchor=east, draw] (g_1_1) {\scalebox{0.45}{$\core_{1,1}$}};
  \draw[thick] (root) -- ++(0:\spikelength) node[anchor= west, draw] (g_1_2) {\scalebox{0.45}{$\core_{1,2}$}};
  \draw[thick] (g_2_1) -- ++(90:\spikelength+0.05cm) node[rotate=90, yshift=2pt, xshift=-5pt, draw=none] {\scalebox{0.45}{$r_{2,1}$}};
  \draw[thick] (g_2_2) -- ++(90:\spikelength+0.05cm) node[rotate=90, yshift=2pt, xshift=-5pt, draw=none] {\scalebox{0.45}{$r_{2,2}$}};
  \draw[thick] (g_2_3) -- ++(90:\spikelength+0.05cm) node[rotate=90, yshift=2pt, xshift=-5pt, draw=none] {\scalebox{0.45}{$r_{2,3}$}};
  \draw[thick] (g_2_4) -- ++(90:\spikelength+0.05cm) node[rotate=90, yshift=2pt, xshift=-5pt, draw=none] {\scalebox{0.45}{$r_{2,4}$}};

  \draw[thick] (g_1_1) -- ++(180:\spikelength) node[anchor=east, draw=none] {\scalebox{0.7}{$r_{2,1}r_{2,2}$}};
  \draw[thick] (g_1_2) -- ++(0:\spikelength) node[anchor=west, draw=none] {\scalebox{0.7}{$r_{2,3}r_{2,4}$}};

  \path[-] 
        (root) edge node[yshift=-2pt,draw=none,sloped,anchor=south,minimum size=0.2cm] {\scalebox{0.45}{$r_{1,1}$}} (g_1_1)
        (root) edge node[yshift=-2pt,draw=none,sloped,anchor=south,minimum size=0.2cm] {\scalebox{0.45}{$r_{1,2}$}} (g_1_2)



        (g_2_1) edge node[yshift=-2pt,draw=none,sloped,anchor=south,minimum size=0.2cm] {\scalebox{0.45}{$r_{3,1}$}} (g_3_1)
        (g_2_1) edge node[yshift=-2pt,draw=none,sloped,anchor=south,minimum size=0.2cm] {\scalebox{0.45}{$r_{3,2}$}} (g_3_2)


    ;

\end{tikzpicture}
            \end{subfigure}\hspace{2cm}
            \begin{subfigure}{0.22\columnwidth}\centering
                \begin{tikzpicture}[
    level 1/.style={sibling distance=1.2cm, level distance=0.5cm},
    level 2/.style={sibling distance=0.7cm, level distance=0.5cm},
    level 3/.style={sibling distance=0.5cm, level distance=0.7cm},
    every node/.style={circle, draw, minimum size=0.15cm, inner sep=0pt},
    edge from parent/.style={draw= none, thick, sloped, anchor=south},
    edge from parent node/.style={midway, sloped},
    every path/.style={thick},
    ]
    \def\core{\mathcal{G}}
    \def\spikelength{0.35cm}

    \node (root) {\scalebox{0.45}{$\core_{0,1}$}}
        child {node (g_1_1) {\scalebox{0.45}{$\core_{1,1}$}}
            child {node (g_2_1) {\scalebox{0.45}{$\core_{2,1}$}}
                child{
                    node (g_3_1) {\scalebox{0.45}{$\core_{3,1}$}}
                }
                child{
                    node (g_3_2) {\scalebox{0.45}{$\core_{3,2}$}}
                }
            }
            child {node (g_2_2) {\scalebox{0.45}{$\core_{2,2}$}}
            }
        }
        child {node (g_1_2) {\scalebox{0.45}{$\core_{1,2}$}}
            child {node (g_2_3) {\scalebox{0.45}{$\core_{2,3}$}}
            }
            child {node (g_2_4) {\scalebox{0.45}{$\core_{2,4}$}}
            }
        };
    \path[-] 
        (root) edge node[yshift=-2pt,draw=none,sloped,anchor=south,minimum size=0.2cm] {\scalebox{0.45}{$r_{1,1}$}} (g_1_1)
        (root) edge node[yshift=-2pt,draw=none,sloped,anchor=south,minimum size=0.2cm] {\scalebox{0.45}{$r_{1,2}$}} (g_1_2)

        (g_1_1) edge node[yshift=-2pt,draw=none,sloped,anchor=south,minimum size=0.2cm] {\scalebox{0.45}{$r_{2,1}$}} (g_2_1)
        (g_1_1) edge node[yshift=-2pt,draw=none,sloped,anchor=south,minimum size=0.2cm] {\scalebox{0.45}{$r_{2,2}$}} (g_2_2)

        (g_1_2) edge node[yshift=-2pt,draw=none,sloped,anchor=south,minimum size=0.2cm] {\scalebox{0.45}{$r_{2,3}$}} (g_2_3)
        (g_1_2) edge node[yshift=-2pt,draw=none,sloped,anchor=south,minimum size=0.2cm] {\scalebox{0.45}{$r_{2,4}$}} (g_2_4)

        (g_2_1) edge node[yshift=-2pt,draw=none,sloped,anchor=south,minimum size=0.2cm] {\scalebox{0.45}{$r_{3,1}$}} (g_3_1)
        (g_2_1) edge node[yshift=-2pt,draw=none,sloped,anchor=south,minimum size=0.2cm] {\scalebox{0.45}{$r_{3,2}$}} (g_3_2)


    ;

  \newcounter{q}
  \setcounter{q}{1}

  \foreach \pos in {(g_3_1),(g_3_2),(g_2_2),(g_2_3),(g_2_4)} {
      \draw[thick] \pos -- ++(270:\spikelength) node[anchor=north, yshift=2pt, draw=none] {\scalebox{0.7}{$n_{\theq}$}};
      \stepcounter{q};
  }
  \draw[thick] (root) -- ++(90:\spikelength) node[anchor=south, draw=none] {\scalebox{0.7}{$N$}};
\end{tikzpicture}
            \end{subfigure}
            \caption{\small\emph{Step-by-step decomposition of an $N$-batch of 5-D tensors with the \texttt{BHT-l2r} algorithm~(\Cref{alg:batch_htucker}). The decomposition starts with the leaves of the last layer. Since the dimension tree is constructed beforehand, \Cref{alg:batch_htucker} has the information about which dimensions' leaves will be on which layer through the dimension tree $\tree$. Note that the batch dimension ($N$) remains intact throughout the entire decomposition process.}}
            \label{fig:batch_htucker_steps}
        \end{figure}

    \subsection{Incremental Updates}
        In this section we propose a solution to \Cref{prob:incremental_ht} via a method to update an existing HT representation incrementally when new batches of tensors become available. Assume that at time $k$, we have an HT approximation $\hierarchical_{\mathcal{X}^{k-1}}$ of the accumulation tensor $\mathcal{X}^{k-1}$ in batch-HT format $(\tree, \bm{\core}^{k-1})$. Then, a new batch of $\newdim^{k}$ $d$-dimensional tensors $\mathcal{Y}^{k}\in\reals^{n_{1}\times\cdots n_{d}\times\newdim^{k}}$ arrives, and our task is to update $\bm{\core}^{k-1}$ to $\bm{\core}^{k}$ so as to generate a new approximation $\hierarchical_{\mathcal{X}^{k}}$. In other words, we describe an approach to update each core $\core^{k-1}_{\ell,i_{\ell}}$ into a new core $\core^{k}_{\ell, i_{\ell}}$, assuming an unchanged dimension tree $\tree$. The proposed approach has three components: 1) project onto existing HT cores 2) compute residuals, and 3) update HT cores. The overall algorithm, \ouralgorithm, is provided in \Cref{alg:HIT} and each step is described in detail next.
        \begin{algorithm}[!hbtp]
            \centering
            \resizebox{!}{0.45\textheight}{
            \caption{\ouralgorithm: Incremental updates to a tensor represented as batch hierarchical Tucker}
            \label{alg:HIT}
            \begin{minipage}{1.15\textwidth}
            \begin{algorithmic}[1]
                    \Input
                        \Desc{$\mathcal{Y}^{k}\in\reals^{n_{1}\times\cdots\times n_{d}\times \newdim^{k}}$}{streamed $N^{k}$-batch of tensors at the $k$-th step}
                        \Desc{$\hierarchical_{\mathcal{X}^{k-1}}$}{Hierarchical Tucker representation with dimension tree $\tree$ and cores $\mathcal{G}^{k-1}_{\ell,j}$}
                        \Desc{$\indexset^{k-1}$}{index set from $(k-1)$-th batch}
                        \Desc{$\varepsilon_{rel}$}{desired relative error truncation threshold}
                    \EndInput
                    \Output
                        \Desc{$\hierarchical_{\mathcal{X}^{k}}$}{\textit{Updated} hierarchical Tucker representation with dim. tree $\tree$ and cores $\mathcal{G}^{k}_{\ell,j}$}
                        \Desc{$\indexset^{k}$}{\textit{Updated} index set }
                    \EndOutput
                    \State \textcolor{grey}{\textbf{Project streamed batch and check the representation quality}}
                    \State $\varepsilon_{des}\gets \varepsilon_{rel}\|\mathcal{Y}^{k}\|_{F}$
                    \State $\mathcal{C}\gets\mathcal{Y}^{k}$
                    \For{$\ell=p$ to $1$} \Comment{Project $\mathcal{Y}^{k}$ onto existing cores}\label{alg:line:projection_start}
                        \State $\mathcal{C} \gets \mathcal{C} \underset{\mathcal{I}}{\times} \llbracket \core_{\ell,1},\dots, \core_{\ell,|\tree_{\ell}|} \rrbracket $ \Comment{$\mathcal{I}=\{d_{p,i}; i=1,\dots,|\tree_{p}|\}$ needed just for the $p$-th layer}
                        \State $\mathcal{C}\gets\reshape\left( \mathcal{C},\indexset^{k-1}_{\tree_{\ell-1}} \right)$ \Comment{Reshape $\mathcal{C}$ for the next layer}
                    \EndFor \label{alg:line:projection_end}
                    \If{$\sqrt{\|\mathcal{Y}^{k}\|^{2}_{F}-\|\mathcal{C}\|^{2}_{F}}\leq\varepsilon_{des}$}
                        \State \textcolor{grey}{\textbf{Skip updating all cores \textit{except the root node}}}
                        \State $\mathcal{G}^{k}\gets \mathcal{G}^{k-1}\qquad\forall \mathcal{G}^{k-1}\in {\hierarchical_{\mathcal{X}^{k-1}}}\setminus \mathcal{G}^{k-1}_{0,1}$
                    \Else
                        \State \textcolor{grey}{\textbf{Update cores on the last layer}}
                        \State $\mathcal{C}\gets \mathcal{Y}^{k}$
                        \State $\varepsilon_{nw}\gets \varepsilon_{des}/\sqrt{2d-2}$ \Comment{Compute node-wise truncation error tolerance from $\varepsilon_{des}$}
                        \For{$j=1$ to $|\tree_{p}|$}
                            \State $C\gets \texttt{unfold}\left( \mathcal{C},d_{p,j} \right)$ \Comment{$d_{p,j}\in\{1,\dots,d\}$ is the index of the dimension corresponding to $\node_{p,j}$}
                            \State $R_{p,j}\gets\Pi^{k-1}_{p,j}C$ \Comment{$\Pi^{k-1}_{p,j}$ is computed according to \eqref{eq:compute_residual}}
                            \State $U^{k}_{R}\gets \texttt{SVD}(R_{p,j},\varepsilon_{nw})$ \Comment{Error-truncated \texttt{SVD} on the residual according to \eqref{eq:svd_on_residual}}
                            \State $\core^{k}_{p,j},\indexset^{k}_{\node_{p,j}} \gets \texttt{expandCore}(\node_{p,j},\indexset^{k-1}_{\node_{p,j}},\core^{k-1}_{p,j},U^{k}_{R})$ \Comment{Using \Cref{alg:expandCore,alg:updateIndexSet}}
                            \State $\core^{k-1}_{p-1,m}\gets\texttt{padWithZeros}(\core^{k-1}_{p-1,m},t,r_{R_{p,j}})$ \Comment{Assume $\node_{p,j}$ is the $t$-th successor of $\node_{p-1,m}$}
                        \EndFor
                        \State $\indexset^{k}_{\tree_{p-1}} \gets \texttt{updateIndexSet}(\indexset^{k-1}_{\tree_{p-1}},\tree_{p-1})$ \Comment{Update using \Cref{alg:updateIndexSet} according to \eqref{eq:ht_index_set_construction}}
                        \State $\mathcal{C} \gets \mathcal{C} \underset{\mathcal{I}}{\times} \llbracket \core^{k}_{p,1},\dots,\core^{k}_{p,|\tree_{p}|} \rrbracket $ \Comment{$\mathcal{I}=\{d_{p,i}; i=1,\dots,|\tree_{p}|\}$ as in \eqref{eq:ht_leaf_contraction}}
                        \State \textcolor{grey}{\textbf{Update cores on the remaining layers}}
                        \For{$\ell=p-1$ to $1$}
                            \State $\mathcal{C}\gets \texttt{reshape}\left( \mathcal{C},\indexset^{k}_{\tree_{\ell}} \right)$
                            \For{$j=1$ to $|\tree_{\ell}|$}
                                \State $C\gets \texttt{unfold}\left( \mathcal{C},\ell \right)$ 
                                \State $R_{\ell,j}\gets\Pi^{k-1}_{\ell,j}C$ \Comment{$\Pi^{k-1}_{\ell,j}$ is computed according to \eqref{eq:compute_residual}}
                                \State $U^{k}_{R}\gets \texttt{SVD}(R_{\ell,j},\varepsilon_{nw})$ \Comment{Compute error truncated \texttt{SVD} on the residual according to \eqref{eq:svd_on_residual}}
                                \State $\core^{k}_{\ell,j},\indexset^{k}_{\node_{\ell,j}} \gets \texttt{expandCore}(\node_{\ell,j},\indexset^{k-1}_{\node_{\ell,j}},\core^{k-1}_{\ell,j},U^{k}_{R})$ \Comment{Using \Cref{alg:expandCore,alg:updateIndexSet}}
                                \State $\core^{k-1}_{\ell-1,m}\gets\texttt{padWithZeros}(\core^{k-1}_{\ell-1,m},t,r_{R_{\ell,j}})$ \Comment{Assume $\node_{\ell,j}$ is the $t$-th successor of $\node_{\ell-1,m}$}
                            \EndFor
                            \State $\indexset^{k}_{\tree_{\ell-1}} \gets \texttt{updateIndexSet} (\indexset^{k-1}_{\tree_{\ell-1}},\tree_{\ell-1})$
                            \State $\mathcal{C} \gets \mathcal{C}\times \llbracket \core^{k}_{\ell,1},\dots,\core^{k}_{\ell,|\tree_{\ell}|} \rrbracket$
                        \EndFor
                        \EndIf
                    \State $\mathcal{G}^{k}_{0,1}\gets \mathcal{G}^{k-1}_{0,1} \pad^{3} \mathcal{C}$ 
                    \end{algorithmic}
                    \end{minipage}
                    }
                \end{algorithm}
        \paragraph{\textbf{Step 1:} Projection onto existing HT cores}
            Our proposed approach is centered around the fact that both Tucker leaves and Tucker cores are matrices with orthonormal columns (or under reshapings) when trained with \Cref{alg:batch_htucker}. This structure allows us to compute an approximation of newly streamed data using an existing set of hierarchical cores through a simple projection. This approximation will then guide refinement.

            The projection is done by contracting the incoming batch with existing cores sequentially in leaves to the root direction (\Cref{alg:HIT} lines \ref{alg:line:projection_start}-\ref{alg:line:projection_end}) and results in $\tilde{\mathcal{C}}
            _{1}$ ,the latent representation of $\mathcal{Y}^{k}$ using the cores of $\hierarchical_{\mathcal{X}^{k-1}}$.
            Since the projection onto the cores is simply a series of orthogonal projections, we can compute how well a set of existing HT representation approximates the input tensor by computing the difference in Frobenius norm between the input tensor and the projected tensor\footnote{Please refer to \Cref{thm:reconstruction_norm_equality} for details and proof.}. This error is then used to determine if the cores need to be updated. 
            Then, the error of the projection $\varepsilon_{proj}$ can be directly computed as
            \begin{equation}\label{eq:relative_error_threshold}
                    \varepsilon_{proj}=\sqrt{\|\mathcal{Y}^{k}\|^{2}_{F}-\|\bar{\mathcal{C}}^{k}_{1}\|^{2}_{F}}.
            \end{equation}
            If $\varepsilon_{proj}$ is above the desired threshold $\varepsilon_{abs}$, the cores are updated to reduce the error below $\varepsilon_{abs}$. These updates will start from the last layer (layer $p$) and propagate towards the root of the hierarchical representation (layer $0$) sequentially. 

        \paragraph{\textbf{Step 2:} Computing residuals}
            To identify the missing orthogonal directions in the existing HT representation, we compute the residual $R_{\ell,j}$ for each core $\core^{k-1}_{\ell,j}$ that corresponds to the missing information related to the newest data batch. The residual is computed by projecting the mode-$j$ unfolding of the core $\mathcal{C}^{k}_{\ell}$ onto the orthogonal complement of the existing core $\core^{k-1}_{\ell,j}$ as
                \begin{equation}\label{eq:compute_residual}
                    \resizebox{\textwidth}{!}{$
                    R_{\ell,j} = \Pi^{k-1}_{\ell,j} C^{k}_{\ell,(j)},\quad \text{where} \quad \Pi^{k-1}_{\ell,j} = I - U^{k-1}_{\ell,j}\left(U^{k-1}_{\ell,j}\right)^{T}
                    \text{with}~ U^{k-1}_{\ell,j} = 
                    \begin{cases}
                        \core^{k-1}_{\ell,j} \quad \text{if $\node_{\ell,j}$ is a leaf} \\
                        \texttt{reshape}\left( \core^{k}_{\ell,j}, \left[\alpha,r^{k-1}_{\ell,j}\right] \right) ~ \text{else}
                    \end{cases}, 
                $}
            \end{equation}
            where $\Pi^{k-1}_{\ell,j}$ is the projection operator and $\alpha = \frac{1}{\rank^{\ell}_{j}}\prod_{\gamma\in\indexset^{k-1}_{\node^{\ell}_{j}}}\gamma.$ 
            Note that the projection operator $\Pi^{k-1}_{\ell,j}$ slightly differs for the Tucker leaves and the Tucker cores.  Since $\core_{\ell,j}$ are already 2-dimensional for leaf nodes, there is no reshaping required, while the transfer nodes $\core^{k-1}_{\ell,j}$ must first be reshaped into orthonormal matrices $U^{k-1}_{\ell,j}\in\reals^{\alpha\times r^{k-1}_{\ell,j}}$.
            Once the residual $R_{\ell,j}$ is computed, the next step is to find the directions in which $\core^{k-1}_{\ell,j}$ must be expanded.

            \paragraph{\textbf{Step 3:} Performing core updates}
            The idea behind incremental updates is similar to~\cite{aksoy2024incremental}. It seeks to append directions that span the residual to the existing basis. The core update process starts with an error truncated \texttt{SVD} on the residual $R_{\ell,j}\in\reals^{\alpha\times\beta}$ with $\alpha$ same as Step 2 and $\beta=\frac{1}{\alpha}\prod_{\theta\in\indexset_{\tree_{\ell}}}\theta$
            as
            \begin{equation}\label{eq:svd_on_residual}
                R_{\ell,j}=U^{k}_{R_{\ell,j}}\Sigma^{k}_{R_{\ell,j}}(V^{k}_{R_{\ell,j}})^{T}+E_{R_{\ell,j}},
            \end{equation}
            such that $U^{k}_{R_{\ell,j}}\in\reals^{\alpha\times \rank_{R_{\ell,j}}}$, $\Sigma^{k}_{R_{\ell,j}}\in\reals^{r_{R_{\ell,j}}\times r_{R_{\ell,j}}}$, and $V^{k}_{R_{\ell,j}}\in\reals^{\beta\times \rank_{R_{p,1}}}$. Similar to \texttt{BHT-l2r} algorithm, we distribute the error uniformly over the cores and determine $\rank_{R_{\ell,j}}$ such that the truncation error $\|E_{R_{\ell,j}}\|_{F}$ is at most $\varepsilon_{abs}/\sqrt{2d-2}$. Other approaches to distribute the error over the cores can be considered as well. We include one such alternative method in \Cref{app:alt_error}.

            Since we compute $U^{k}_{R_{\ell,j}}$ from the residual $R_{\ell,j}$, we have  $U^{k}_{R_{\ell,j}}\perp U^{k-1}_{\ell,j}$ by definition. This allows us to expand the orthogonal bases of $U^{k-1}_{\ell,j}$ by simply concatenating with $U^{k}_{R_{\ell,j}}$ as
            \begin{equation}\label{eq:concatenating_orthonormal_vectors}
                U_{\ell,j}^{k}=\mat{U_{\ell,j}^{k-1} & U^{k}_{R_{\ell,j}}},
            \end{equation}
            with $U_{\ell,j}^{k}\in\reals^{\alpha\times r^{k}_{\ell,j}}$, such that $r^{k}_{\ell,j}=r^{k-1}_{\ell,j}+r_{R_{\ell,j}}$. An update of the index set $\indexset^{k}_{\node_{\ell,j}}$ using \eqref{eq:ht_tucker_core_reshape} follows \eqref{eq:concatenating_orthonormal_vectors} to reflect the updated rank $r^{k}_{\ell,j}$. The new basis $U_{\ell,j}^{k}$ can then be appropriately reshaped into transfer or leaf cores in the same manner as for the HT described in Section~\ref{sec:error_truncated_htucker}. The new core $U_{\ell,j}^{k}$ is then reshaped appropriately to form $\core^{k}_{\ell,j}$.
            
            After updating $\core^{k-1}_{\ell,j}$ to $\core^{k}_{\ell,j}$, the new dimensions result in a shape mismatch between $\core^{k}_{\ell,j}$ and its parent Tucker core $\core^{k-1}_{\ell-1,t}$ on layer $(\ell-1)$. If $\core^{k-1}_{\ell-1,t}\in\reals^{r^{k-1}_{\ell,j}\times r^{k-1}_{\ell,m} \times r^{k-1}_{\ell-1,t}}$, then the core is padded according to
            \begin{equation}\label{eq:incremental_zero_padding1}
                \core^{k-1}_{\ell-1,t} = \core^{k-1}_{\ell-1,t}\pad^{1}\mathbf{0}_{r_{R_{\ell,1}}\times r^{k-1}_{\ell,m} \times r^{k-1}_{\ell-1,t}}, \quad \text{or} \quad \core^{k-1}_{\ell-1,t} = \core^{k-1}_{\ell-1,t}\pad^{2}\mathbf{0}_{r^{k-1}_{\ell,j} \times r_{R_{\ell,1}} \times r^{k-1}_{\ell-1,t}},
            \end{equation}
            depending on the order of $\core^{k-1}_{\ell,j}$ in the set of successors of $\core^{k-1}_{\ell-1,t}$.
            Following~\eqref{eq:incremental_zero_padding1}, the index sets of $\node_{\ell-1,t}$ is updated using \eqref{eq:ht_tucker_core_reshape} to reflect the padding.

            Steps 2 and 3 are repeated for all nodes on the $\ell$-th layer (i.e., for $j=1,\dots,|\tree_{\ell}|$). Although we describe the updates to be sequential within a layer, the updates can be computed in parallel to increase the computational speed. 
            Once all $\core_{\ell,j}^{k-1}$ are updated to $\core^{k}_{\ell,j}$ for $j=1,\dots,|\tree_{\ell}|$, the index set of $(\ell-1)$-th layer is updated to $\indexset^{k}_{\tree_{\ell-1}}$ using \eqref{eq:ht_index_set_construction} to reflect the rank updates. 
            
            \paragraph{\textbf{Step 1 (revisited):} Projection onto updated HT cores}
            Following core updates, the algorithm projects $\mathcal{C}^{k}_{\ell}$ onto the \textit{updated} cores on the $\ell$-th layer by
            \begin{equation}\label{eq:incremental_projection_onto_updated_cores}
                \bar{\mathcal{C}}^{k}_{\ell}= \mathcal{C}^{k}_{\ell} \times \llbracket \core^{k}_{\ell,1},\dots,\core^{k}_{\ell,|\tree_{\ell}|} \rrbracket.
            \end{equation}
            Then, the algorithm reshapes $\bar{\mathcal{C}}^{k}_{\ell}$ using the index set $\indexset^{k}_{\ell-1}$ to obtain $\mathcal{C}^{k}_{\ell-1}$. After reshaping, the algorithm returns to Step 2 and repeats the process for all layers up to to $\ell=1$.

            Once we update the cores on the first layer, $\core^{k}_{1,1}$ and $\core^{k}_{1,2}$ are used to project $\mathcal{C}^{k}_{1}$ and obtain $\bar{\mathcal{C}}^{k}_{1}$, the representation of $\mathcal{Y}^{k}$ using the updated Tucker cores.
            Finally, we update $\bar{\core}^{k-1}_{0}$ as
            \begin{equation}\label{eq:incremental_root_concatenate}
                \mathcal{G}^{k}_{0}=\bar{\mathcal{G}}^{k-1}_{0}\pad^{3}\bar{\mathcal{C}}^{k}_{1}, \quad \text{where} \quad \bar{\mathcal{C}}^{k}_{1} = \mathcal{C}^{k}_{1}\times \llbracket \core^{k}_{1,1},\core^{k}_{1,2} \rrbracket,
            \end{equation}
            and therefore conclude updating $\hierarchical_{\mathcal{X}^{k-1}}$ to $\hierarchical_{\mathcal{X}^{k}}$.
            The flow of this presented update scheme is summarized in \Cref{alg:HIT} as \ouralgorithm.

\section{Numerical Experiments}\label{sec:experiments}
 
In this section, we compare the proposed \ouralgorithm~algorithm against an incremental tensor train decomposition algorithm, \texttt{TT-ICE${}^{*}$}, which is proven to demonstrate state-of-the-art compression performance for the tensor train format in~\cite{aksoy2024incremental}.

Comparisons are made on both scientific as well as image based datasets. Scientific datasets include compressible Navier-Stokes simulations from PDEBench dataset~\citep{takamoto2022pdebench}, as well as simulations of a PDE-driven chaotic system of self-oscillating gels~\citep{alben2019semi}. For image based datasets, we will compare the algorithms with Minecraft video frames from MineRL Basalt competition dataset~\citep{milani2024bedd} and multispectral images from the BigEarthNet dataset~\citep{sumbul2019bigearthnet,sumbul2021bigearthnet}. \Cref{tab:datasets} summarizes the datasets used in the experiments. We also include further analyses including the effect of tensor reshapings (\Cref{app:reshaping}) and the effect of different axis reorderings (\Cref{app:axis_ordering}) on the performance of \ouralgorithm~in the appendices.

As a preview, we observe that \texttt{TT-ICE${}^{*}$} to perform better on simpler datasets such as the self-oscillating gel snapshots or at higher relative error tolerances, while \ouralgorithm~performs better on larger datasets with intricate multi-scale features. Furthermore, \ouralgorithm~is able to generalize better to unseen data with fewer batches compared to \texttt{TT-ICE${}^{*}$}. For image datasets, \ouralgorithm~retains more of the qualitative features of the original images compared to \texttt{TT-ICE${}^{*}$} at the same relative error threshold.

    \begin{table}[h!]
        \centering
        \caption{\small\emph{Summary of the datasets used in the experiments. Train units and test units refer to the number of simulations, videos, or images in the training and test sets, respectively. The batch size refers to the number of simulations or images in a single batch. The batch shape refers to the original shape of a single batch before performing any reshaping/resizing operations. The total size refers to the total size of the dataset on disk. Sims, vids, and ims refer to simulations, videos, and images, respectively.}}
        \label{tab:datasets}
        \resizebox{\textwidth}{!}{
        \begin{tabular}{l | c c c c c}
            Name & Train units& Test units & Batch size& Batch shape & Total size \\
            \hline
            PDEBench & 480 sims& 120 sims& 1 sim & $64\times 64\times 64\times 5 \times 21\times 1$ & 66 GB \\
            Self-oscillating gels & 8,000 sims& 15,000 sims & 1 sim & $3367\times 3\times 10\times 1$ & 18 GB \\
            Basalt MineRL & 5449 vids & 17 vids & 20 frames & $360\times 640 \times 3\times20$ & 183 GB \\
            BigEarthNet & 566,712 ims & 23,612 ims & 100 ims & $120\times120\times12\times100$ & 104 GB \\
        \end{tabular}
        }
    \end{table}

    All experiments are executed on University of Michigan compute nodes on the Lighthouse cluster with 16 Intel(R) Xeon(R) Platinum 8468 cores and 64GB of memory. All algorithms presented in this work are implemented in Python using the \texttt{NumPy} library~\citep{harris2020array}. We used the publicly available version of \texttt{TT-ICE${}^{*}$} from \href{https://github.com/dorukaks/TT-ICE}{GitHub} in our experiments.

    \subsection{Performance Metrics}\label{sec:performance_metrics}
        This section presents the performance metrics used to compare the algorithms. The metrics are chosen to evaluate the accuracy of the approximation, the compression ratio, the reduction ratio, the execution time, and the generalization performance.
        
        \subsubsection{Compression ratio (CR) and Reduction ratio (RR)}
            The compression ratio is a measure of how much each approximation compresses the original accumulation. It is defined as
            \begin{equation}\label{eq:compression_ratio}
                CR = \frac{\texttt{num\_elem}(\mathcal{X}^{k})}{\texttt{num\_elem}(\hierarchical_{\mathcal{X}^{k}})},
            \end{equation}
            where $\texttt{num\_elem}(\mathcal{X}^{k})$ denotes the number of elements of the original accumulation and $\texttt{num\_elem}(\hierarchical_{\mathcal{X}^{k}})$ is the number of elements of the approximation, which corresponds to the sum of number of elements in all HT cores.
            A compression ratio of $1$ indicates that the approximation does not compress the original accumulation at all, while a compression ratio below $1$ indicates to an inefficient approximation that uses more elements to approximate the original accumulation.

            As shown in~\cite{chen2024lowrank}, the \texttt{TT-ICE${}^{*}$} algorithm also provides a latent space representation of the data, which can be used in downstream learning tasks. For downstream tasks, the size of the input space plays a key role in computational efficiency. Therefore in this study we also investigate how much reduction is achieved through the incremental tensor decomposition algorithms and call this the \textit{reduction ratio} (RR). The RR is the ratio of the size of a single tensor in the accumulation to the size of the latent space. It is defined as
            \begin{equation}\label{eq:reduction_ratio}
                RR = \frac{\texttt{num\_elem}(\mathcal{Y}^{k,i})}{\texttt{num\_elem}(\texttt{encode}(\mathcal{Y}^{k,i}))},
            \end{equation}
            where $\mathcal{Y}^{k,i}$ denotes a single tensor from the $N^{k}$-batch of tensors, and $\texttt{project}(\mathcal{Y}^{k,i})$ denotes the resulting latent space representation from encoding $\mathcal{Y}^{k,i}$ using the relevant algorithm --- e.g., \texttt{TT-ICE${}^{*}$}, \ouralgorithm, etc..
        \subsubsection{Compression time}
            Compression time is a measure of how long it takes to update the approximation. Note that this does not consider the time it takes to load and preprocess the data.

            In addition to the compression time, we also set a maximum walltime limit for each experiment. The maximum walltime considers the time it takes to load the data onto the memory, preprocess the data (if necessary), update the approximation, and check the approximation error of the updated cores on the test set. If an experiment exceeds the maximum walltime, the experiment is terminated and the metrics at the time of termination are reported.
        \subsubsection{Relative test error}
        When using the compression for downstream learning tasks, having a latent mapping that achieves a similar approximation error for both in- and out-of-sample datasets is valuable.
        The relative test error (RTE) is a measure of how well the approximation generalizes to unseen data. It is measured using the relative approximation error of tensors averaged over the test set. This performance is expected to get better as each approximation method is presented with more data. In the worst case, we expect the RTE to converge to $\varepsilon_{rel}$ asymptotically in the limit of infinite data. Specifically, the RTE is defined as
        \begin{equation}\label{eq:generalization_performance}
            RTE = \frac{1}{N_{\text{test}}}\sum_{i=1}^{N_{\text{test}}} \frac{\|\mathcal{Y}^{i}_{\text{test},k}-\tilde{\mathcal{Y}}^{i}_{\text{test},k}\|_{F}}{\|\mathcal{Y}^{i}_{\text{test},k}\|_{F}},
        \end{equation}
        where $\mathcal{Y}^{i}_{\text{test},k}$ is the $i$-th tensor in the test set and $\tilde{\mathcal{Y}}^{i}_{\text{test},k}$ is the reconstruction of $\mathcal{Y}^{i}_{\text{test},k}$ using the approximation $\hierarchical_{\mathcal{X}^{k}}$. Note that the RTE is averaged over individual tensors in the test set, which can lead to higher RTE values than the target $\varepsilon_{rel}$ especially when the HT-cores are updated with batches of tensors.

        The RTE is recomputed if the approximation is updated with streamed tensor. If the approximation is updated frequently, computing the RTE becomes a dominant part of the computational cost and often becomes the reason for the experiment to exceed the maximum walltime. 
        
    \subsection{Scientific data results}\label{sec:scientific_data}
        This section includes tests of \ouralgorithm's performance to compress simulation outputs of high-dimensional PDEs.

        One of the main motivations for developing \ouralgorithm is to provide a tool for scientists to analyze large-scale scientific data~\citep{aksoy2022inverse}. In this section we compare the performance of \ouralgorithm~with other incremental tensor decomposition algorithms on scientific data. We will use two different scientific datasets: PDEBench~\citep{takamoto2022pdebench} and a dataset comprised of self-oscillating gel simulations~\citep{alben2019semi}.

        Each snapshot of simulation contains states corresponding to various physical quantities (e.g. displacement, velocity, pressure, density etc.). Therefore, we normalize each physical quantity individually similar to~\citet{aksoy2024compressed}. In this work we consider maximum absolute value~($\mathcal{Y}_{\text{maxabs}}$), unit vector~($\mathcal{Y}_{\text{unitvec}}$), and z-score~($\mathcal{Y}_{\text{z-score}}$) normalizations in addition to compressing the unnormalized simulations. The normalizations are computed as follows:
        \begin{equation}
            \begin{aligned}
                \mathcal{Y}_{\text{maxabs}} = \frac{\mathcal{Y}}{\max(|\mathcal{Y}|)}, \qquad \qquad
                \mathcal{Y}_{\text{unitvec}} = \frac{\mathcal{Y}}{\|\mathcal{Y}\|}, \qquad \qquad
                \mathcal{Y}_{\text{z-score}} = \frac{\mathcal{Y}-\mu}{\sigma},
            \end{aligned}
        \end{equation}
        where $\mathcal{Y}$ is the original tensor, and $\mu$, $\sigma$ are the mean and the standard deviation of the entries of $\mathcal{Y}$. 
        Note that we do not learn a set of normalization parameters that encompasses the entire accumulation. Instead, we treat each simulation in the training set individually and normalize their states separately.

        \begin{figure}[htbp]
            \centering
            \begin{subfigure}{0.4\textwidth}
                \includegraphics[width=\columnwidth]{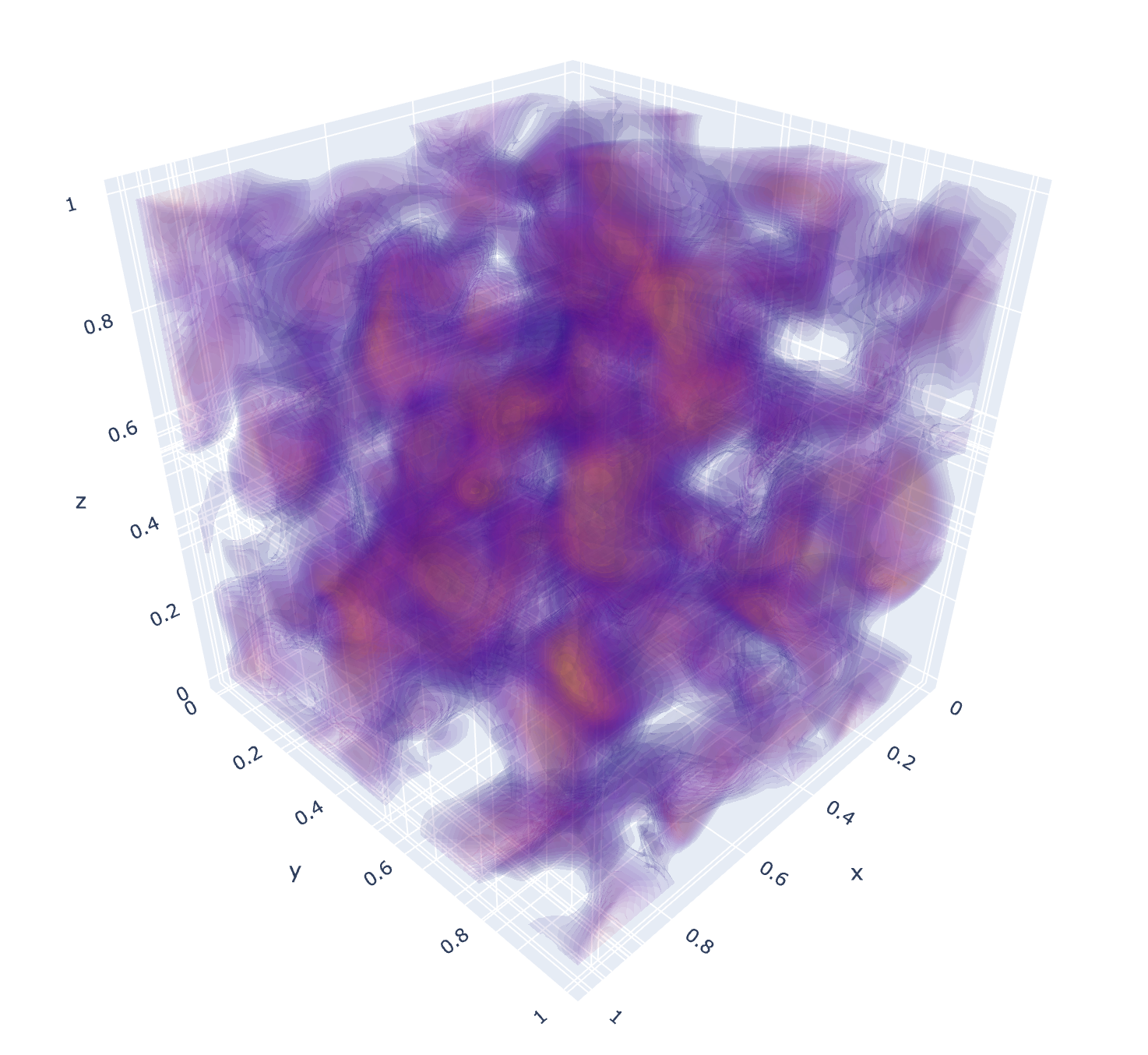}
                \caption{\small\emph{3D compressible Navier-Stokes simulations from PDEBench dataset}}
                \label{fig:example_pdebench}
            \end{subfigure}
            \hspace{0.1\textwidth}
            \begin{subfigure}{0.4\textwidth}
                \includegraphics[width=\columnwidth]{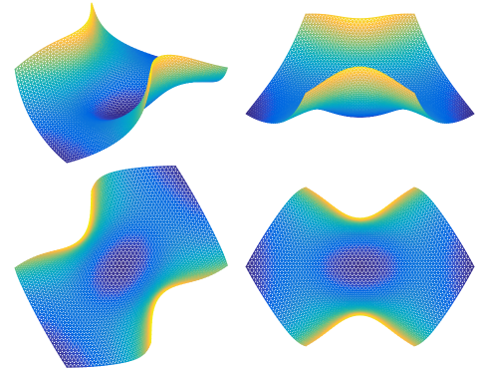}
                \caption{\small\emph{Snapshots from self-oscillating gel simulations}}
                \label{fig:example_gels}
            \end{subfigure}
            \caption{\small\emph{Example snapshots from scientific datasets.}}
            \label{fig:scientific_datasets}
        \end{figure}
    
        \subsubsection{Self-oscillating gel simulations}\label{sec:gel_experiments}
        The first scientific dataset arises from solutions of a parametric PDE that simulates the motion of a hexagonal sheet of self-oscillating gels. This dataset is used both in donwstream learning tasks such as inverse design~\citep{aksoy2022inverse} as well as in experiments of incremental tensor decompositions~\citep{aksoy2024incremental} as a benchmark dataset.

        The motion of the gel is governed by the following time-dependent parametric PDE~\citep{alben2019semi}
        \begin{align}\label{eq:catgelparametricpde}
            \mu\frac{\partial r}{\partial t}=f_{s}\left(r,\mathbf{K_{s}},\eta\right)+f_{B}\left(r\right), && \eta\left(x,y,t,\mathbf{A},\mathbf{k}\right)=1+\mathbf{A}\sin\left(2\pi\left(\mathbf{k}\sqrt{x^{2}+y^{2}}-t\right)\right),
        \end{align}
        where the bold terms indicate the input parameters to the forward model that define the characteristics of the excitation as well as the mechanical properties of the gel. More specifically, $\mathbf{K_s}$ denotes the stretching stiffness of the sheet, $\mathbf{k}$ determines the wavenumber of the sinusoidal excitation, and $\mathbf{A}$ determines the amplitude of the wave traveling on the sheet. Other terms governing the overdamped sheet dynamics are: internal damping coefficient $\mu$, material coordinates $r=(x,y,z)$, stretching force $f_s$, bending force $f_b$, rest strain $\eta$, and time $t$.
        
        The simulations are chaotic, but we use 10 sequential timesteps from each simulation as our data. Specifically, we seek to compress the $x,y$, and $z$ coordinates of 3367 mesh nodes on a hexagonal gel sheet for 10 time snapshots as shown in Figure~\ref{fig:example_gels}.
        To summarize, the data consists of $3367\times3\times10$ tensors for \textit{each} parameter combination that contain the coordinate information of the mesh. We refer to those output tensors as \textit{simulations} for brevity and treat them as individual incremental units. To increase the dimensionality of the data, we reshape the 3-dimensional original tensor into $7\times13\times37\times3\times10$ and accumulate as batches of single simulation trajectories along an auxiliary 6-th dimension.
        
        For this dataset, we have separate training and test sets. To construct the training set, we uniformly discretize the 3-dimensional parameter space into 20 values along each dimension and use their cross product to obtain $20\times20\times20=8000$ unique parameter combinations and then simulate each of those parameter combinations using the approach in~\cite{alben2019semi}. For the test set, randomly sample from the parameter space to obtain 15,000 unique parameter combinations and simulate each of those parameter combinations. We use the training set to update the approximation and the test set to evaluate the generalization performance. Since the selection of the training set is not random, we execute each experiment only once and report the performance. We set the maximum wall time for this experiment to 2 days.
        
        We repeat the experiments for two relative error tolerances: $\varepsilon_{rel}=0.10$, and $\varepsilon_{rel}=0.01$. \Cref{fig:catgel_compression_reduction_simple} shows the best results for compression ratio and reduction ratio from the selected normalization methods for both $\varepsilon_{rel}$ settings. \Cref{fig:catgel_time_validation_simple} shows the best results for compression time and test error for both $\varepsilon_{rel}$ settings. More detailed results including different normalization methods are presented in \Cref{tab:catgel_results} below and \Cref{fig:catgel_compression_reduction_010,fig:catgel_compression_reduction_001,fig:catgel_time_validation_010,fig:catgel_time_validation_001} in \Cref{app:additional_results}.

        \begin{figure}[htbp]
            \centering
            \includegraphics[width=\textwidth]{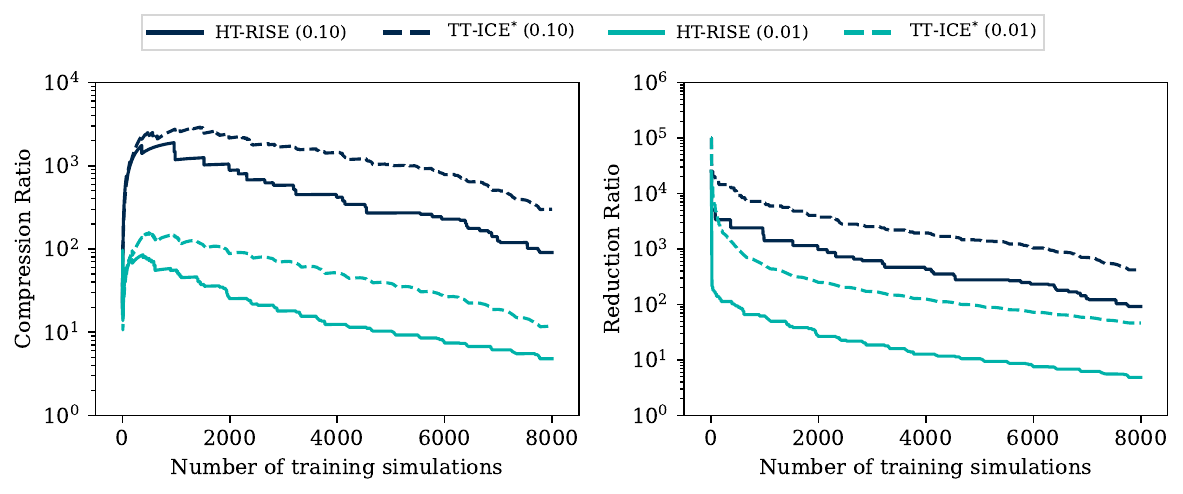}
            \caption{\small\emph{Compression ratio~(CR - left) and reduction ratio~(RR - right) of the algorithms on the self-oscillating gel dataset. \texttt{TT-ICE${}^{*}$} offers $2.5-3.3\times$ the CR and $4.6-9.5\times$ the RR of \ouralgorithm. 
            For more detailed comparisons, please refer to \Cref{fig:catgel_compression_reduction_010} for experiments with $\varepsilon_{rel}=0.10$ and \Cref{fig:catgel_compression_reduction_001} for experiments with $\varepsilon_{rel}=0.01$.}}
            \label{fig:catgel_compression_reduction_simple}
        \end{figure}

        \Cref{fig:catgel_compression_reduction_simple} shows the results of experiments without using any normalization. \ouralgorithm~at $\varepsilon_{rel}=0.10$ achieves a CR of $90.70$ and RR of $91.82$ whereas \texttt{TT-ICE${}^{*}$} achieves a CR of $300.4$ and RR of $420.9$. Similar to $\varepsilon_{rel}=0.10$, \ouralgorithm~at $\varepsilon_{rel}=0.01$ achieves a lower CR ($4.79\times$) compared to \texttt{TT-ICE${}^{*}$} ($11.86\times$) and a lower RR ($4.85\times$) compared to \texttt{TT-ICE${}^{*}$} ($46.27\times$). This can be explained by the relatively small size and the simpler nature of the self-oscillating gel dataset. This allows \texttt{TT-ICE${}^{*}$} to find and exploit a low-rank structure across the tensors in the accumulation. This claim is also supported by the reduction ratio of \texttt{TT-ICE${}^{*}$}, which corresponds to a latent space size of 2183 (recall that the training set has 8000 simulations).

        \begin{figure}[htbp]
            \centering
            \includegraphics[width=\textwidth]{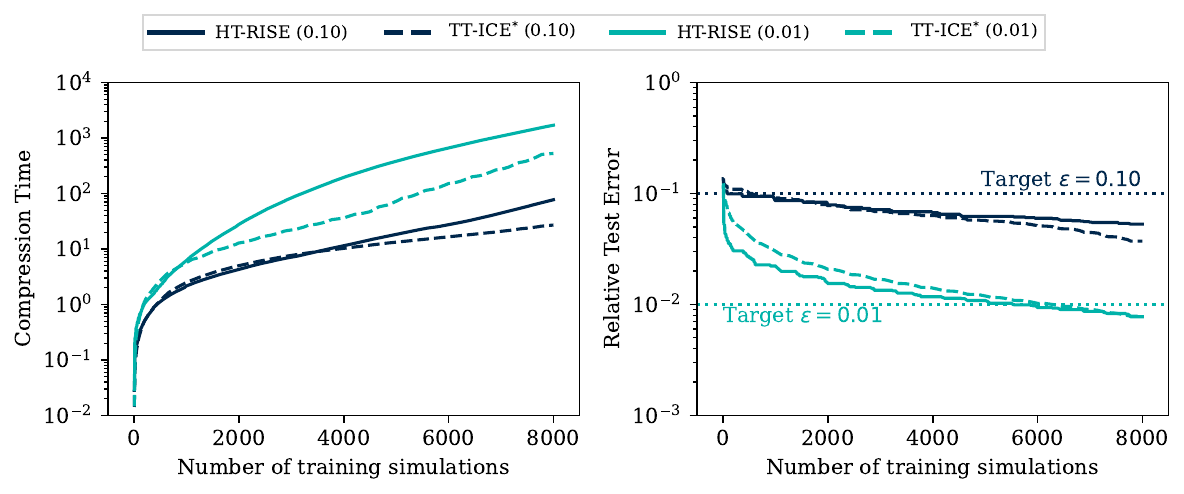}
            \caption{\small\emph{Compression time~(left) and Relative Test Error~(right) of the algorithms on the self-oscillating gel dataset. \ouralgorithm~takes $2.8-3.2\times$ the time it takes \texttt{TT-ICE${}^{*}$} to complete compressing the stream. Neither method struggles to reduce the RTE below the target $\varepsilon$ for both $\varepsilon_{rel}=0.10$ and $0.01$.
            For more detailed comparisons, please refer to \Cref{fig:catgel_time_validation_010} for experiments with $\varepsilon_{rel}=0.10$ and \Cref{fig:catgel_time_validation_001} for experiments with $\varepsilon_{rel}=0.01$.}}
            \label{fig:catgel_time_validation_simple}
        \end{figure}

        \Cref{fig:catgel_time_validation_simple} shows the results of experiments on compression time and relative test error without using any normalization. \ouralgorithm~at $\varepsilon_{rel}=0.10$ takes $77.16$ seconds to compress all 8000 simulations and achieves a RTE of $0.052$. On the other hand, \texttt{TT-ICE${}^{*}$} takes $26.99$ seconds to compress all 8000 simulations and achieves a RTE of $0.037$. \ouralgorithm~at $\varepsilon_{rel}=0.01$ takes $1699.8$ seconds to compress all 8000 simulations and achieves a RTE of $0.007$. On the other hand, \texttt{TT-ICE${}^{*}$} takes $526.9$ seconds to compress all 8000 simulations and achieves a RTE of $0.007$. \ouralgorithm~takes $2.8-3.2\times$ the time it takes \texttt{TT-ICE${}^{*}$} to complete compressing the stream.
        
        Neither method struggles to reduce the RTE below the target $\varepsilon$ for both $\varepsilon_{rel}=0.10$ and $0.01$. This is expected as the self-oscillating gel dataset is relatively simple and has a low-dimensional structure that can be captured by both methods. This is evidenced by both algorithms driving the RTE well below the target $\varepsilon$. The hierarchy of dimensions introduced by the HT format does not provide a significant advantage in terms of approximation performance but becomes a hindrance in terms of computational efficiency. This is reflected as the difference in compression time between \ouralgorithm~and \texttt{TT-ICE${}^{*}$}.

        \Cref{tab:catgel_results} summarizes the rest of our experiments with the self-oscillating gel dataset. The table shows the total time taken to compress the entire dataset, the compression ratio, the reduction ratio, and the mean relative approximation error over the test set. The table also shows the normalization method used for each experiment. The experiments are repeated for two relative error tolerances: $\varepsilon_{rel}=0.10$ and $\varepsilon_{rel}=0.01$. The results further support our claim that \ouralgorithm~is overcomplicated for this dataset and therefore \texttt{TT-ICE${}^{*}$} performs better in terms of our metrics. \texttt{TT-ICE${}^{*}$} consistently outperforms \ouralgorithm~in terms of compression ratio, compression time, and relative test error, except for the experiments at $\varepsilon_{rel}=0.01$ using unit vector and z-score normalizations. At those experiments \texttt{TT-ICE${}^{*}$} runs into maximum walltime issues. 
        
        \begin{table}[h!]
            \centering
            \caption{\small\emph{Summary of the compression experiments with the self-oscillating gel dataset. Norm: method of normalization, Algorithm: the incremental tensor decomposition algorithm, \#Sims: number of simulations compressed, Comp. Time: total time in seconds, CR: compression ratio, RR: reduction ratio, RTE: mean relative test error over the test set. $\ddagger$ indicates that the experiment did not complete due to a timeout, $\dagger$ indicates that the experiment did not complete due to running out of memory.}}
            \label{tab:catgel_results}
            \begin{tabular}{l|c c c c c c c c c c}
                $\varepsilon_{rel}$ & Norm & Algorithm & \#Sims & Comp. Time (s) & CR & RR & RTE \\
                \hline
                \multirow{6}{*}{0.10}& \multirow{2}{*}{None} & HT-RISE & 8000 & 77.16 & 90.70 & 91.82 & 0.052 \\
                &  & TT-ICE${}^{*}$ & 8000 & 26.99 & 300.4 & 420.9 & 0.037 \\
                \cline{2-8}
                & \multirow{2}{*}{UnitVec} & HT-RISE & 8000 & 203.2 & 35.45 & 36.16 & 0.059 \\
                &  & TT-ICE${}^{*}$ & 8000 & 39.29 & 124.9 & 220.1 & 0.059 \\
                \cline{2-8}
                & \multirow{2}{*}{Z-score} & HT-RISE & 8000 & 226.9 & 33.67 & 34.35 & 0.066 \\
                &  & TT-ICE${}^{*}$ & 8000 & 41.83 & 117.6 & 215.4 & 0.058 \\
                \hline
                \hline
                \multirow{6}{*}{0.01}& \multirow{2}{*}{None} & HT-RISE & 8000 & 1699.8 & 4.79 & 4.85 & 0.007 \\
                &  & TT-ICE${}^{*}$ & 8000 & 526.9 & 11.86 & 46.27 & 0.007 \\
                \cline{2-8}
                & \multirow{2}{*}{UnitVec} & HT-RISE & 8000 & 1929.2 & 4.11 & 4.17 & 0.012 \\
                &  & TT-ICE${}^{*}$ & 5100\wt & 119.4 & 27.15 & 70.34 & 0.020 \\
                \cline{2-8}
                & \multirow{2}{*}{Z-score} & HT-RISE & 8000 & 3029.1 & 3.03 & 3.08 & 0.098 \\
                &  & TT-ICE${}^{*}$ & 5086\wt & 177.1 & 16.69 & 67.25 & 0.022 \\
            \end{tabular}
        \end{table}

    \subsubsection{PDEBench 3D Navier-Stokes simulations}\label{sec:pdebench_experiments}

        Next we compare the performance of \ouralgorithm with other incremental tensor decomposition algorithms on a more challenging scientific dataset.
        PDEBench dataset~\citep{takamoto2022pdebench} is a benchmark suite for scientific machine learning tasks. It provides diverse datasets with distinct properties based on 11 well-known time-dependent and time-independent PDEs. 

        From the PDEBench dataset, we use the 3D compressible Navier-Stokes simulations with $M=1.0$ and turbulent initial conditions. The data consists of 3 velocity ($v_{x}, v_{y}, v_{z}$), pressure, and density fields at 21 time steps for 600 different initial conditions. The simulation space is dicretized into a $64\times64\times64$ grid. The compressible Navier stokes equations that are used to generate the dataset are:
        \begin{equation}\label{eq:compressible_cfd}
            \begin{aligned}
                \partial_{t}\rho + \nabla\cdot(\rho\mathbf{v}) &= 0,\\
                \rho\left(\partial_{t}\mathbf{v} + \mathbf{v}\cdot\nabla\mathbf{v}\right) &= -\nabla p + \eta\Delta\mathbf{v} +\left(\zeta+\frac{\eta}{3}\right)\nabla(\nabla\cdot\mathbf{v}),\\
                \partial_{t}(\epsilon+\rho\nu^{2}) + \nabla\cdot\left[(p+\epsilon+\frac{\rho\nu^{2}}{2})\mathbf{v} - \mathbf{v}\cdot\sigma^{\prime}\right] &= \mathbf{0},
            \end{aligned}
        \end{equation}
        where $\rho$ is the mass density, $\mathbf{v}$ is the fluid velocity, $p$ is the gas pressure, $\epsilon$ is the internal energy, $\sigma^{\prime}$ is the viscous stress tensor, and $\eta$ and $\zeta$ are shear and bulk viscosities, respectively. Each simulation from the dataset creates a 5-dimensional tensor of size $64\times64\times64\times5\times21$. This results in simulations that are $\sim270\times$ larger than the self-oscillating gel simulations.
        
        To increase the dimensionality of the dataset, we reshape each simulation into a $8$-dimensional tensor of shape $8\times 8 \times8 \times 8 \times 8\times 8 \times 5 \times 21$ and accumulate in batches of single simulation trajectory along an auxiliary $9$-th dimension.

        We randomly split 600 simulations into training and test sets with a ratio of $80\%$ and $20\%$, respectively. We use the training set to update the approximation and the test set to evaluate the generalization performance. Since the selection of the training set is random, we repeat each experiment with five different seeds and report the average performance. We set the maximum wall time for each experiment to four days.
        We repeat the experiments for two relative error tolerances: $\varepsilon_{rel}=0.10$, and $\varepsilon_{rel}=0.05$. \Cref{fig:pdebench_compression_reduction_simple} shows the best results for compression ratio and reduction ratio from the selected normalization methods for both $\varepsilon_{rel}$ settings. \Cref{fig:pdebench_time_validation_simple} shows the best results for compression time and test error for both $\varepsilon_{rel}$ settings. More detailed results are presented in \Cref{tab:pdebench_results} below and \Cref{fig:pdebench_compression_reduction_005,fig:pdebench_compression_reduction_010,fig:pdebench_time_validation_005,fig:pdebench_time_validation_010} in \Cref{app:additional_results}.

        \begin{figure}[h!]
            \centering
            \includegraphics[width=\textwidth]{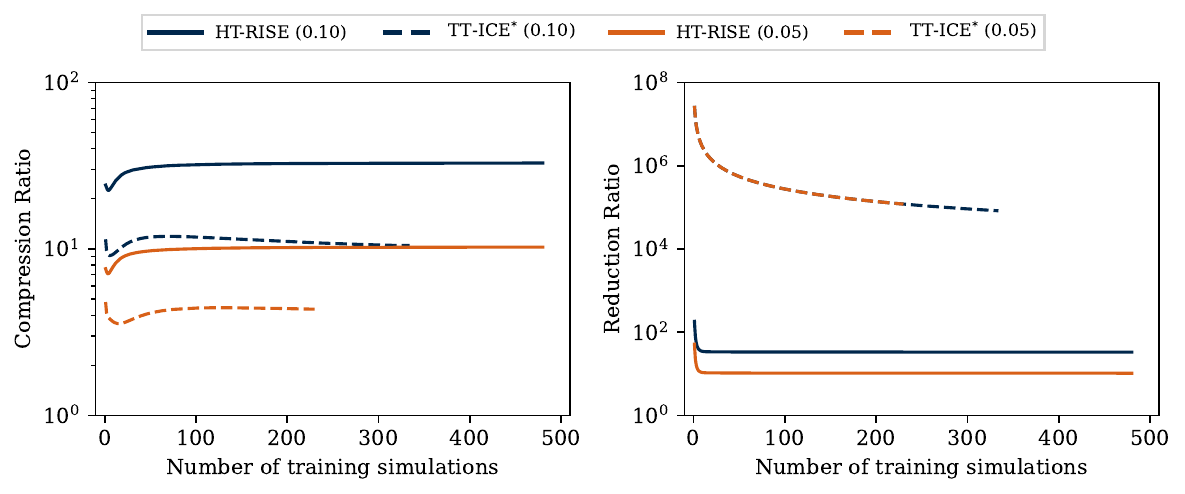}
            \caption{\small\emph{Compression ratio~(CR - left) and reduction ratio~(RR- - right) of the algorithms on the PDEBench 3D turbulent Navier-Stokes dataset. \ouralgorithm~offers $2.4-3.1\times$ the CR of \texttt{TT-ICE${}^{*}$} but results in orders of magnitude lower RR. \texttt{TT-ICE${}^{*}$} does not complete the entire stream due to maximum walltime timeout whereas \ouralgorithm~successfully completes the task. The results are averaged over 5 seeds. For more detailed comparisons, please refer to \Cref{fig:pdebench_compression_reduction_010} for experiments with $\varepsilon_{rel}=0.10$ and \Cref{fig:pdebench_compression_reduction_005} for experiments with $\varepsilon_{rel}=0.05$.}}
            \label{fig:pdebench_compression_reduction_simple}
        \end{figure}

        \Cref{fig:pdebench_compression_reduction_simple} shows the results of experiments without using any normalization. \ouralgorithm~at $\varepsilon_{rel}=0.10$ achieves a CR of $32.83$ and RR of $33.09$ whereas \texttt{TT-ICE${}^{*}$} achieves a CR of $10.49$ and RR of $82,658$. Note that \texttt{TT-ICE${}^{*}$} runs into maximum walltime issues at this experiment and is only able to compress 333 simulations out of 480.
        Similar to $\varepsilon_{rel}=0.10$, \ouralgorithm~at $\varepsilon_{rel}=0.05$ achieves a higher CR ($10.2\times$) compared to \texttt{TT-ICE${}^{*}$} ($4.34\times$) but a lower RR ($10.34\times$) compared to \texttt{TT-ICE${}^{*}$} ($119,674\times$). This time \texttt{TT-ICE${}^{*}$} runs into maximum walltime issues and is only able to compress 229 simulations out of 480.
        The orders of magnitude discrepancy between the RR of \texttt{TT-ICE${}^{*}$} and \ouralgorithm~is caused by the fact that the size of the latent space representation is upper bounded by the number of tensors in the accumulation for \texttt{TT-ICE${}^{*}$}. As \ouralgorithm~does not have such a limit, the latent space grows in parallel to the complexity of the streamed data. \texttt{TT-ICE${}^{*}$} hitting the latent space upper bound further indicates that \texttt{TT-ICE${}^{*}$} is not able to find a low-rank structure \textit{across} the tensors in the accumulation in contrast to self-oscillating gels dataset.

        \begin{figure}[h!]
            \centering
            \includegraphics[width=\textwidth]{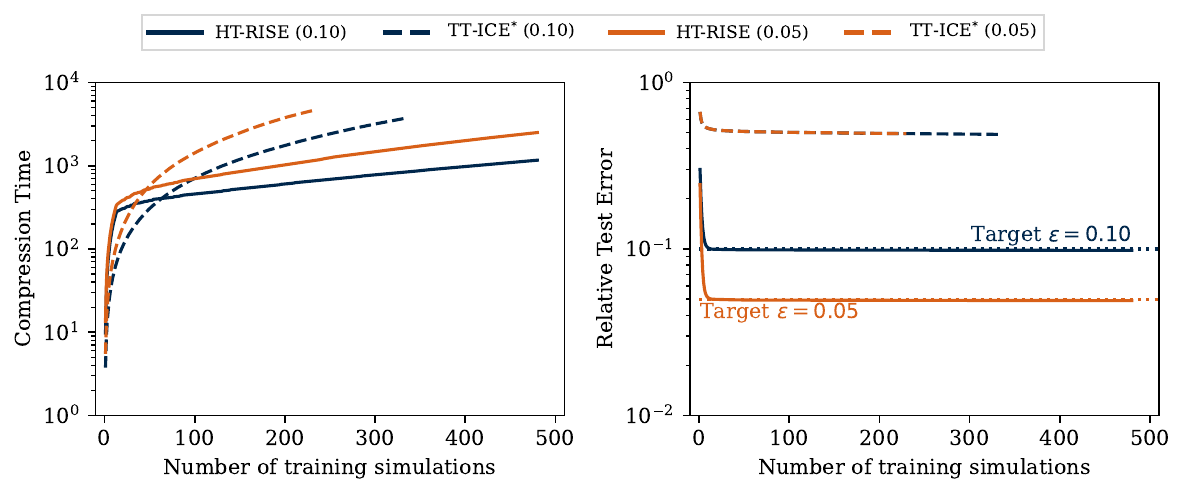}
            \caption{\small\emph{Compression time~(left) and Relative Test Error~(right) of the algorithms on the PDEBench 3D turbulent Navier-Stokes dataset. Until \texttt{TT-ICE${}^{*}$} hits the maximum walltime limit, \texttt{TT-ICE${}^{*}$} takes $1.8-3.2\times$ of the time it takes \ouralgorithm~to complete compressing the stream. In addition, \texttt{TT-ICE${}^{*}$} struggles to reduce the RTE below the target $\varepsilon$ levels in both cases whereas \ouralgorithm~reduces the RTE below the $\varepsilon$ within 15 simulations. Recall that \texttt{TT-ICE${}^{*}$} does not complete the entire stream due to maximum walltime timeout whereas \ouralgorithm~successfully completes the task. The results are averaged over 5 seeds. For more detailed comparisons, please refer to \Cref{fig:pdebench_time_validation_010} for experiments with $\varepsilon_{rel}=0.10$ and \Cref{fig:pdebench_time_validation_005} for experiments with $\varepsilon_{rel}=0.05$.}}
            \label{fig:pdebench_time_validation_simple}
        \end{figure}

        \Cref{fig:pdebench_time_validation_simple} shows the results of experiments on compression time and relative test error without using any normalization. \ouralgorithm~at $\varepsilon_{rel}=0.10$ takes $1165.90$ seconds to compress all 480 simulations and achieves a RTE of $0.098$. On the other hand, \texttt{TT-ICE${}^{*}$} takes $3727.38$ seconds to compress 333 simulations and achieves a RTE of $0.487$. \ouralgorithm~at $\varepsilon_{rel}=0.05$ takes $2505.58$ seconds to compress all 480 simulations and achieves a RTE of $0.049$. On the other hand, \texttt{TT-ICE${}^{*}$} takes $4581.78$ seconds to compress 229 simulations and achieves a relative test error of $0.493$. \texttt{TT-ICE${}^{*}$} struggles to reduce the RTE below the target $\varepsilon$ levels in both cases whereas \ouralgorithm~reduces the relative test error below the $\varepsilon$ within the first 15 simulations. 
        This observation further supports the claim that \texttt{TT-ICE${}^{*}$} is not able to find a generalized low-rank structure across the tensors in the accumulation.
        As a result of its hierarchical structure, \ouralgorithm~discovers a low-rank structure across the tensors in the accumulation and is able to generalize well to unseen data.
    As a result, \ouralgorithm~requires much fewer updates to the approximation to achieve the target $\varepsilon$ than \texttt{TT-ICE${}^{*}$}. This is also reflected in the compression time as both \ouralgorithm~curves taper off as soon as the RTE falls below their respective $\varepsilon$, whereas curves for \texttt{TT-ICE${}^{*}$} continue to increase steadily. Frequent updates of \texttt{TT-ICE${}^{*}$} is also the reason why \texttt{TT-ICE${}^{*}$} runs into maximum walltime issues, as each update is succeeded with RTE computation over the \textit{entire} test set.

        \Cref{tab:pdebench_results} summarizes the rest of our experiments with the PDEBench 3D turbulent Navier-Stokes dataset. The table shows the total time taken to compress the entire dataset, the compression ratio, the reduction ratio, and the RTE. The table also shows the normalization method used for each experiment. The experiments are repeated for two relative error tolerances: $\varepsilon_{rel}=0.10$ and $\varepsilon_{rel}=0.05$. The results show that \ouralgorithm~consistently outperforms \texttt{TT-ICE${}^{*}$} in terms of compression ratio, compression time, and RTE. \ouralgorithm~also completes the task within the maximum walltime limit (4 days) in all experiments (except for Z-score normalization at $\varepsilon_{rel}=0.05$, where it runs into memory issues), whereas \texttt{TT-ICE${}^{*}$} runs into maximum walltime issues in all experiments.

        \begin{table}[h!]
            \small
            \centering
            \caption{\small\emph{Summary of the compression experiments with the PDEBench 3D turbulent Navier-Stokes dataset. The results are averaged over 5 seeds. Norm: method of normalization, Algorithm: the incremental tensor decomposition algorithm, \#Sims: number of simulations compressed, Comp. Time: total time in seconds, CR: compression ratio, RR: reduction ratio, RTE: mean relative test error over the test set.
            Please refer to \Cref{fig:pdebench_compression_reduction_010,fig:pdebench_time_validation_010} for experiments with $\varepsilon_{rel}=0.10$ and \Cref{fig:pdebench_compression_reduction_005,fig:pdebench_time_validation_005} for experiments with $\varepsilon_{rel}=0.05$. MaxAbs: Maximum absolute value normalization, None: No normalization, UnitVec: Unit vector normalization, Z-score: Z-score normalization. $\ddagger$ indicates that the experiment did not complete due to a timeout, $\dagger$ indicates that the experiment did not complete due to running out of memory.}}
            \label{tab:pdebench_results}
            \begin{tabular}{l|c c c c c c c c c c}
                $\varepsilon_{rel}$ & Norm & Algorithm & \#Sims & Comp. Time (s) & CR & RR & RTE \\
                \hline
                \multirow{8}{*}{0.10}& \multirow{2}{*}{MaxAbs} & HT-RISE & 480 & 1706.84 & 19.73 & 19.89 & 0.087 \\
                &  & TT-ICE${}^{*}$ & 292\wt & 3824.33 & 7.94 & 94,264 & 0.609 \\
                \cline{2-8}
                & \multirow{2}{*}{None} & HT-RISE & 480 & 1165.90 & 32.83 & 33.09 & 0.098 \\
                &  & TT-ICE${}^{*}$ & 333\wt & 3727.38 & 10.49 & 82,658 & 0.487 \\
                \cline{2-8}
                & \multirow{2}{*}{UnitVec} & HT-RISE & 480 & 1863.73 & 16.80 & 16.92 & 0.096 \\
                &  & TT-ICE${}^{*}$ & 267\wt & 3989.26 & 6.72 & 102,705 & 0.762 \\
                \cline{2-8}
                & \multirow{2}{*}{Z-score} & HT-RISE & 480 & 3007.92 & 9.08 & 9.20 & 0.097 \\
                &  & TT-ICE${}^{*}$ & 213\wt & 4287.14 & 3.97 & 128,622 & 0.962 \\
                \hline
                \hline
                \multirow{8}{*}{0.05}& \multirow{2}{*}{MaxAbs} & HT-RISE & 480 & 3594.18 & 7.08 & 7.18 & 0.044 \\
                &  & TT-ICE${}^{*}$ & 194\wt & 4462.31 & 3.53 & 141,154 & 0.617 \\
                \cline{2-8}
                & \multirow{2}{*}{None} & HT-RISE & 480 & 2505.58 & 10.26 & 10.34 & 0.049 \\
                &  & TT-ICE${}^{*}$ & 229\wt & 4581,78 & 4.34 & 119,674 & 0.493 \\
                \cline{2-8}
                & \multirow{2}{*}{UnitVec} & HT-RISE & 480 & 3911.02 & 6.26 & 6.30 & 0.048 \\
                &  & TT-ICE${}^{*}$ & 177\wt & 4503.89 & 2.93 & 154,635 & 0.771 \\
                \cline{2-8}
                & \multirow{2}{*}{Z-score} & HT-RISE & 477\oom & 7122.89 & 3.18 & 3.19 & 0.048 \\
                &  & TT-ICE${}^{*}$ & 144\wt & 4877.19 & 1.82 & 189,828 & 0.970 \\
            \end{tabular}
        \end{table}

      \paragraph{\textbf{Summmary of conclusions from experiments with scientific data:}}
        The experiments with scientific data demonstrate that \ouralgorithm~provides superior generalization performance compared to \texttt{TT-ICE${}^{*}$} across both datasets. \ouralgorithm~consistently reduces the RTE below the target $\varepsilon$ levels within fewer simulations. \texttt{TT-ICE${}^{*}$} results in significantly higher RR due to the properties of the TT-format across both datasets but falls short in CR when the dataset becomes large. The importance of normalization for scientific data is also demonstrated in \Cref{tab:catgel_results,tab:pdebench_results}, showing that it can significantly affect the performance of the compression algorithms. Overall, the results suggest that \ouralgorithm~is more suitable for scientific datasets with complex multi-scale structures, while \texttt{TT-ICE${}^{*}$} may be more appropriate for simpler datasets where the low-rank structure can be effectively captured. 
        
    \subsection{Image data}\label{sec:image_data}
        In addition to the scientific data, we will also compare the algorithms on image data. The image data will be obtained from two different sources: The MineRL Basalt competition dataset~\citep{milani2024bedd} and multispectral images from the BigEarthNet dataset~\citep{sumbul2019bigearthnet,sumbul2021bigearthnet}. 
        \begin{figure}[htbp]
            \centering
            \begin{subfigure}{0.41\textwidth}\centering
                \includegraphics[width=0.88\columnwidth]{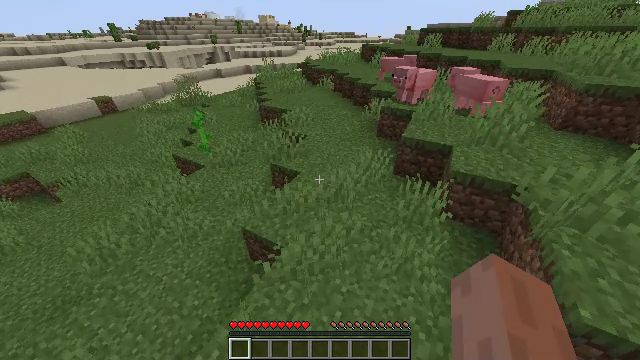}
                \includegraphics[width=0.88\columnwidth]{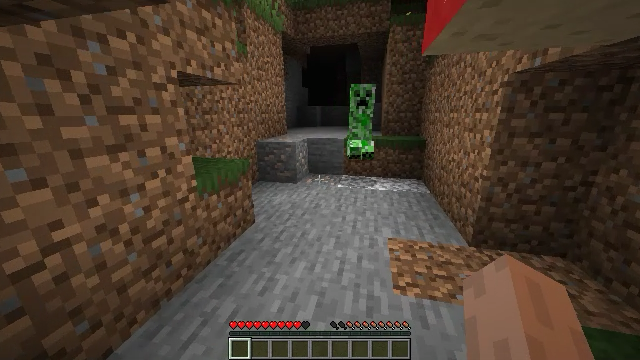}
                \caption{\small\emph{MineRL Basalt competition dataset}}
                \label{fig:example_minecraft}
            \end{subfigure}
            \begin{subfigure}{0.42\textwidth}
                    \includegraphics[width=0.32\columnwidth]{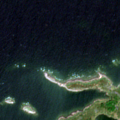}
                    \includegraphics[width=0.32\columnwidth]{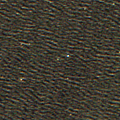}
                    \includegraphics[width=0.32\columnwidth]{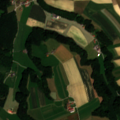}
                    \includegraphics[width=0.32\columnwidth]{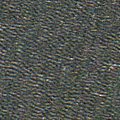}
                    \includegraphics[width=0.32\columnwidth]{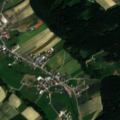}
                    \includegraphics[width=0.32\columnwidth]{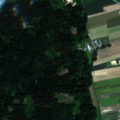}
                    \includegraphics[width=0.32\columnwidth]{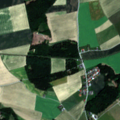}
                    \includegraphics[width=0.32\columnwidth]{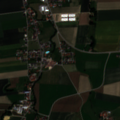}
                    \includegraphics[width=0.32\columnwidth]{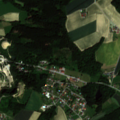}
                \caption{\small\emph{Sentinel2 images~(RGB channels only)}}
                \label{fig:example_multispectral}
            \end{subfigure}
            \caption{\small\emph{Example frames from image-based datasets.}}
            \label{fig:image_datasets}
        \end{figure}

    \subsubsection{Minecraft frames}\label{sec:minecraft_frames}
    First person games are of interest for inverse reinforcement learning purposes. The MineRL Basalt competition dataset~\citep{milani2024bedd} contains gameplay episodes from the game Minecraft and is designed for the purpose of training agents to perform tasks in the game. The dataset contains gameplay episodes from 4 different tasks, which are designed to be challenging for reinforcement learning agents. 
        Each episode is an end-to-end gameplay video that starts at the corresponding initial setting for each scenario and ends with successful completion of the assignment. Therefore, each episode varies in duration. The original videos have a resolution of $640\times360$ pixels and a frame rate of 20 frames per second.
        
        The MineRL dataset contains demonstrations for the following 4 tasks:
        \begin{itemize}
            \item \textbf{Find Cave:} Look around for a cave. When you are inside one, press ESCAPE to end the minigame.
            \item \textbf{Make Waterfall:} After spawning in a mountainous area with a water bucket and various tools, the agent builds a waterfall and then repositions themself to “take a scenic picture” of the same waterfall.
            \item \textbf{Build Animal Pen:} After spawning in a village, the agent builds an animal pen next to one of the houses in a village using fence posts from inventory to build one animal pen that contains at least two of the same animal.
            \item \textbf{Build Village House:} Taking advantage of the items in the inventory, the agent builds a new house in the style of the village (random biome), in an appropriate location (e.g. next to the path through the village), without harming the village in the process. Then gives a brief tour of the house (i.e. spin around slowly such that all of the walls and the roof are visible).
        \end{itemize}
        The dataset contains approximately 20,000 gameplay episodes in total. This translates to approximately 160 GB of gameplay epsiodes \textit{per task}. In this study we only focus on the \textit{Find Cave} task. The \textit{Find Cave} task has 5466 gameplay episodes, totaling a size of 183GB on disk. Similar to~\cite{baker2022video}, we downsample the $640\times360$ frames to $128\times128$. This results in a 4-dimensional dataset $128\times128\times3\times N_{frames}$, where $N_{frames}$ is the number of frames in a video. Furthermore, we increase the dimensionality of the dataset by reshaping the original frames into 9-dimensional tensor of size $2\times 4\times 4\times 4\times 8\times 8 \times 2 \times 3 \times N_{frames}$ and accumulate them as batches of 20 along the 9-th dimension, i.e. $N_{frames}=20$.
        
        We select 17 videos from the dataset as test set at random and use the rest for training. Since the size of the dataset exceeds the allocated memory for each experiment by almost a factor of three, we do not expect the algorithms to compress all 5466 videos. This experiment aims to push the algorithms to their limits given limited computational and hardware resources. We repeat the experiments for three relative error tolerances: $\varepsilon_{rel}=0.30$, $\varepsilon_{rel}=0.20$, and $\varepsilon_{rel}=0.10$. The maximum wall time for this experiment is set to two days. Since the selection of the test set is random, we execute each experiment five times and report the average performance. The results are summarized in \Cref{tab:minecraft_results} and are presented in detail in \Cref{fig:minerl_compression_reduction,fig:minerl_time_validation}.

        \begin{figure}[htbp]
            \centering
            \includegraphics[width=\textwidth]{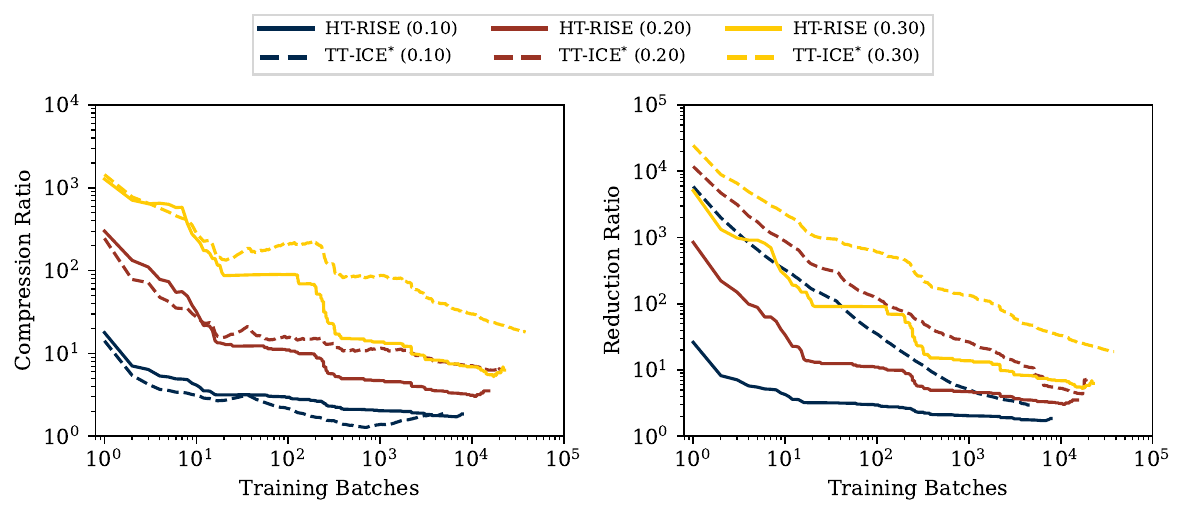}
            \caption{\small\emph{Compression ratio~(CR - left) and reduction ratio~(RR - right) of the algorithms on the Basalt MineRL dataset.\texttt{TT-ICE${}^{*}$} offers $1-2.9\times$ the CR and $1.6-3\times$ the RR of \ouralgorithm. The results are averaged over 5 seeds.}}
            \label{fig:minerl_compression_reduction}
        \end{figure}

        \Cref{fig:minerl_compression_reduction} shows the results of experiments on compression ratio and reduction ratio for the MineRL Basalt competition dataset. \ouralgorithm~at $\varepsilon_{rel}=0.30$ achieves a CR of $6.29$ and RR of $6.29$ whereas \texttt{TT-ICE${}^{*}$} achieves a CR of $18.3$ and RR of $18.9$. Similar to $\varepsilon_{rel}=0.30$, \ouralgorithm~at $\varepsilon_{rel}=0.20$ achieves a lower CR ($3.53\times$) compared to \texttt{TT-ICE${}^{*}$} ($6.45\times$) and a lower RR ($3.53\times$) compared to \texttt{TT-ICE${}^{*}$} ($7.13\times$). This trend continues for $\varepsilon_{rel}=0.10$ where \ouralgorithm~achieves a lower CR ($1.85\times$) compared to \texttt{TT-ICE${}^{*}$} ($1.88\times$) and a lower RR ($1.85\times$) compared to \texttt{TT-ICE${}^{*}$} ($2.96\times$).

        \begin{figure}[htbp]
            \centering
            \includegraphics[width=\textwidth]{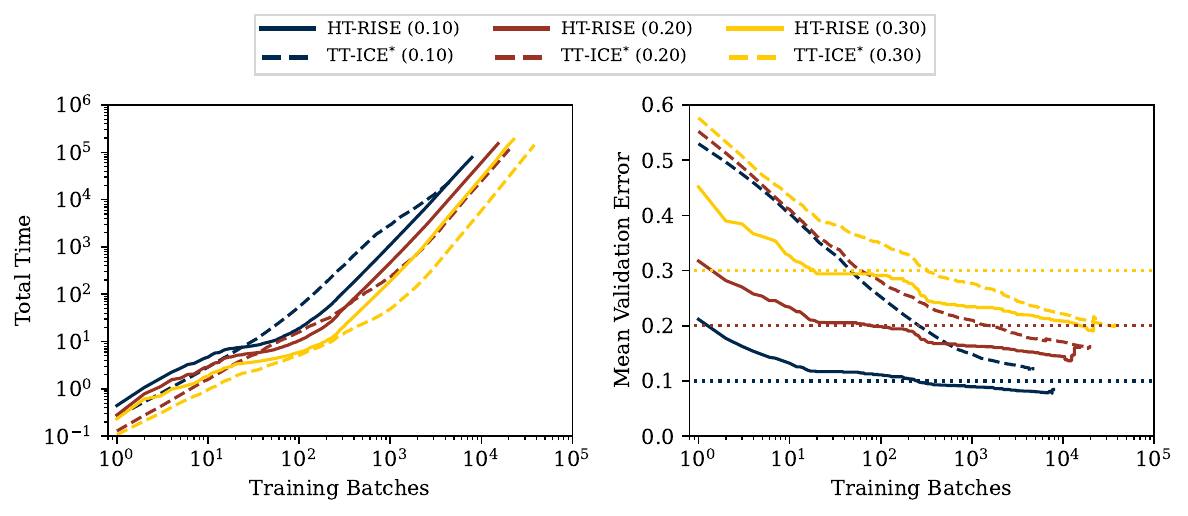}
            \caption{\small\emph{Compression time~(left) and Relative Test Error~(right) of the algorithms on the Basalt MineRL competition dataset. Each target $\varepsilon$ level is shown with dashed lines of their corresponding color. \ouralgorithm~takes $2.8-3.2\times$ the time it takes \texttt{TT-ICE${}^{*}$} to complete compressing the stream. Only \texttt{TT-ICE${}^{*}$} at $\varepsilon_{rel}=0.1$ struggles to reduce the RTE below the target $\varepsilon$ threshold. The results are averaged over 5 seeds.}}
            \label{fig:minerl_time_validation}
        \end{figure}

        \Cref{fig:minerl_time_validation} shows the results of experiments on compression time and relative test error for the MineRL Basalt competition dataset. \ouralgorithm~at $\varepsilon_{rel}=0.30$ takes $188,405$ seconds to compress 22,625 batches (449k frames - 375 videos) and achieves a RTE of $0.208$. On the other hand, \texttt{TT-ICE${}^{*}$} takes $143,873$ seconds to compress 38,146 batches (757k frames - 635 videos) and achieves a RTE of $0.202$. \ouralgorithm~at $\varepsilon_{rel}=0.20$ takes $153,146$ seconds to compress 15,222 batches (302k frames - 252 videos) and achieves a RTE of $0.166$. On the other hand, \texttt{TT-ICE${}^{*}$} takes $116,418$ seconds to compress 20,350 batches (404k frames - 339 videos) and achieves a RTE of $0.161$. \ouralgorithm~at $\varepsilon_{rel}=0.10$ takes $78,286$ seconds to compress 7863 batches (156k frames - 127 videos) and achieves a RTE of $0.083$. On the other hand, \texttt{TT-ICE${}^{*}$} takes $26,756$ seconds to compress 4861 batches (96k frames - 80 videos) and achieves a RTE of $0.119$. Only \texttt{TT-ICE${}^{*}$} at $\varepsilon_{rel}=0.10$ struggles to reduce the RTE below the target $\varepsilon$ threshold. The trend in \Cref{fig:minerl_time_validation} suggests that this issue pertains to the lack of sufficient training data for \texttt{TT-ICE${}^{*}$} to learn a generalizable approximation and should be resolved once the algorithm is presented with further gameplay episodes. Another important takeaway from \Cref{fig:minerl_time_validation} is that for $\varepsilon_{rel}=0.20$ and $\varepsilon_{rel}=0.30$, \ouralgorithm~reduces the RTE below the target $\varepsilon$ threshold within significantly less training batches compared to \texttt{TT-ICE${}^{*}$}. \ouralgorithm~crosses the RTE threshold in 82 and 19 batches, whereas \texttt{TT-ICE${}^{*}$} crosses the RTE threshold in 1597 and 322 batches for $\varepsilon_{rel}=0.20$ and $\varepsilon_{rel}=0.30$, respectively. This suggests that \ouralgorithm~generalizes better with less training data compared to \texttt{TT-ICE${}^{*}$} regardless of the target $\varepsilon$ level.

        One important thing to note from \Cref{tab:minecraft_results} is that except for $\varepsilon_{rel}=0.10$, \texttt{TT-ICE${}^{*}$} compresses more frames than \ouralgorithm. Furthermore, in contrast to \ouralgorithm, \texttt{TT-ICE${}^{*}$} runs into walltime issues rather than memory issues for all target $\varepsilon$ levels. This suggests that \texttt{TT-ICE${}^{*}$} is more memory efficient than \ouralgorithm. This is no surprise as the number of tensors in the accumulation goes to extreme levels (+100k), the RR dominates the overall cost of storing the compressed representations in memory. Since \texttt{TT-ICE${}^{*}$} has a 1-dimensional (i.e. a vector) latent space as opposed to the 2-dimensional (i.e. a matrix) latent space of \ouralgorithm, the growth of the latent space is slower for \texttt{TT-ICE${}^{*}$} compared to \ouralgorithm.
        
        \begin{table}[htbp]
            \centering
            \caption{\small\emph{Summary of the compression experiments with the MineRL Basalt competition dataset. The results are averaged over 5 seeds. 
            Algorithm: the incremental tensor decomposition algorithm, \#Batches: number of batches compressed, Comp. Time: total time in seconds, CR: compression ratio, RR: reduction ratio, RTE: mean relative test error over the test set. $\ddagger$ indicates that the experiment did not complete due to a timeout, $\dagger$ indicates that the experiment did not complete due to running out of memory.}}
            \label{tab:minecraft_results}
            \begin{tabular}{c | c c c c c c}
                $\varepsilon_{rel}$ & Algorithm & \#Batches & Comp. Time & CR & RR & RTE \\
                \hline
                \multirow{2}{*}{0.30}& HT-RISE & 22,625\oom\wt & 188,405 & 6.29 & 6.29 & 0.208\\
                & TT-ICE${}^{*}$ & 38,146\wt & 143,873 & 18.3 & 18.9 & 0.202 \\
                \hline
                \hline
                \multirow{2}{*}{0.20}& HT-RISE & 15,222\oom & 153,146 & 3.53 & 3.53 & 0.166\\
                & TT-ICE${}^{*}$ & 20,350\wt & 116,418 & 6.45 & 7.13 & 0.161 \\
                \hline
                \hline
                \multirow{2}{*}{0.10}& HT-RISE & 7863\oom & 78,286 & 1.85 & 1.85 & 0.083\\
                 & TT-ICE${}^{*}$ & 4861\wt & 26,756 & 1.88 & 2.96 & 0.119 \\
                
            \end{tabular}
        \end{table}

    \subsubsection{Multispectral images}\label{sec:multispectral_frames}
        Another natural occurence of high-dimensional data in the form of images is multispectral satellite imagery. The BigEarthNet dataset~\citep{sumbul2021bigearthnet,sumbul2019bigearthnet} consists of 590,326 image patches, each of which is a 120x120x12 pixel multispectral image. The dataset has a size of 66GB.
        The images are taken from the Sentinel-2 satellite and contain 12 spectral bands. \Cref{tab:sentinel2_bands} lists the spectral bands and their wavelengths.
        The images are taken from 125 different locations around the world. The dataset is originally designed for land cover classification and land use analysis. 
        Therefore, the images are labeled with 43 different land cover classes.
        As we do not aim to showcase the use cases of the latent space learned through incremental tensor decompositions in this work, we do not consider the labels of the images in any part of the compression experiments.
        
        \begin{table}[htbp]
            \centering
            \caption{\small\emph{Names and wavelengths of spectral bands of Sentinel2 images. SWIR: Short Wave Infrared, NIR: Near Infrared. The bands are ordered in increasing wavelength.}}
            \label{tab:sentinel2_bands}
            \begin{tabular}{c | c | c | c}
                \multirow[c]{2}{*}{Band Name} & Central & \multirow[c]{2}{*}{Bandwidth (mm)} & \multirow[c]{2}{*}{Description}\\
                & Wavelength (nm) & & \\
                \hline
                B01 & 442.7 & 21 & Coastal aerosol\\
                B02 & 492.4 & 66 & Blue\\
                B03 & 559.8 & 36 & Green\\
                B04 & 664.6 & 31 & Red\\
                B05 & 704.1 & 15 & Vegetation Red Edge 1\\
                B06 & 740.5 & 15 & Vegetation Red Edge 2\\
                B07 & 782.8 & 20 & Vegetation Red Edge 3\\
                B08 & 832.8 & 106 & NIR\\
                B08A & 864.7 & 21 & Narrow NIR\\
                B09 & 945.1 & 20 & Water Vapor\\
                B11 & 1613.7 & 91 & SWIR 1\\
                B12 & 2202.4 & 175 & SWIR 2\\
            \end{tabular}
        \end{table}

        From the 590,326 images, we select 23,602 images at random as the test set and use the rest for training. We repeat the experiments for two relative error tolerances: $\varepsilon_{rel}=0.30$ and $\varepsilon_{rel}=0.10$. The maximum wall time for this experiment is set to 2 days. Since the selection of the test set is random, we execute each experiment 5 times and report the average performance. The results are summarized in \Cref{tab:bigearthnet_results} and are presented in detail in \Cref{fig:bigearthnet_compression_reduction,fig:bigearthnet_time_validation}.
        To increase the dimensionality of the dataset, we reshape the original multispectral images into 5-dimensional tensors of size $12\times 10\times 12\times 10\times 12$ and accumulate them as batches of 100 along an auxiliary 6-th dimension.

        \begin{table}[htbp]
            \centering
            \caption{\small\emph{Summary of the compression experiments with the BigEarthNet multispectral satellite imagery dataset. The results are averaged over 5 seeds. 
            Algorithm: the incremental tensor decomposition algorithm, \#Batches: number of batches compressed, Comp. Time: total time in seconds, CR: compression ratio, RR: reduction ratio, RTE: mean relative test error over the test set. $\ddagger$ indicates that the experiment did not complete due to a timeout, $\dagger$ indicates that the experiment did not complete due to running out of memory.}}
            \label{tab:bigearthnet_results}
            \begin{tabular}{l | c c c c c c}
                $\varepsilon_{rel}$ & Algorithm & \#Batches & Comp. Time & CR & RR & RTE \\
                \hline
                \multirow{2}{*}{0.30}   & HT-RISE & 5668 & 2482.3 & 1154.3 & 962.04 & 0.242\\
                                        & TT-ICE${}^{*}$ & 5668 & 6239.6 & 2274.0 & 1901.1 & 0.260 \\
                \hline
                \hline
                \multirow{2}{*}{0.15}   & HT-RISE & 5668 & 14,896 & 91.45 & 76.21 & 0.149\\
                                        & TT-ICE${}^{*}$ & 5668 & 18,866 & 117.9 & 100.8 & 0.133 \\
                \hline
                \hline
                \multirow{2}{*}{0.10}   & HT-RISE & 5668 & 49,408 & 32.68 & 27.24 & 0.104\\
                                        & TT-ICE${}^{*}$ & 5668\wt & 44,687 & 35.31 & 31.03 & 0.088 \\
                \hline
                \hline
                \multirow{2}{*}{0.05}   & HT-RISE & 1781\oom & 20,880 & 7.29 & 6.07 & 0.054\\
                                        & TT-ICE${}^{*}$ & 129\wt & 1720.0 & 3.59 & 19.76 & 0.069 \\
            \end{tabular}
        \end{table}

        \Cref{tab:bigearthnet_results} shows that both algorithms are able to completely compress the entire dataset, which exceeds the allocated memory for the experiments, without running into any memory issues for target relative error thresholds $\varepsilon_{rel}=0.30-0.10$. Three out of five experiments for \texttt{TT-ICE${}^{*}$} at $\varepsilon_{rel}=0.10$ and five out of five experiments at $\varepsilon_{rel}=0.05$ did not complete due to a timeout. \ouralgorithm~at $\varepsilon_{rel}=0.05$ runs out of the allocated memory after compressing 1781 batches ($\sim$178,000 images). 

        For higher relative error thresholds, \texttt{TT-ICE${}^{*}$} achieves higher CR and RR than \ouralgorithm~but takes a longer time. At $\varepsilon_{rel}=0.30$, \ouralgorithm~achieves a CR of $1154.3\times$ and a RR of $962.04\times$ in $2482$ seconds, whereas \texttt{TT-ICE${}^{*}$} surpasses \ouralgorithm~with a CR of $2274.0\times$ and a RR of $1901.1\times$ in $6239$ seconds. The results at $\varepsilon_{rel}=0.15$ show that \texttt{TT-ICE${}^{*}$} compresses the training dataset in $18,866$ seconds, which is $1.27\times$ the time it takes \ouralgorithm~to compress the same data. However, \texttt{TT-ICE${}^{*}$} achieves a CR of $117.9\times$ and a RR of $100.8\times$ in contrast to \ouralgorithm~with a CR of $91.5\times$ and a RR of $76.2\times$.

        At $\varepsilon_{rel}=0.10$, \texttt{TT-ICE${}^{*}$} compresses the training dataset in $44,687$ seconds and achieves a CR of $35.31\times$ and a RR of $31.03\times$. \ouralgorithm~compresses the same data in $49,408$ seconds and achieves a CR of $32.68\times$ and a RR of $27.24\times$.
        Note that despite the compression time for \ouralgorithm~at $\varepsilon_{rel}=0.10$ is longer than \texttt{TT-ICE${}^{*}$}, \ouralgorithm~successfully completes the compression task within the allocated maximum walltime. This indicates that \texttt{TT-ICE${}^{*}$} performs significantly more updates to its approximation compared to \ouralgorithm~and therefore spends most of the allocated walltime to compute the RTE.
        One interesting thing to note here is that the RTE of \ouralgorithm~is above the target $\varepsilon_{rel}=0.10$. The updates are computed for batches of 100 images but the RTE is computed over single images in the test set. This discrepancy in the batch size used for updates and the batch size used for computing the RTE is likely the reason for the high RTE.
        
        At $\varepsilon_{rel}=0.05$, \ouralgorithm~compresses 1781 batches in $20,880$ seconds and achieves a CR of $7.29\times$ and a RR of $6.07\times$. \texttt{TT-ICE${}^{*}$} compresses 129 batches in $1720$ seconds and achieves a CR of $3.59\times$ and a RR of $19.76\times$. Note that at this target $\varepsilon$ level, \texttt{TT-ICE${}^{*}$} runs into a timeout issue and does not complete the compression task. On the other hand, \ouralgorithm~runs into memory issues for all five seeds and does not complete the compression task. Further quantitative comparisons are provided in \Cref{fig:bigearthnet_compression_reduction,fig:bigearthnet_time_validation} in \Cref{app:bigearthnet_results}.

        \Cref{fig:bigearthnet_reconstructions_train,fig:bigearthnet_reconstructions_test} present a qualitative comparison for the reconstructions of the multispectral images from the BigEarthNet dataset using $\varepsilon_{rel}$ between $0.05$ and $0.30$. Images in \Cref{fig:bigearthnet_reconstructions_train} are generated by reconstructing latent space representations of a training image learned during the compression process. Images in \Cref{fig:bigearthnet_reconstructions_test} are generated by reconstructing latent space representations of a test image that was not seen during the compression process. This is done by first projecting the unseen multispectral image onto the cores of the corresponding tensor networks and then reconstructing the image from the projection. The images are displayed in RGB format, where only the bands B02, B03, and B04 are used.

        \begin{figure}[htbp]
            \centering
                \begin{subfigure}{0.15\textwidth}\centering
                    \includegraphics[width=\columnwidth]{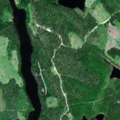}\\
                    \includegraphics[width=\columnwidth]{figures/train_26_58_og.png}
                    \caption{Original}
                \end{subfigure}
                \begin{subfigure}{0.15\textwidth}\centering
                    \includegraphics[width=\columnwidth]{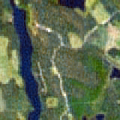}\\
                    \includegraphics[width=\columnwidth]{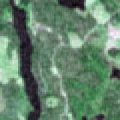}
                    \caption{$\varepsilon_{rel}=0.05$}
                \end{subfigure}
                \begin{subfigure}{0.15\textwidth}\centering
                    \includegraphics[width=\columnwidth]{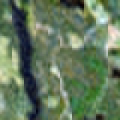}\\
                    \includegraphics[width=\columnwidth]{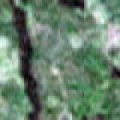}
                    \caption{$\varepsilon_{rel}=0.10$}
                \end{subfigure}
                \begin{subfigure}{0.15\textwidth}\centering
                    \includegraphics[width=\columnwidth]{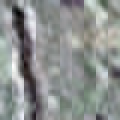}\\
                    \includegraphics[width=\columnwidth]{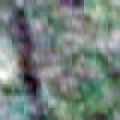}
                    \caption{$\varepsilon_{rel}=0.15$}
                \end{subfigure}
                \begin{subfigure}{0.15\textwidth}\centering
                    \includegraphics[width=\columnwidth]{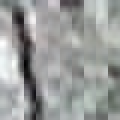}\\
                    \includegraphics[width=\columnwidth]{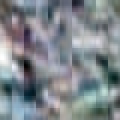}
                    \caption{$\varepsilon_{rel}=0.20$}
                \end{subfigure}
                \begin{subfigure}{0.15\textwidth}\centering
                    \includegraphics[width=\columnwidth]{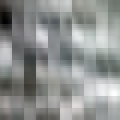}\\
                    \includegraphics[width=\columnwidth]{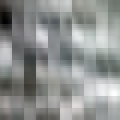}
                    \caption{$\varepsilon_{rel}=0.30$}
                \end{subfigure}
            \caption{\small\emph{Reconstructed multispectral images from BigEarthNet's training set, captured by the Sentinel-2 satellite, are displayed. Images compressed using \ouralgorithm~are in the top row, while those compressed with \texttt{TT-ICE${}^{*}$} appear in the bottom row. Different columns represent reconstructions at varying $\varepsilon$ target levels, from $\varepsilon_{rel}=0.05$ to $\varepsilon_{rel}=0.30$. Image quality decreases severely after $\varepsilon_{rel}=0.10$ for both algorithms. For $\varepsilon_{rel}=0.05$ and $0.10$, \ouralgorithm~results in shaper images compared to \texttt{TT-ICE${}^{*}$}. These images have been reconstructed from their latent representations, with only the RGB bands (B02, B03, B04) being depicted.}}
                \label{fig:bigearthnet_reconstructions_train}
            \end{figure}

            \Cref{fig:bigearthnet_reconstructions_train} shows that the quality of the reconstructed images decreases as the target $\varepsilon$ level increases. Specifically beyond $\varepsilon_{rel}=0.10$, the images become significantly pixelated and lose their original features. For $\varepsilon_{rel}=0.05$ and $0.10$, \ouralgorithm~reconstructs sharper images compared to \texttt{TT-ICE${}^{*}$}. This becomes evident especially at the finer details such as roads. \ouralgorithm~at $\varepsilon_{rel}=0.05$ results in the best reconstruction. This reconstruction has slight discoloration in comparison to the original image and maintains the fine features well.

            \begin{figure}[htbp]
                \centering
                    \begin{subfigure}{0.15\textwidth}\centering
                        \includegraphics[width=\columnwidth]{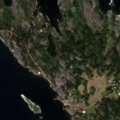}\\
                        \includegraphics[width=\columnwidth]{figures/test_85_12_og.png}
                        \caption{Original}
                    \end{subfigure}
                    \begin{subfigure}{0.15\textwidth}\centering
                        \includegraphics[width=\columnwidth]{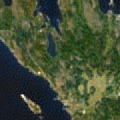}\\
                        \includegraphics[width=\columnwidth]{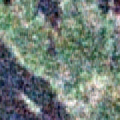}
                        \caption{$\varepsilon_{rel}=0.05$}
                    \end{subfigure}
                    \begin{subfigure}{0.15\textwidth}\centering
                        \includegraphics[width=\columnwidth]{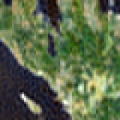}\\
                        \includegraphics[width=\columnwidth]{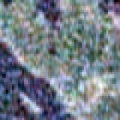}
                        \caption{$\varepsilon_{rel}=0.10$}
                    \end{subfigure}
                    \begin{subfigure}{0.15\textwidth}\centering
                        \includegraphics[width=\columnwidth]{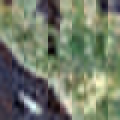}\\
                        \includegraphics[width=\columnwidth]{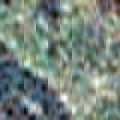}
                        \caption{$\varepsilon_{rel}=0.15$}
                    \end{subfigure}
                    \begin{subfigure}{0.15\textwidth}\centering
                        \includegraphics[width=\columnwidth]{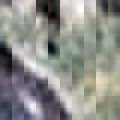}\\
                        \includegraphics[width=\columnwidth]{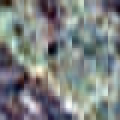}
                        \caption{$\varepsilon_{rel}=0.20$}
                    \end{subfigure}
                    \begin{subfigure}{0.15\textwidth}\centering
                        \includegraphics[width=\columnwidth]{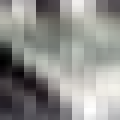}\\
                        \includegraphics[width=\columnwidth]{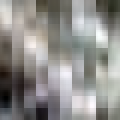}
                        \caption{$\varepsilon_{rel}=0.30$}
                    \end{subfigure}
                \caption{\small\emph{Reconstructed multispectral images from BigEarthNet's test set, captured by the Sentinel-2 satellite, are displayed. Images compressed using \ouralgorithm~are in the top row, while those compressed with \texttt{TT-ICE${}^{*}$} appear in the bottom row. Different columns represent reconstructions at varying $\varepsilon$ target levels, from $\varepsilon_{rel}=0.05$ to $\varepsilon_{rel}=0.30$. Image quality decreases severely after $\varepsilon_{rel}=0.15$ for \ouralgorithm~and after $\varepsilon_{rel}=0.10$ for \texttt{TT-ICE${}^{*}$}. For all $\varepsilon_{rel}$ levels, \ouralgorithm~results in shaper images compared to \texttt{TT-ICE${}^{*}$}. These images have been reconstructed from their latent representations, with only the RGB bands (B02, B03, B04) being depicted.}}
                \label{fig:bigearthnet_reconstructions_test}
                \end{figure}
                
            To demonstrate the generalization performance qualitatively, we present the reconstructions of the test images in \Cref{fig:bigearthnet_reconstructions_test}. Similar to training data, the quality of the reconstructed images decreases as the target $\varepsilon$ level increases. However, this time the difference between \ouralgorithm~and \texttt{TT-ICE${}^{*}$} is more pronounced. For $\varepsilon_{rel}=0.05$ and $0.10$, \texttt{TT-ICE${}^{*}$} introduces a significant amount of noise in the images which reduces the image quality. While offering a better reconstruction quality at all target $\varepsilon$ levels, \ouralgorithm~at $\varepsilon_{rel}=0.05$ results in the best reconstruction. In parallel with our findings with the training images, the reconstruction has slight discoloration in comparison to the original image and maintains the fine features well.

            \paragraph{\textbf{Summmary of conclusions from experiments with image data:}}
            Similar to the experiments with scientific data, the experiments with image data also demonstrate that \ouralgorithm~provides superior generalization performance compared to \texttt{TT-ICE${}^{*}$} across both investigated datasets. \ouralgorithm~consistently achieves lower RTE at all $\varepsilon$ levels within fewer frames and compresses in less time than \texttt{TT-ICE${}^{*}$}.
            \texttt{TT-ICE${}^{*}$} starts off with higher RR due to the properties of the TT-format but the RRs of both algorithms become comparable towards the end of the experiments, regardless of the target $\varepsilon$ level. One interesting observation here is that for higher $\varepsilon$ levels, \texttt{TT-ICE${}^{*}$} achieves higher CR than \ouralgorithm~but the image reconstructions yield lower quality at those target $\varepsilon$ levels. For the MineRL competition dataset \texttt{TT-ICE${}^{*}$} compresses more frames than \ouralgorithm~except for $\varepsilon=0.10$, but this is caused by the 1D latent space of \texttt{TT-ICE${}^{*}$} which is more memory efficient than the 2D latent space of \ouralgorithm~when $>200,000$ frames are used. Overall, the results suggest that \ouralgorithm~is more suitable for image datasets especially when the training dataset is limited or the visual quality of the reconstructions is important. \texttt{TT-ICE${}^{*}$} may be more appropriate when higher errors are allowed for compression.

\section{Conclusion}\label{sec:conclusion}
    
    In this work we proposed the batch hierarchical Tucker format, a slightly modified but more efficient version of the hierarchical Tucker format that is also suitable for incremental updates, and a new incremental tensor decomposition algorithm called Hierarchical Incremental Tucker (\ouralgorithm), which, to the best of our knowledge, is the first incremental tensor decomposition algorithm that updates an approximation in the hierarchical Tucker format. We compared \ouralgorithm~with an state-of-the-art incremental tensor decomposition algorithm \texttt{TT-ICE${}^{*}$} using two PDE-based and two image-based datasets. The results indicate that in datasets with multi-scale features \ouralgorithm~offers a fast and memory efficient way to compress high-dimensional data with a low relative test error. In simpler datasets, \texttt{TT-ICE${}^{*}$} outperforms \ouralgorithm~in terms of compression ratio, compression time, and relative test error. The results also suggest that \ouralgorithm~generalizes better with less training data compared to \texttt{TT-ICE${}^{*}$} regardless of the target $\varepsilon$ level and the dataset.

    Future work includes parallelizing \ouralgorithm~to speed up the compression process, provide a version of \ouralgorithm~to take advantage of GPU acceleration, and extending the \ouralgorithm~algorithm to work with n-ary trees. Furthermore, we plan to use the latent space learned through \ouralgorithm~in downstream learning tasks (such as generative modeling) to showcase the usability of the learned representations through \ouralgorithm.

\acks{
We acknowledge partial support by Los Alamos National Laboratories under the project “Algorithm/Software/Hardware Co-design for High Energy Density applications” at the University of Michigan, and partial support from the Automotive Research Center at the University of Michigan (UM) in accordance with Cooperative Agreement W56HZV-19-2-0001 with U.S. Army DEVCOM Ground Vehicle Systems Center.
This work used computing resources provided by an AFOSR DURIP under Program Manager Dr. Fariba Fahroo and grant number FA9550- 23-1-006 and additional computational resources and services provided by Advanced Research Computing (ARC), a division of Information and Technology Services (ITS) at the University of Michigan, Ann Arbor}

\appendix
\renewcommand\thefigure{\thesection.\arabic{figure}}
\renewcommand\thetable{\thesection.\arabic{table}}
\renewcommand\thealgorithm{\thesection.\arabic{algorithm}}
\setcounter{figure}{0}
\setcounter{table}{0}
\setcounter{algorithm}{0}
\begin{appendices}
    \section{Proofs}\label{app:proofs}
        This section contains mathematical proofs for the correctness of the \texttt{BHT-l2r} and \ouralgorithm~algorithms.

        To guarantee the correctness of the \texttt{BHT-l2r} algorithm, we first need to show that we can compute an upper bound for the approximation error incurred at an individual layer. This is done in the following lemma.
        \begin{lemma}[Layerwise approximation error]\label{thm:layerwise_approx_error}
            Given an $N$-batch of $d$-dimensional tensors $\mathcal{Y}\in\reals^{n_{1}\times\cdots\times n_{d}\times N}$, the approximation error incurred at the $\ell$-th layer $\sqrt{\|\mathcal{C}_{\ell+1}\|_{F}^{2}-\|\mathcal{C}_{\ell}\|_{F}^{2}}$ is upper bounded by
            \begin{equation*}
                \sqrt{\sum_{i=1}^{|\tree_{\ell}|}\|E_{\ell,i}\|^{2}_{F}}, \quad \text{ where } \quad
                E_{\ell,i}:=C_{\ell,(i)}-U_{\ell,i}\Sigma_{\ell,i}V^{T}_{\ell,i}
            \end{equation*}
            is the truncated portion the SVD performed on the mode-$i$ unfolding of the intermediate tensor $\mathcal{C}_{\ell}$ at the $\ell$-th layer.
        \end{lemma}
        \begin{proof}
            This is a straight forward result following the application of \cite[Lemma 3.8]{grasedyck2010hierarchical} to each node in $\tree_{\ell}$.
            In \Cref{eq:ht_leaf_contraction}, we have shown that the core tensor at the $\ell$-th layer is computed by contracting the core tensor at the $(\ell+1)$-th layer with orthonormal matrices obtained through error-truncated SVDs on the $\ell$-th layer.
            Thus, decomposing the $\ell$-th layer can be seen analogous to the notion \textit{successive truncation} of Lemma 3.8 of~\cite{grasedyck2010hierarchical}.
            Therefore, we can upper bound the approximation error at the $\ell$-th layer with the sum of the Frobenius norms of the truncated error matrices $E_{\ell,i}$.
         \end{proof}

        \Cref{thm:layerwise_approx_error} is a crucial step in providing an upper bound on the total approximation error of the \texttt{BHT-l2r} algorithm. The total approximation error is computed by summing the approximation errors at each layer, as shown in the following theorem.
        \begin{theorem}[Total approximation error~{\cite[Lemma 3.10]{grasedyck2010hierarchical}}]\label{thm:total_approx_error}
          Given an $N$-batch of $d$-dimensional tensors $\mathcal{Y}\in\reals^{n_{1}\times\cdots\times n_{d}\times N}$, the approximation error of the reconstructed batch hierarchical Tucker decomposition $\tilde{\mathcal{Y}}$ is upper bounded by $\sum_{\ell=0}^{p}\sum_{i=1}^{|\tree_{\ell}|}\|E_{\ell,i}\|^{2}_{F},$
            \begin{equation}\label{eq:total_approx_error}
                \|\mathcal{Y}-\tilde{\mathcal{Y}}\|^{2}_{F}\leq\sum_{\ell=0}^{p}\sum_{i=1}^{|\tree_{\ell}|}\|E_{\ell,i}\|^{2}_{F},
            \end{equation}
            where $p$ is the depth of the dimension tree $\tree$.
        \end{theorem}
        \begin{proof}
            This is a direct generalization of \Cref{thm:layerwise_approx_error} to the entire dimension tree $\tree$. Once an upper bound on the approximation error at each layer is established, the total approximation error can be computed by summing the approximation errors at each layer.
        \end{proof}
        
        This next corollary provides a direct application of how the approximation error can be distributed over the individual cores of the hierarchical tensor network.
        \begin{corollary}[\texttt{BHT-l2r} approximation error]\label{thm:bht_leaves_to_root}
            Let $\mathcal{Y}\in\reals^{n_{1}\times\cdots\times n_{d}\times N}$ be an $N$-batch of $d$-dimensional tensors and $\hierarchical_{\mathcal{Y}}$ be its approximation in batch HT format computed with \Cref{alg:batch_htucker} with an absolute error tolerance $\varepsilon_{abs}$, cores $\core$, and dimension tree $\tree$. Then, the reconstruction of this approximation using \cref{eq:htucker_reconstruction}, $\tilde{\mathcal{Y}}$, satisfies $\|\mathcal{Y}-\tilde{\mathcal{Y}}\|_{F}\leq\varepsilon_{abs}$ if a nodewise error tolerance $\varepsilon_{nw}$ is chosen such that
            \begin{equation}
                \varepsilon_{nw}=\frac{\varepsilon_{abs}}{\sqrt{2d-2}}.
            \end{equation}
        \end{corollary}    
        \begin{proof}
            The proof follows from \Cref{thm:total_approx_error}. If we set $\|E_{\ell,i}\| \leq \varepsilon_{nw}$ for all $\ell$ and $i$, then using \eqref{eq:total_approx_error} the expression for the overall approximation error becomes
            \begin{equation}
                \|\mathcal{Y}-\tilde{\mathcal{Y}}\|^{2}_{F}\leq\sum_{\ell=0}^{p}\sum_{i=1}^{|\tree_{\ell}|}\varepsilon_{nw}^{2}.
            \end{equation}
            Since $\varepsilon_{nw}$ is assumed constant for all layers and nodes, the double sum on the right-hand side is simply repeated $2d-2$ times, making the expression
            \begin{equation*}
                \|\mathcal{Y}-\tilde{\mathcal{Y}}\|^{2}_{F}\leq(2d-2)\varepsilon_{nw}^{2}.
            \end{equation*}
            Setting $\varepsilon_{nw}=\varepsilon_{abs}/\sqrt{2d-2}$ and taking the square root of both sides yields \eqref{eq:total_approx_error} and therefore completes the proof.
        \end{proof}

        Once we can guarantee an approximation error upper bound to \texttt{BHT-l2r}, a next natural step is to compute the exact approximation error of incurred by the \texttt{BHT-l2r} algorithm. We use \Cref{eq:relative_error_threshold} as a proxy for the approximation error. In the following claim we will prove that the relative error of the approximation can be computed directly from the Frobenius norms of the original tensor and its projected latent space representation.
        \begin{claim}[Orthogonal reconstructions]\label{thm:reconstruction_norm_equality}
            Given an $N^{k}$-batch of $d$-dimensional tensors $\mathcal{Y}^{k}\in\reals^{n_{1}\times\cdots\times n_{d}\times N^{k}}$, an accumulation of $d$-dimensional tensors in batch HT format $\hierarchical_{\mathcal{X}^{k-1}}$, the projection of $\mathcal{Y}^{k}$ onto the HT-cores $\bar{\mathcal{C}}^{k}_{1}\in\reals^{r_{1,1}\times r_{1,2}\times N^{k}}$, and the reconstruction of $\mathcal{Y}^{k}$ from $\bar{\mathcal{C}}^{k}_{1}$ using the HT-cores $\tilde{\mathcal{Y}}^{k}$, the error of approximation is equal to the difference in squared Frobenius norm of $\tilde{\mathcal{Y}}^{k}$ and $\bar{\mathcal{C}}^{k}_{1}$, i.e,
            \begin{equation*}
                \|\mathcal{Y}^{k}-\tilde{\mathcal{Y}}^{k}\|_{F}=\sqrt{\|\mathcal{Y}^{k}\|_{F}^{2}-\|\bar{\mathcal{C}}^{k}_{1}\|_{F}^{2}}.
            \end{equation*}
        \end{claim}
        \begin{proof}
            The proof is direct and has two steps. First we will show that $\|\mathcal{Y}^{k} - \tilde{\mathcal{Y}}^{k}\|_{F}^{2} = \|\mathcal{Y}^{k}\|_{F}^{2} - \|\tilde{\mathcal{Y}}^{k}\|_{F}^{2}$ and then we will show that $\|\tilde{\mathcal{Y}}^{k}\|_{F}^{2} = \|\bar{\mathcal{C}}^{k}_{1}\|_{F}^{2}$.

            First, we take the square of the approximation error to get $\|\mathcal{Y}^{k}-\tilde{\mathcal{Y}}^{k}\|_{F}^{2}$. Then, by using $\|A+B\|_{F}^{2}=\|A\|_{F}^{2} + \|B\|_{F}^{2}+2\langle A,B \rangle_{F}$ we expand the squared Frobenius norm of the approximation error as
            \begin{equation*}
                    \|\mathcal{Y}^{k}-\tilde{\mathcal{Y}}^{k}\|_{F}^{2} = \|\mathcal{Y}^{k}\|_{F}^{2} - 2\langle\mathcal{Y}^{k},\tilde{\mathcal{Y}}^{k}\rangle_{F} + \|\tilde{\mathcal{Y}}^{k}\|_{F}^{2}\\
            \end{equation*}
            Without loss of generality, we can split $\mathcal{Y}$ into $\mathcal{Y}^{k} = \mathcal{Y}^{k}_{\parallel} + \mathcal{Y}^{k}_{\perp}$, where $\mathcal{Y}^{k}_{\parallel}$ is the projection of $\mathcal{Y}^{k}$ onto the HT-cores $\core^{k-1}$ and $\mathcal{Y}^{k}_{\perp}$ is the orthogonal component of $\mathcal{Y}^{k}_{\parallel}$ with respect to the HT-cores of the accumulation $\mathcal{X}^{k-1}_{\hierarchical}$. Then, using the properties of the Frobenius inner product we can write $\langle\mathcal{Y}^{k},\tilde{\mathcal{Y}}^{k}\rangle_{F} = \langle\mathcal{Y}^{k}_{\parallel},\tilde{\mathcal{Y}}^{k}\rangle_{F} + \langle\mathcal{Y}^{k}_{\perp},\tilde{\mathcal{Y}}^{k}\rangle_{F}$. Since $\tilde{\mathcal{Y}}^{k}$ is the reconstruction of $\bar{\mathcal{C}}^{k}_{1}$ using the HT-cores, the first term is equal to the squared Frobenius norm of $\tilde{\mathcal{Y}}^{k}$. The second term is zero since $\mathcal{Y}^{k}_{\perp}$ is by definition orthogonal to $\tilde{\mathcal{Y}}^{k}$. Therefore, we can write the squared Frobenius norm of the approximation error as
            \begin{equation*}
                \|\mathcal{Y}^{k}-\tilde{\mathcal{Y}}^{k}\|_{F}^{2} = \|\mathcal{Y}^{k}\|_{F}^{2} - \|\tilde{\mathcal{Y}}^{k}\|_{F}^{2}
            \end{equation*}
            and complete the first part of the proof.

            For the second part, we need to show that $\|\tilde{\mathcal{Y}}^{k}\|_{F}^{2} = \|\bar{\mathcal{C}}^{k}_{1}\|_{F}^{2}$. The reconstruction of $\tilde{\mathcal{Y}}^{k}$ from $\bar{\mathcal{C}}^{k}_{1}$ is done by performing a sequence of outer products with the HT-cores $\core^{k-1}$.
            For the first layer of HT-cores, the outer product that yields the intermediate core for the $i$-th tensor in the $N^{k}$-batch, $\tilde{\mathcal{C}}^{k}_{2}(:,:,:,:,i)$\footnote{$\mathcal{A}(:,:,:,:,i)$ refers to the $i$-th mode-$5$ slice for the 5-dimensional tensor $\mathcal{A}$}, can be shown as a reshaping of
            \begin{equation*}
                \sum_{p,q=1}^{r_{1,1},r_{1,2}}\bar{\mathcal{C}}^{k}_{1}(p,q,i)\core^{k-1}_{1,1}(:,:,p) \otimes \core^{k-1}_{1,2}(:,:,q),
            \end{equation*}
            where $\otimes$ denotes the Kronecker product. 
            Note that $\core^{t}_{\ell,j}(:,:,i)$ is an orthonormal matrix for all $t,\ell,j,i$ if it is obtained through \texttt{bht-l2r} or \ouralgorithm.
            For Kronecker products, we have $\|A \otimes B\| = \|A\| \|B\|$ for any induced norm, therefore we have $\|\tilde{\mathcal{C}}^{k}_{2}(:,:,:,:,i)\|_{F}^{2}=\|\bar{\mathcal{C}}^{k}_{1}(p,q,i)\core^{k-1}_{1,1}(:,:,p)\otimes \core^{k-1}_{1,2}(:,:,q)\|_{F}^{2}=\|\bar{\mathcal{C}}^{k}_{1}(p,q,i)\|_{F}^{2}$ and $\|\tilde{\mathcal{C}}^{k}_{2}\|_{F}^{2}=\|\bar{\mathcal{C}}^{k}_{1}\|_{F}^{2}$.

            Since the orthonormality of $\core^{k-1}_{\ell,j}(:,:,i)$ holds for all $t,\ell,j,i$, the same argument can be applied to all layers of the HT-cores. Therefore, the squared Frobenius norm of the reconstruction is equal to the squared Frobenius norm of the projection, i.e., $\|\tilde{\mathcal{Y}}^{k}\|_{F}^{2}=\|\bar{\mathcal{C}}^{k}_{1}\|_{F}^{2}$. This completes the proof.
        \end{proof}
        \Cref{thm:reconstruction_norm_equality} also applies to the \ouralgorithm~algorithm, as the reconstruction of the streamed tensor from the hierarchical Tucker approximation is done in the same way as in \texttt{BHT-l2r}, using orthonormal cores.
        
        Finally with then next theorem we provide a proof of correctness for the \ouralgorithm~algorithm.
        
        \begin{theorem}[\ouralgorithm~approximation error]\label{thm:hit_approx_error}
        \end{theorem}
        \begin{proof}
            The proof uses \Cref{thm:total_approx_error} and is direct.
            Let $\bar{\mathcal{C}}^{k}_{\ell}$ (for any $\ell=p,\dots,1$ and $\bar{\mathcal{C}}_{p}=\mathcal{Y}^{k}$) be the intermediate tensor obtained at the $\ell$-th layer while processing the streamed tensor $\mathcal{Y}^{k}$ by \ouralgorithm.
            After updating $\core^{k-1}_{\ell,j}$ to $\core^{k}_{\ell,j}$ using $\bar{\mathcal{C}}^{k}_{\ell}$ and 
            \Cref{eq:compute_residual,eq:svd_on_residual,eq:concatenating_orthonormal_vectors}, we can write $\bar{\mathcal{C}}^{k}_{\ell}$ as
            \begin{equation}\label{thm:hit_approx_error_eq:composition_of_error}
                \bar{C}^{k}_{\ell,(j)}
                =U^{k-1}_{\ell,j}(U^{k-1}_{\ell,j})^{T}\bar{C}^{k}_{\ell,(j)}+U^{k}_{R_{\ell,j}}\Sigma^{k}_{R_{\ell,j}}(V^{k}_{R_{\ell,j}})^{T}+E_{R_{\ell,j}},
            \end{equation}
            where $U^{k-1}_{\ell,j}$ is the reshaped core $\core^{k-1}_{\ell,j}$. The first term on the RHS is the part of $\bar{\mathcal{C}}^{k}_{\ell}$ that is already represented in $\mathcal{X}^{k-1}_{\hierarchical}$, and the second term is the part represented due to the added basis vectors $U^{k}_{R_{\ell,j}}$.
            If we consider the updated core $U^{k}_{\ell,j}=\mat{U^{k-1}_{\ell,j} & U^{k}_{R_{\ell,j}}}$, we then can manipulate \eqref{thm:hit_approx_error_eq:composition_of_error} to resemble \Cref{thm:layerwise_approx_error} as
            \begin{equation*}
                \bar{C}^{k}_{\ell,(j)}=U^{k}_{\ell,j}(U^{k}_{\ell,j})^{T}\bar{C}^{k}_{\ell,(j)}+E_{R_{\ell,j}}= U^{k}_{\ell,j}\Sigma^{k}_{\ell,j}(V^{k}_{\ell,j})^{T}+E_{R_{\ell,j}},
            \end{equation*}
            where the error in the approximation of $\bar{\mathcal{C}}^{k}_{\ell}$ is given by $E_{\ell,j}$.
            This allows us to compute an upper bound on the error incurred for each HT-core on the $\ell$-th layer using \Cref{thm:layerwise_approx_error} and accumulate the layerwise approximation upper bounds to obtain a total approximation error using \Cref{thm:total_approx_error}.
        \end{proof}

        Since we have shown that we can guarantee an upper bound on the approximation error of the \ouralgorithm~algorithm, now the last step is to show that the updates performed by \ouralgorithm~do not affect the representation of the tensors that were streamed in the past. This is shown in the following theorem.
        \begin{theorem}[Error guarantees for the past stream]\label{thm:hit_past_stream}
        Let $\mathcal{X}^{k}_{m} \in \mathbb{R}^{n_{1} \times \cdots \times n_{d}}$ be the $m$-th tensor in the accumulation $\mathcal{X}^{k}$ at time $k$, and $\hierarchical_{\mathcal{X}^{k}}$ be the hierarchical Tucker approximation computed using \Cref{alg:HIT} (\ouralgorithm) with cores $\core^{k}_{\ell,j}$. For any $t \geq k$, the reconstruction of the $m$-th tensor from $\hierarchical_{\mathcal{X}^{t}_{m}}$, denoted $\tilde{\mathcal{X}}^{t}_{m}$, equals the reconstruction from $\hierarchical_{\mathcal{X}^{k}_{m}}$, i.e., $\tilde{\mathcal{X}}^{k}_{m} = \tilde{\mathcal{X}}^{t}_{m}$.
        \end{theorem}
        \begin{proof}
            The proof simply demonstrates that the core modifications performed by \ouralgorithm~have no impact on the representation of earlier tensors and is similar to \cite[Theorem 4]{aksoy2024incremental}.
            To demonstrate $\tilde{\mathcal{X}}^{k}_{m} = \tilde{\mathcal{X}}^{t}_{m}$, we need to reconstruct $\hierarchical_{\mathcal{X}^{t}_{m}}$.
            Let $\bar{C}^{k,m}_{1}\in\reals^{r^{k}_{1,1}\times r^{k}_{1,2}}$ be the slice of the root core $\bar{\mathcal{C}}^{k}_{1}$ that corresponds to $\mathcal{X}^{k}_{m}$.
            For any $t>k$, $\bar{C}^{k,m}_{1}$ gets padded with zeros according to  \Cref{eq:incremental_zero_padding1} and becomes
            \begin{equation}\label{thm:hit_past_stream_eq:root_slice}
                \bar{C}^{t,m}_{1}=\mat{\bar{C}^{k,m}_{1} & \mathbf{0} \\ \mathbf{0} & \mathbf{0}},
            \end{equation}
            with shape $r^{t}_{1,1}\times r^{t}_{1,2}$.
            
            Since the reconstruction happens in the root-to-leaf direction, we use the $\bar{C}^{t,m}_{1}$ from \eqref{thm:hit_past_stream_eq:root_slice} and contract with the HT-cores on the first level as $\tilde{\mathcal{C}}^{t,m}_{1}=\bar{C}^{t,m}_{1}~\tenprod{1}{3}~\core^{t}_{1,1}~\tenprod{3}{3}~\core^{t}_{1,2}$.
            As a result, we obtain $\tilde{\mathcal{C}}^{t,m}_{1}\in\reals^{r^{t}_{2,1}\times r^{t}_{2,2} \times r^{t}_{2,3} \times r^{t}_{2,4}}$ is the intermediate core of layer 1.
            After the contraction, we compute the $(i_{1},i_{2},i_{3},i_{4})$-th element of $\tilde{\mathcal{C}}^{t,m}_{1}$ as
            \begin{equation}\label{thm:hit_past_stream_eq:layer_1_elementwise_reconstruction}
                \tilde{\mathcal{C}}^{t,m}_{1}(i_{1},i_{2},i_{3},i_{4})=\sum_{\alpha=1}^{r^{t}_{1,1}}\sum_{\beta=1}^{r^{t}_{1,2}} \bar{C}^{t,m}_{1}(\alpha,\beta)\core^{t}_{1,1}(i_{1},i_{2},\alpha)\core^{t}_{1,2}(i_{3},i_{4},\beta).
            \end{equation}
            However, since $\bar{C}^{t,m}_{1}(\alpha,\beta)=0~ \forall \alpha>r^{k}_{1,1},~ \beta>r^{k}_{1,2}$, we can rewrite \Cref{thm:hit_past_stream_eq:layer_1_elementwise_reconstruction} as
            \begin{equation*}
                \tilde{\mathcal{C}}^{t,m}_{1}(i_{1},i_{2},i_{3},i_{4})=\sum_{\alpha=1}^{r^{k}_{1,1}}\sum_{\beta=1}^{r^{k}_{1,2}} \bar{C}^{k,m}_{1}(\alpha,\beta)\core^{t}_{1,1}(i_{1},i_{2},\alpha)\core^{t}_{1,2}(i_{3},i_{4},\beta).
            \end{equation*}
            Next, we consider the padding of the contracted cores.
            Since cores on layer $\ell$ are padded with zeros to ensure dimensional consistency after updating layer $\ell+1$ as in \Cref{eq:incremental_zero_padding1}, the entries of $\core^{t}_{1,1}(i_{1},i_{2},:r^{k}_{1,1})$ and $\core^{t}_{1,2}(i_{3},i_{4},:r^{k}_{1,2})$ are all zero for any $i_{j} > r^{k}_{2,j}$, where $\core^{t}_{\ell,j}(\alpha,\beta,:r^{k}_{\ell,j})$ is the mode-3 fiber of $\core^{t}_{\ell,j}$ at position $(\alpha,\beta)$ with depth $r^{k}_{\ell,j}$. Therefore, $\tilde{\mathcal{C}}^{t,m}_{1}(i_{1},i_{2},i_{3},i_{4})=0$ for any $i_{j}>r^{k}_{2,j}$. For a full dimension tree this results in 
            \begin{equation}\label{thm:hit_past_stream_eq:reconstruction_layer_2}
                \begin{aligned}
                \tilde{\mathcal{C}}^{t,m}_{2}(i_{1},\dots,i_{8})= &\sum_{\alpha=1}^{r^{t}_{2,1}}\sum_{\beta=1}^{r^{t}_{2,2}}\sum_{\gamma=1}^{r^{t}_{2,3}}\sum_{\theta=1}^{r^{t}_{2,4}} \tilde{\mathcal{C}}^{t,m}_{1}(\alpha,\beta,\gamma,\theta)\core^{t}_{2,1}(i_{1},i_{2},\alpha)\cdots\core^{t}_{4,1}(i_{7},i_{8},\alpha),\\
                =& \sum_{\alpha=1}^{r^{k}_{2,1}}\sum_{\beta=1}^{r^{k}_{2,2}}\sum_{\gamma=1}^{r^{k}_{2,3}}\sum_{\theta=1}^{r^{k}_{2,4}} \tilde{\mathcal{C}}^{k,m}_{1}(\alpha,\beta,\gamma,\theta)\core^{t}_{2,1}(i_{1},i_{2},\alpha)\cdots\core^{t}_{4,1}(i_{7},i_{8},\alpha).
                \end{aligned}
            \end{equation}
            for the second layer. 
            This process is repeated for layers $\ell=3,\dots,p-1$. Then for the last layer $p$, the reconstruction is performed as
            \begin{equation}
                \begin{aligned}\label{thm:hit_past_stream_eq:reconstruction_last_layer}
                \tilde{\mathcal{X}}^{t}_{m}=\tilde{\mathcal{C}}^{t,m}_{p}(i_{1},\dots,i_{d})= &\sum_{\alpha=1}^{r^{t}_{p,1}}\cdots\sum_{\zeta=1}^{r^{t}_{p,d}} \tilde{\mathcal{C}}^{t,m}_{p-1}(\alpha,\dots,\zeta)\core^{t}_{p,1}(i_{1},\alpha)\cdots\core^{t}_{p,d}(i_{d},\zeta),\\
                =&\sum_{\alpha=1}^{r^{k}_{p,1}}\cdots\sum_{\zeta=1}^{r^{k}_{p,d}} \tilde{\mathcal{C}}^{k,m}_{p-1}(\alpha,\dots,\zeta)\core^{t}_{p,1}(i_{1},\alpha)\cdots\core^{t}_{p,d}(i_{d},\zeta).
                \end{aligned}
            \end{equation}
            \Cref{thm:hit_past_stream_eq:layer_1_elementwise_reconstruction,thm:hit_past_stream_eq:reconstruction_layer_2,thm:hit_past_stream_eq:reconstruction_last_layer} all use the ranks up to $r^{k}_{\ell,j}$ while reconstructing $\tilde{\mathcal{X}}^{t}_{m}$. This shows that for any $t\geq k$, the reconstruction of the $m$-th tensor in the accumulation $\hierarchical_{\mathcal{X}^{t}}$ is the same as the reconstruction of the $m$-th tensor in the accumulation $\hierarchical_{\mathcal{X}^{k}}$. Therefore, the error guarantees of the \ouralgorithm~algorithm are preserved for the past stream.
        \end{proof}

    \section{Supplementary Algorithms}\label{app:supplementary_algorithms}
        This section contains pseudocodes for the supplementary algorithms that are used in \Cref{,alg:batch_htucker,alg:HIT}.
        \begin{itemize}
            \item \texttt{updateIndexSet}:
            During the update process with \ouralgorithm~(\Cref{alg:HIT}), it is important to keep track of the added basis vectors to each core. To ensure dimensional consistency throughout the update process, the \texttt{updateIndexSet} function~(\cref{alg:updateIndexSet}) updates the index sets corresponding to layers and nodes. 
        \begin{algorithm}
            \centering
            \caption{\texttt{updateIndexSet}: Updates the index set of a given object}
            \label{alg:updateIndexSet}
            \begin{algorithmic}[1]
                \Input
                \Desc{$\indexset_{\Gamma_{\alpha}}$}{index set of the object $\Gamma$}
                \Desc{$\Gamma_{\alpha}$}{object to update the index set}
                \EndInput
                \Output
                    \Desc{$\indexset_{\Gamma_{\alpha}}$}{updated index set}
                \EndOutput
                \If{$\Gamma_{\alpha}$ is a node}\Comment{\small$\alpha$ is a tuple with layer and node indices}
                    \State $\ell,j \gets \alpha$
                \State $\indexset_{\Gamma_{\alpha}} \gets \begin{cases}
                    \{n_{d_{\ell,j}}, r_{\ell,j}\} & \text{if $\Gamma_{\alpha}$ is a leaf}\\
                    \left\{\displaystyle{\bigcup_{(\ell+1,p) \in \successors_{\ell,i}}}\rank_{\ell+1,p}\right\}\cup \{{\rank_{\ell, i}}\}& \text{if $\Gamma_{\alpha}$ is a transfer node.}
                \end{cases}$
                \ElsIf{$\Gamma$ is a layer} \Comment{\small$\alpha$ is the layer index}
                    \State $\indexset_{\Gamma_{\alpha}}\gets
                    \bigcup_{j=1}^{|\Gamma_{\alpha}|}
                      \begin{cases}
                          n_{d_{\alpha,j}} & \text{if $\Gamma_{\alpha,j}$ is a leaf}\\ 
                          \displaystyle{\prod_{(\alpha+1,m) \in \successors_{\alpha,j}}} \rank_{\alpha+1,m} & \text{if $\Gamma_{\alpha,j}$ is a transfer node.}
                      \end{cases}$ \Comment{\small$\Gamma_{\alpha,j}$ is the $j$-th node on layer $\alpha$}
                \EndIf
            \end{algorithmic}
        \end{algorithm}

            \item \texttt{expandCore}:
            Once the missing basis vectors of a core are identified, then the next step is to merge the missing basis vectors with the existing ones stored in a Tucker core. The \texttt{expandCore}~(\Cref{alg:expandCore}) algorithm takes in a node, along with its dimension tree, existing Tucker core and the new basis vectors corresponding to that node and updates the core.
            \begin{algorithm}[h!]
                \caption{\texttt{expandCore}: Expands the basis of a Tucker core}
                \label{alg:expandCore}
                \begin{algorithmic}
                    \Input
                        \Desc{$\node_{\ell,j}$}{node whose core is to be expanded}
                        \Desc{$\indexset_{\node_{\ell,j}}$}{index set of the node}
                        \Desc{$\core_{\ell,j}$}{core to be expanded}
                        \Desc{$U$}{$\rank_{U}$ orthogonal vectors to be appended to $\core_{\ell,j}$}
                    \EndInput
                    \Output
                        \Desc{$\core_{\ell,j}$}{\textit{updated} core}
                        \Desc{$\indexset_{\node_{\ell,j}}$}{\textit{updated} index set}
                    \EndOutput
                    \If{$\node_{\ell,j}$ is a leaf node}
                        \State $\core_{\ell,j}\gets\mat{\core_{\ell,j} & U}$ \Comment{\small $\core_{\ell,j}$ is already in orthonormal matrix form, we can directly append $U$ to $\core_{\ell,j}$}
                        \State $\indexset_{\node_{\ell,j}}\gets \texttt{updateIndexSet}(\indexset_{\node_{\ell,j}},\node_{\ell,j})$ \Comment{\small Update the index set of the node to reflect the new rank}
                    \ElsIf{$\node_{\ell,j}$ is a transfer node}
                        \State $\alpha\gets\frac{1}{\rank_{\ell,j}}\prod_{\gamma\in\indexset_{\node_{\ell,j}}}\gamma$
                        \State $\core_{\ell,j}\gets\reshape(\core_{\ell,j},[\alpha,\rank_{\ell,j}])$ \Comment{\small$\core_{\ell,j}$ is reshaped to an orthonormal matrix before appending $U$}
                        \State $\core_{\ell,j}\gets\mat{\core_{\ell,j} & U}$
                        \State $\indexset_{\node_{\ell,j}}\gets \texttt{updateIndexSet}(\indexset_{\node_{\ell,j}},\node_{\ell,j})$ \Comment{\small Update the index set of the node to reflect the new rank}
                        \State $\core_{\ell,j}\gets\reshape(\core_{\ell,j}, I_{\node_{\ell,j}})$ \Comment{\small Fold $\core_{\ell,j}$ back to 3D core}
                    \EndIf
                \end{algorithmic}
            \end{algorithm}
        \item \texttt{padWithZeros}:
            The \texttt{padWithZeros} algorithm is used to pad a Tucker core with zeros to ensure dimensional consistency after updating the core. The algorithm is presented in \Cref{alg:padWithZeros}.
            \begin{algorithm}[h!]
                \caption{\texttt{padWithZeros}: Pad a hierarchical Tucker core with zeros}
                \label{alg:padWithZeros}
                \begin{algorithmic}[1]
                \Input
                    \Desc{$\core_{\ell,j}\in\reals^{\alpha\times\beta\times r}$}{HT-core to be padded}
                    \Desc{$t$}{index of the node to be padded among the children of its parent node}
                    \Desc{$r^{\prime}$}{size of the zeros to be padded}
                \EndInput
                \Output
                    \Desc{$\core_{\ell,j}$}{padded HT-core}
                \EndOutput
                \State $\mathbf{0} \in \reals^{r^{\prime}\times\beta\times r}$ \Comment{\small Assume $t=1$ for simplicity}
                \State $\core_{\ell,j}\gets \core_{\ell,j} \pad^{t} \mathbf{0}$ \Comment{\small$\core_{\ell,j} \in \reals^{(\alpha+r^{\prime}\times\beta\times r)}$}
                \end{algorithmic}
            \end{algorithm}
        \end{itemize}

            \section{Alternative approach to determine an upper bound to the approximation error \texorpdfstring{$\|E_{\ell,j}\|_{F}$}{}} \label{app:alt_error}            
            
        In this section, we propose an alternative method to determine an upper bound to the approximation error $\|E_{\ell,j}\|_{F}$ for \texttt{BHT-l2r} and \ouralgorithm~algorithms. 
        
        Since all the operations in both algorithms are projections onto truncated orthogonal bases, we can compute the error of approximation directly by comparing the norm of an intermediate tensor $\mathcal{C}_{\ell}$ to the original $d$-dimensional tensor $\mathcal{Y}$. This allows us to update the nodewise error tolerance $\varepsilon_{nw}$ adaptively throughout both algorithms. Let $\varepsilon_{abs}$ be the determined absolute error tolerance for the hierarchical Tucker approximation. Both \texttt{BHT-l2r} and \ouralgorithm~will perform $(2d-2)$ \texttt{SVD}s to either compute a hierarchical Tucker representation or update an existing one. At the beginning of both algorithms, the nodewise error tolerance $\varepsilon_{nw}$ is set to $\varepsilon_{abs}/\sqrt{2d-2}$.

        Once the computation of the $\ell$-th layer is complete and $\mathcal{Y}$ is projected onto $\core_{\ell,j}$, we can track the current error of approximation. Let $\bar{\mathcal{C}}_{\ell}$ be that approximation. Then, the current error of approximation is then computed as
        \begin{equation}
            \varepsilon_{\ell}=\sqrt{\|\mathcal{Y}\|^{2}_{F}-\|\bar{\mathcal{C}}_{\ell}\|^{2}_{F}},
        \end{equation}
        where $\varepsilon_{\ell}$ denotes the error of approximation at the $\ell$-th layer. To update the nodewise error tolerance $\varepsilon_{nw}$ for the upcoming computations, we define an error budget $\varepsilon_{rem}$ as
        \begin{equation}
            \varepsilon_{rem}=\sqrt{\varepsilon_{abs}^{2}-\varepsilon_{\ell}^{2}}.
        \end{equation}
        In addition to updating the remaining error budget, we also update the nodewise error tolerance $\varepsilon_{nw}$ accordingly as
        \begin{equation}
            \varepsilon_{nw}=\frac{\varepsilon_{rem}}{\sqrt{\texttt{svd}_{rem}}},
        \end{equation}
        with $\texttt{svd}_{rem}$ being the number of remaining \texttt{SVD}s to be performed.
        
        This approach aims to compensate for the actual error of approximation incurred at each layer being lower than the desired approximation error as shown in \Cref{thm:bht_leaves_to_root} and \Cref{thm:hit_approx_error}. By updating the nodewise error tolerance $\varepsilon_{nw}$ adaptively, we can get progressively more aggressive in the truncation of the Tucker cores as we move up the dimension tree. This approach is particularly useful when the original tensor $\mathcal{Y}$ is comprised of many small dimensions.
    \section{Additional results}\label{app:additional_results}
        This section contains the additional results for the experiments conducted in \Cref{sec:experiments}. In addition to those, this section further presents the comparison of the batch hierarchical Tucker format and the hierarchical Tucker format.

        \subsection{Comparison of the batch hierarchical Tucker format and the hierarchical Tucker format}\label{app:ht_vs_bht}
            
            This section considers the comparison of the batch hierarchical Tucker format and the hierarchical Tucker format. We compute the approximations in batch hierarchical Tucker  format (\Cref{fig:htucker_batch_new}) using the \texttt{BHT-l2r} algorithm (\Cref{alg:batch_htucker}) and the hierarchical Tucker format (\Cref{fig:htucker_as_node}) using the leaf-to-root compression algorithm of~\citet[Alg.~2]{grasedyck2010hierarchical}. We ran all algorithms until they run out of the allocated 64GB memory. Please note that in the experiments in this section do not include any incremental updates as the algorithms compared here are both one-shot algorithms. At every batch size, we compute the compression using the respective formats and algorithms from scratch.
            
            We conduct our experiments using the PDEBench dataset with $\varepsilon_{rel}=0.10$ and $0.05$ as well as the BigEarthNet dataset with $\varepsilon_{rel}=0.05$, $0.10$, $0.15$, and $0.30$ with the same reshapings as in \Cref{sec:experiments}. For all scenarios, we compare compression ratio and compression time of both algorithms. Analogous to other experiments, the results are averaged over 5 seeds.
            When computed with \texttt{BHT-l2r}, our proposed batch hierarchical Tucker format results in a significant reduction in compression time in comparison to the regular hierarchical Tucker format computed with the leaves-to-root decomposition algorithm presented in~\citet[Alg. 2]{grasedyck2010hierarchical} while returning better compression ratios.
            \Cref{fig:ht_vs_bht_pdebench_010,fig:ht_vs_bht_pdebench_005}, and \Cref{fig:ht_vs_bht_bigearthnet} show the results for the PDEBench and BigEarthNet datasets, respectively.
            
            \Cref{fig:ht_vs_bht_pdebench_010} shows the results of the comparison between HT format and BHT format in CR and compression time for $\varepsilon_{rel}=0.10$. Note that the batch hierarchical Tucker format achieves higher compression in comparison to their hierarchical Tucker format version using the same normalization. At this target relative error level, both algorithms compress the same amount of simulations except for HT format with z-score normalization. Another intersting observation from \Cref{fig:ht_vs_bht_pdebench_010} is that the BHT format results in faster compression irrespective of the normalization method used.
            
            \begin{figure}[htbp]
                \centering
                    \includegraphics[width=\columnwidth]{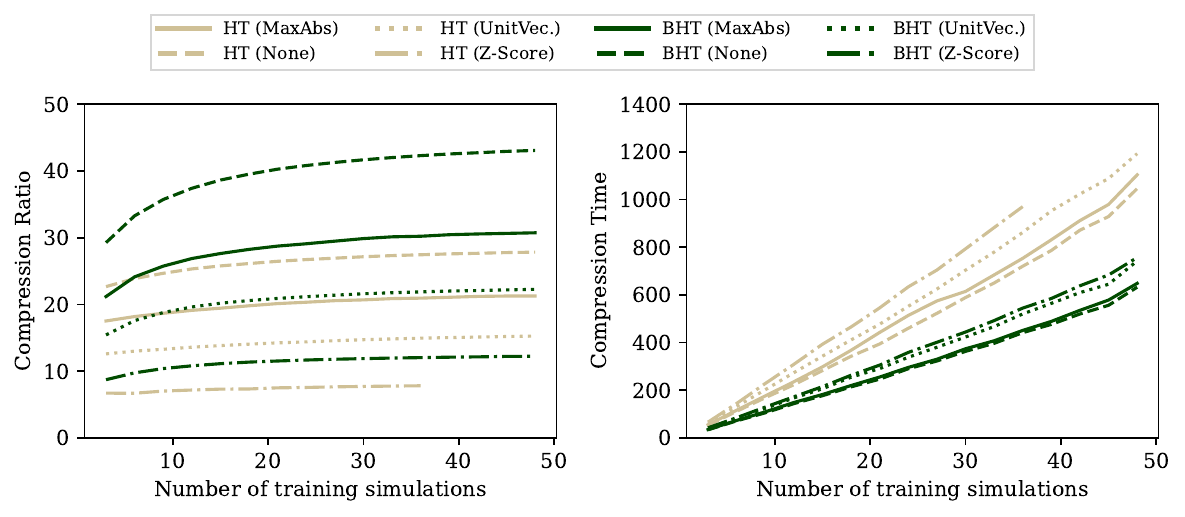}
                    \caption{\small\emph{Comparison of hierarchical Tucker leaf-to-root compression and batch hierarchical Tucker decomposition in terms of compression ratio~(CR - left) and compression time~(right) on the PDEBench 3D Navier-Stokes dataset with $\varepsilon_{rel}=0.10$ using various normalization methods. For the same normalization methods, the batch hierarchical Tucker format achieves on average $1.5\times$ compression over the hierarchical Tucker format, while running on average $1.7\times$ faster. The batch hierarchical Tucker format compresses 12 simulations more than the hierarchical Tucker format using z-score normalization. All experiments are run until failure due to insufficient memory. 
                    MaxAbs: Maximum absolute value normalization, None: No normalization, UnitVec: Unit vector normalization, ZScore: Z-score normalization. The results are averaged over 5 seeds.}}
                    \label{fig:ht_vs_bht_pdebench_010}
            \end{figure}

            \Cref{fig:ht_vs_bht_pdebench_005} shows the results of the comparison between HT format and BHT format in CR and compression time for $\varepsilon_{rel}=0.05$. Note that similar to the case of $\varepsilon_{rel}=0.10$, the BHT format offers higher compression and lower compression time compared to the HT format. This time, the efficiency difference between BHT and HT manifests itself in the compression of more simulations. The BHT format compresses at least 6 simulations more than the HT format across all normalization methods. In parallel to the results of $\varepsilon_{rel}=0.10$, the BHT format results in faster compression irrespective of the normalization method used.

            \begin{figure}[htbp]
                \centering
                \includegraphics[width=\columnwidth]{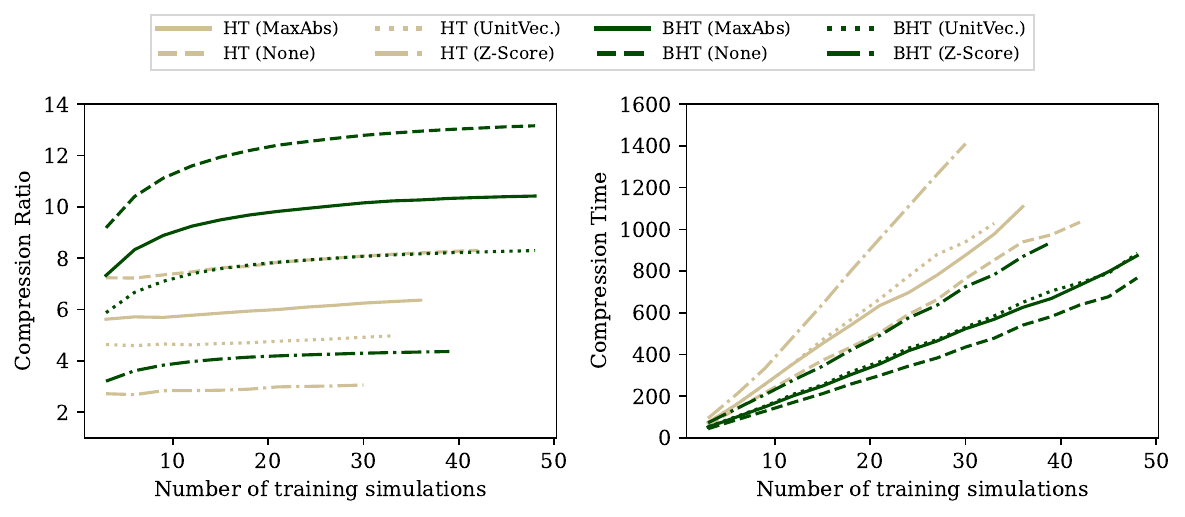}
                \caption{\small\emph{Comparison of HT format and BHT format in terms of compression ratio~(CR - left) and compression time~(right) on the PDEBench 3D Navier-Stokes dataset with $\varepsilon_{rel}=0.05$ using various normalization methods. For the same normalization methods, the BHT format achieves $1.4-1.6\times$ compression over the HT format, while running $1.6-1.9\times$ faster. The BHT format compresses at least 6 simulations more than the HT format at all normalization methods. All experiments are run until failure due to insufficient memory. Please refer to the caption of \Cref{fig:ht_vs_bht_pdebench_010} for the normalization methods.}}
                \label{fig:ht_vs_bht_pdebench_005}
            \end{figure}

            In addition to the 3D Navier-Stokes simualtions, we also conduct experiments on the BigEarthNet dataset. \Cref{fig:ht_vs_bht_bigearthnet} shows the results of the comparison between HT format and BHT format in CR and compression time for the BigEarthNet dataset.
            In line with our findings from the PDEBench dataset, the BHT format offers higher compression and lower compression time compared to the HT format, specifically at higher target relative error levels.
            At the higher target relative error level, $\varepsilon_{rel}=0.30$, the BHT format achieves $6.2\times$ compression over the HT format.
            The difference between HT and BHT formats become less pronounced as the target relative error level decreases. At $\varepsilon_{rel}=0.15$, the BHT format achieves only $1.5\times$ compression over the HT format.
            At the lower target relative error level, $\varepsilon_{rel}=0.10$, both formats yield comparable CRs with the HT format achieving $1.05\times$ compression over the BHT format.
            Finally at $\varepsilon_{rel}=0.05$, the HT format achieves $1.3\times$ compression over the BHT format.
            In terms of compression time, the BHT format consistently outperforms the HT format. Over the target relative errors investigated, the BHT format runs $1.86-3.86\times$ faster than the HT format.

            One interesting observation from \Cref{fig:ht_vs_bht_bigearthnet} is that despite the higher compression ratios achieved by the BHT format, the HT format compresses more images than the BHT format at $\varepsilon_{rel}=0.30$ and $\varepsilon_{rel}=0.15$. The BHT format compresses 700 images less than the HT format at $\varepsilon_{rel}=0.30$ and 100 images less than the HT format at $\varepsilon_{rel}=0.15$. However, once the HT format surpasses the BHT format in terms of CR, the BHT format compresses more images. The BHT format compresses 100 images more than the HT format at $\varepsilon_{rel}=0.10$ and 500 images more than the HT format at $\varepsilon_{rel}=0.05$.  

            \begin{figure}[htbp]
                \centering
                \includegraphics[width=\columnwidth]{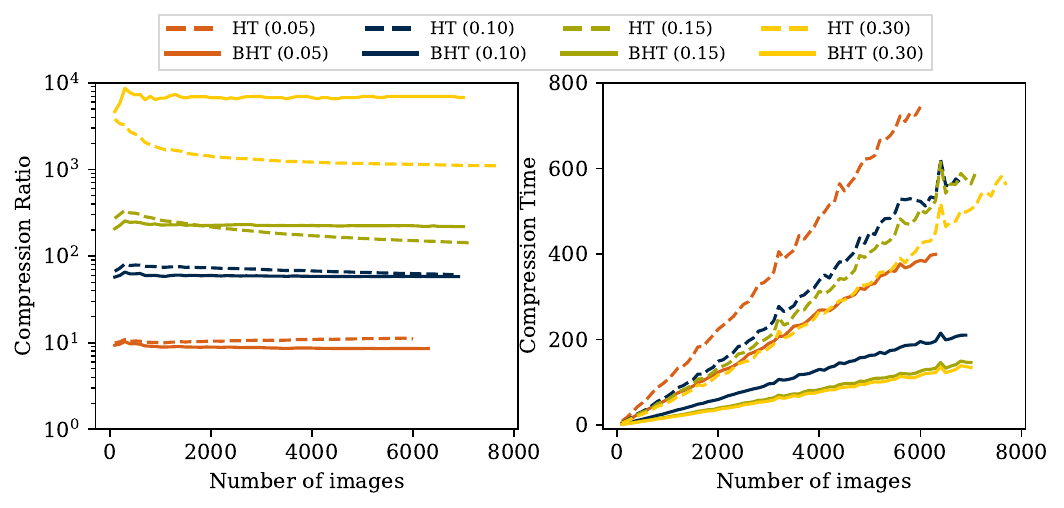}
                \caption{\small\emph{Comparison of hierarchical Tucker format and batch hierarchical Tucker format in terms of compression ratio~(CR - left) and compression time~(right) on the BigEarthNet dataset with $\varepsilon_{rel}=0.05-0.30$. The BHT format achieves $6.2\times$ compression over the HT format at $\varepsilon_{rel}=0.3$ while running $3.75\times$ faster. At $\varepsilon_{rel}=0.1$, the BHT format achieves $0.95\times$ compression over the HT format while running $2.69\times$ faster. The BHT format compresses 100 images more and 700 images less than HT format at $\varepsilon_{rel}=0.1$ and $\varepsilon_{rel}=0.3$, respectively. All experiments are run until failure due to insufficient memory.}}
                \label{fig:ht_vs_bht_bigearthnet}
            \end{figure}

        \subsection{Effect of the reshaping on compression performance}\label{app:reshaping}
            This section investigates the effect of reshaping on the compression performance of the \ouralgorithm~algorithm. We consider the PDEBench 3D turbulent Navier-Stokes dataset with $\varepsilon_{rel}=0.10$ and no normalization. We consider four different reshaping of the tensor. The baseline reshaping is denoted as $8\times8\times8\times8\times8\times8\times5\times21\times1$. We alter the dimensionality by reshaping the tensor. Note that we only change the dimensionality of the tensor by altering the dimensions that correspond to the spatial discretization of the simulation.
            \Cref{fig:pdebench_compression_reduction_reshaping_010} shows the compression time and relative test error of the algorithms using various reshaping. \Cref{fig:pdebench_time_validation_reshaping_010} shows the comparison of compression time and RTE of the algorithms using the considered reshapings. \Cref{tab:reshaping} presents the shapes of the tensors, and the resulting compression time, CR, RR, as well as RTE. For this type of experiments we set the maximum walltime limit to 4 days, similar to the experiments in \Cref{sec:pdebench_experiments}

            \begin{table}[htbp]
                \centering
                \caption{\small\emph{Results of experiments investigating the effect of reshaping on the compression performance of the \ouralgorithm~algorithm using the PDEBench 3D turbulent Navier-Stokes dataset with $\varepsilon_{rel}=0.10$. Time: Compression time in seconds, CR: Compression ratio, RR: Reeduction ratio, RTE: Relative test error. \wt~ denotes that the algorithm failed to compress the dataset due to maximum walltime limit (389/480 simulations).}}
                \label{tab:reshaping}
                \resizebox{\textwidth}{!}{
                \begin{tabular}{l | c | c | c | c}
                    Tensor Shape & Time(s) & CR & RR & RTE \\
                    \hline
                    $8\times8\times8\times8\times8\times8\times5\times21\times1$ & 1158.7 & 32.68 & 32.94 & 0.098\\
                    $64\times64\times64\times5\times21\times1$ & 5289.4\wt & 18.86 & 32.59 & 0.108 \\
                    $4\times4\times4\times4\times4\times4\times4\times4\times4\times5\times21\times1$ & 3563.5 & 33.59 & 33.87 & 0.098 \\
                    {\footnotesize $2\times2\times4\times2\times2\times4\times2\times2\times4\times2\times2\times4\times2\times2\times5\times21\times1$ }& 11,339 & 25.64 & 28.83 & 0.099 \\
                \end{tabular}
            }
            \end{table}

            In \Cref{fig:pdebench_compression_reduction_reshaping_010} we observe that the changing the dimensionality of the tensor has a significant impact on the CR. The $4\times4\times4\times4\times4\times4\times4\times4\times4\times5\times21\times1$ reshaping results in the highest CR~($33.59\times$) and RR~($33.87\times$) among the investigated reshaping. The $64\times64\times64\times5\times21\times1$ reshaping results in the lowest CR~($18.86\times$) and the reshaping $2\times2\times4\times2\times2\times4\times2\times2\times4\times2\times2\times4\times2\times2\times5\times21\times1$ results in the lowest
            RR~($28.83\times$). Furthermore, the $64\times64\times64\times5\times21\times1$ reshaping runs into maximum walltime timeout and therefore fails to compress the entire dataset. 

            \begin{figure}[htbp]
                \centering
                \includegraphics[width=\textwidth]{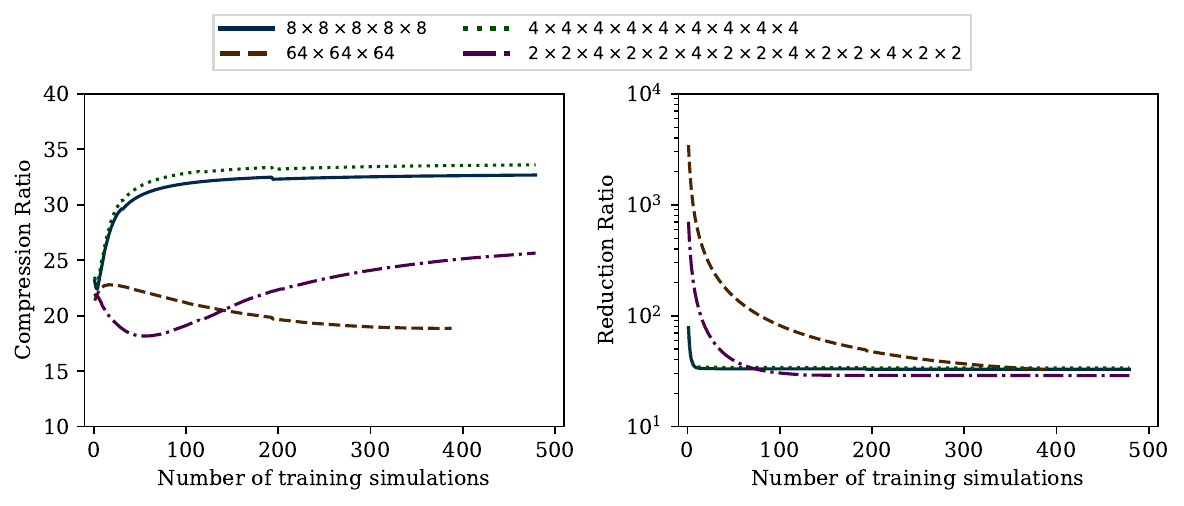}
                \caption{\small\emph{Comparison of various tensor reshapings in terms of compression ratio~(CR - left) and reduction ratio~(RR - right) on the PDEBench 3D Navier-Stokes dataset with $\varepsilon_{rel}=0.10$ using no normalization. The reshaping $4\times4\times4\times4\times4\times4\times4\times4\times4\times5\times21\times1$ results $1.8\times$ the CR of the worst performing reshaping, $64\times64\times64\times5\times21\times1$. Similarly, the reshaping $4\times4\times4\times4\times4\times4\times4\times4\times4\times5\times21\times1$ results in $1.2\times$ of the RR of the worst performing reshaping, $2\times2\times4\times2\times2\times4\times2\times2\times4\times2\times2\times4\times2\times2\times5\times21\times1$.}}
                \label{fig:pdebench_compression_reduction_reshaping_010}
            \end{figure}

            \Cref{fig:pdebench_time_validation_reshaping_010} shows that the baseline reshaping $8\times8\times8\times8\times8\times8\times5\times21\times1$ results in the lowest compression time among the investigated reshapings~($1158.7$s). The reshaping $2\times2\times4\times2\times2\times4\times2\times2\times4\times2\times2\times4\times2\times2\times5\times21\times1$ results in the highest compression time~($11,339$s). The reshaping $64\times64\times64\times5\times21\times1$ runs into maximum walltime timeout and therefore is not considered in this comparison. However, the growth in compression time suggests that it would have been the slowest reshaping among the investigated ones if the entire dataset was successfully compressed.

            Another observation from \Cref{fig:pdebench_time_validation_reshaping_010} is that all reshapings result in almost the same RTE~($0.098-0.099$) except for $64\times64\times64\times5\times21\times1$ reshaping. The $64\times64\times64\times5\times21\times1$ reshaping results in the highest RTE among the investigated reshapings~($0.108$). However, this is likekly due to the fact that the algorithm fails to compress the entire dataset due to maximum walltime limit. The effect of the reshaping on the RTE is at the number of simulations it takes to reduce the RTE below the target $\varepsilon_{rel}$. The $4\times4\times4\times4\times4\times4\times4\times4\times4\times5\times21\times1$ reshaping crosses the $\varepsilon_{rel}$ threshold at 12 simulations, whereas the $2\times2\times4\times2\times2\times4\times2\times2\times4\times2\times2\times4\times2\times2\times5\times21\times1$ reshaping crosses the $\varepsilon_{rel}$ threshold at 127 simulations.

            \begin{figure}[htbp]
                \centering
                \includegraphics[width=\textwidth]{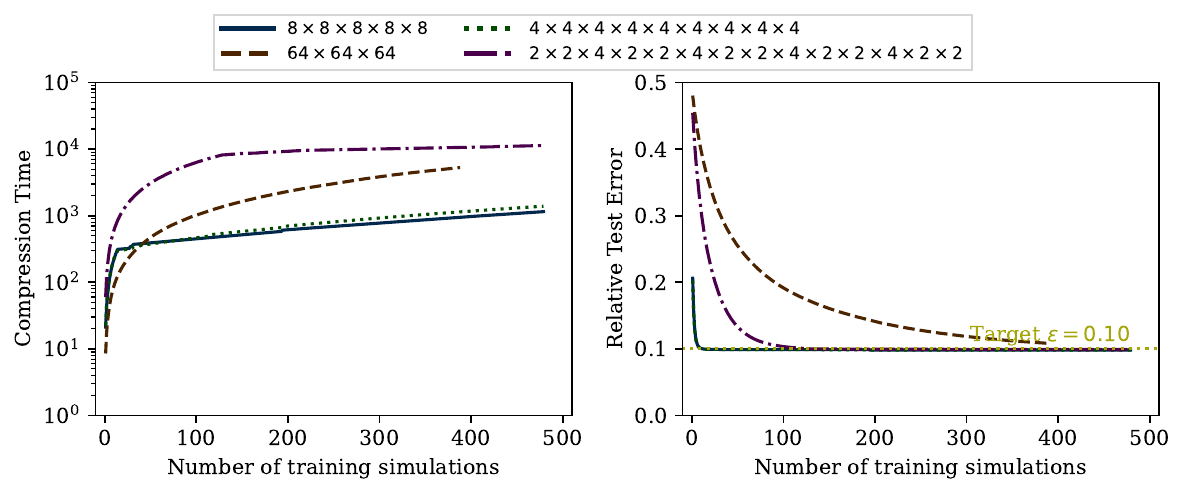}
                \caption{\small\emph{Comparison of various tensor reshapings in terms of compression time~(left) and relative test error~(RTE - right) on the PDEBench 3D Navier-Stokes dataset with $\varepsilon_{rel}=0.10$ using no normalization. The reshaping $2\times2\times4\times2\times2\times4\times2\times2\times4\times2\times2\times4\times2\times2\times5\times21\times1$ takes $9.8\times$ the time it takes the reshaping
                $8\times8\times8\times8\times8\times8\times5\times21\times1$ to compress the dataset. All reshapings except for $64\times64\times64\times5\times21\times1$ result in almost the same RTE.}}
                \label{fig:pdebench_time_validation_reshaping_010}
            \end{figure}
            
            \subsection{Effect of the axis ordering on compression performance}\label{app:axis_ordering}
            This section considers an empirical investigation of the effect of axis ordering on the compression performance of the \ouralgorithm~algorithm. 
            As the ordering of the axes determine the order of interatction between the dimensions, it is expected that the axis ordering has an effect on the compression performance of the \ouralgorithm~algorithm. We investigate the effect of axis ordering on the compression performance of the \ouralgorithm~algorithm using the PDEBench 3D turbulent Navier-Stokes dataset with $\varepsilon_{rel}=0.10$ and no normalization. We consider four different axis reorderings of the tensor. The baseline axis ordering is denoted as $[0,1,2,3,4,5,6,7,8]$. The other axis reorderings are obtained by permuting the indices of the tensor. \Cref{fig:pdebench_compression_reduction_transpose_010} shows the compression time and relative test error of the algorithms using various axis reorderings. \Cref{fig:pdebench_time_validation_transpose_010} shows the comparison of compression time and mean validation error of the algorithms using various axis reorderings.
            \Cref{tab:axis_ordering} presents the axis ordering, shapes of the corresponding tensors, and the resulting compression time, CR, RR, as well as RTE.

            \begin{table}[htbp]
                \centering
                \caption{\small\emph{Results of experiments investigating the effect of axis ordering on the compression performance of the \ouralgorithm~algorithm using the PDEBench 3D turbulent Navier-Stokes dataset with $\varepsilon_{rel}=0.10$. Axis orderings are given in the form of a permutation of the indices of the tensor. Two of the dimensions with magnitude 8 are emphasized to make different axis reorderings distinguishable. The investigated axis reorderings of the axes result in up to $1.72\times$ higher in compression time, CR and RR, while yielding almost the same RTE. Time: Compression time in seconds, CR: Compression ratio, RR: Reeduction ratio, RTE: Relative test error.}}
                \label{tab:axis_ordering}
                \resizebox{\textwidth}{!}{
                \begin{tabular}{l | c | c | c | c | c | c}
                   Name & Axis Ordering & Tensor Shape & Time(s) & CR & RR & RTE \\
                   \hline
                    Baseline & [0,1,2,3,4,5,6,7,8] & $8\times8\times \textbf{8} \times8\times\textit{8}\times8\times5\times21\times1$ & 1158.7 & 32.68 & 32.94 & 0.098 \\
                    Transpose A & [0,1,2,7,4,5,6,3,8] & $8\times8\times \textbf{8} \times21\times \textit{8} \times8\times5\times8\times1$ & 1306.7 & 26.24 & 26.41 & 0.098 \\
                    Transpose B & [0,4,2,7,1,5,6,3,8] & $8\times \textit{8} \times8\times21\times \textbf{8} \times8\times5\times8\times1$ & 1928.7 & 18.93 & 19.06 & 0.097 \\
                    Transpose C & [0,1,2,6,4,5,3,7,8] & $8\times8\times \textbf{8} \times5\times \textit{8} \times8\times8\times21\times1$ & 1992.9 & 25.59 & 26.21 & 0.098 \\
                \end{tabular}
                }
            \end{table}

            \Cref{fig:pdebench_compression_reduction_transpose_010} shows that the baseline axis ordering results in the highest CR and RR among the investigated axis reorderings. The baseline axis ordering achieves $32.68\times$ CR and $32.94\times$ RR, whereas the second best axis ordering, Transpose A, achieves $26.24\times$ CR and $26.41\times$ RR. CR and RR drop down to $18.93\times$ and $19.06\times$ for the worst performing axis ordering, Transpose B.

            \begin{figure}[htbp]
                \centering
                \includegraphics[width=\textwidth]{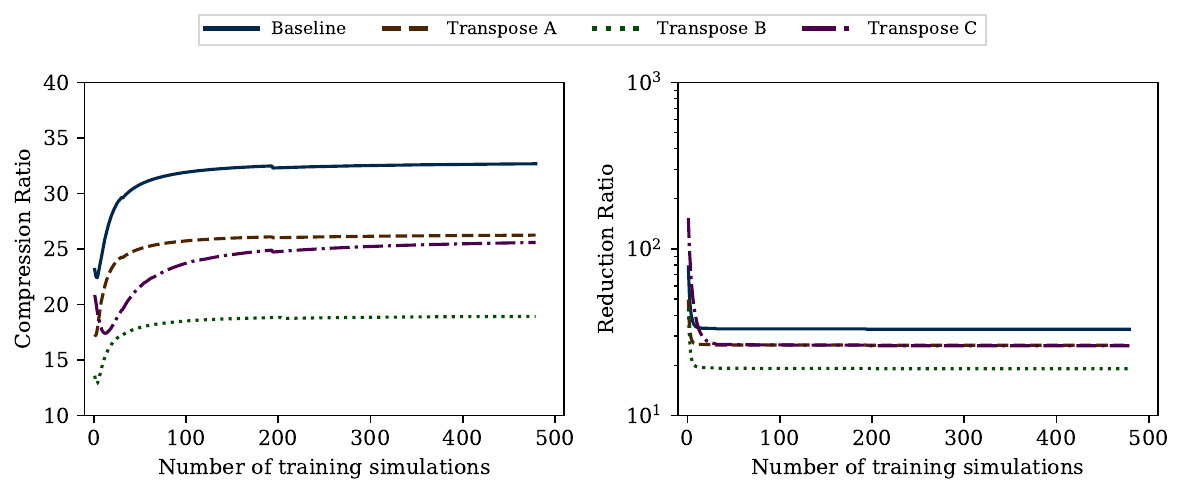}
                \caption{\small\emph{Comparison of various axis orderings in terms of compression ratio~(CR - left) and reduction ratio~(RR - right) on the PDEBench 3D Navier-Stokes dataset with $\varepsilon_{rel}=0.10$ using no normalization. The baseline axis ordering results in the highest CR and RR among the investigated axis reorderings. Changing the axis ordering of the tensor results in up to $1.72\times$ the CR and RR of the worst performing axis ordering.}}
                \label{fig:pdebench_compression_reduction_transpose_010}
            \end{figure}

            \Cref{fig:pdebench_time_validation_transpose_010} shows that the baseline axis ordering also results in the lowest compression time among the investigated axis reorderings. The baseline axis ordering achieves $1158.7$ seconds of compression time, whereas the second best axis ordering, Transpose A, results in $1306.7$ seconds of compression time. Compression time increases up to $1992.9$ seconds for worst performing axis ordering, Transpose C. In terms of RTE, all axis reorderings yield almost the same RTE around $0.098$.
            
            \begin{figure}[htbp]
                \centering
                \includegraphics[width=\textwidth]{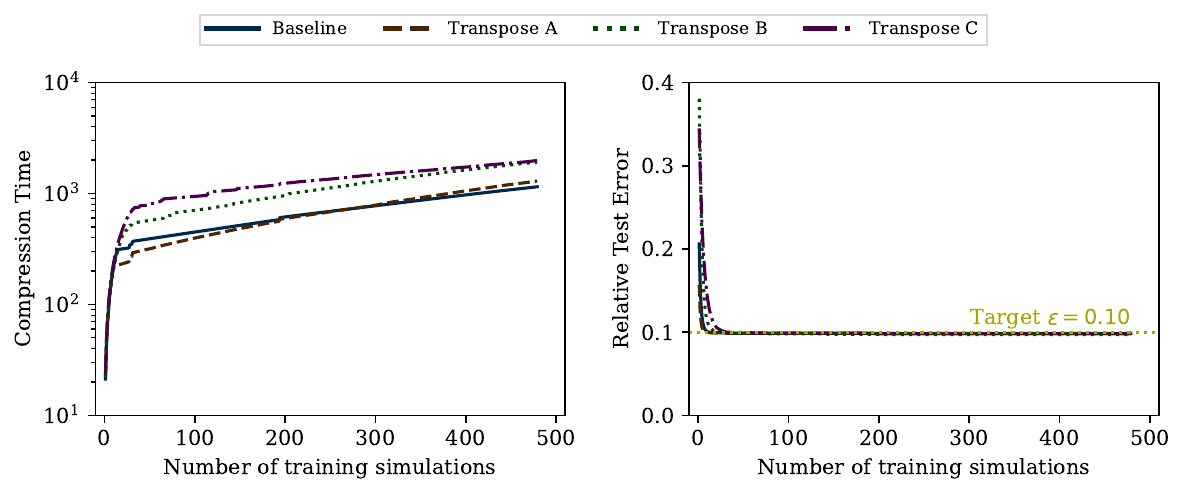}
                \caption{\small\emph{Comparison of various axis orderings in terms of compression time~(left) and relative test error~(RTE - right) on the PDEBench 3D Navier-Stokes dataset with $\varepsilon_{rel}=0.10$ using no normalization. The baseline axis ordering results in the lowest compression time among the investigated axis reorderings. Changing the axis ordering of the tensor results in up to $1.72\times$ the compression time of best performing axis ordering. All axis reorderings result in almost the same RTE.}}
                \label{fig:pdebench_time_validation_transpose_010}
            \end{figure}

        \subsection{Additional PDEBench results}\label{app:pdebench_results}
            This section presents the detailed results discussed in \Cref{sec:pdebench_experiments}. Similar to \Cref{tab:pdebench_results}, we present experiments on the PDEBench 3D turbulent Navier-Stokes simulations with $\varepsilon_{rel}=0.10$ and $\varepsilon_{rel}=0.05$. In contrast to \Cref{sec:pdebench_experiments}, here we provide results for all normalization methods used in the experiments. \Cref{fig:pdebench_compression_reduction_010,fig:pdebench_time_validation_010} show the results for experiments with $\varepsilon_{rel}=0.10$ and \Cref{fig:pdebench_compression_reduction_005,fig:pdebench_time_validation_005} show the results for experiments with $\varepsilon_{rel}=0.10$ and $\varepsilon_{rel}=0.05$.

            \begin{figure}[htbp]
                \centering
                \includegraphics[width=\textwidth]{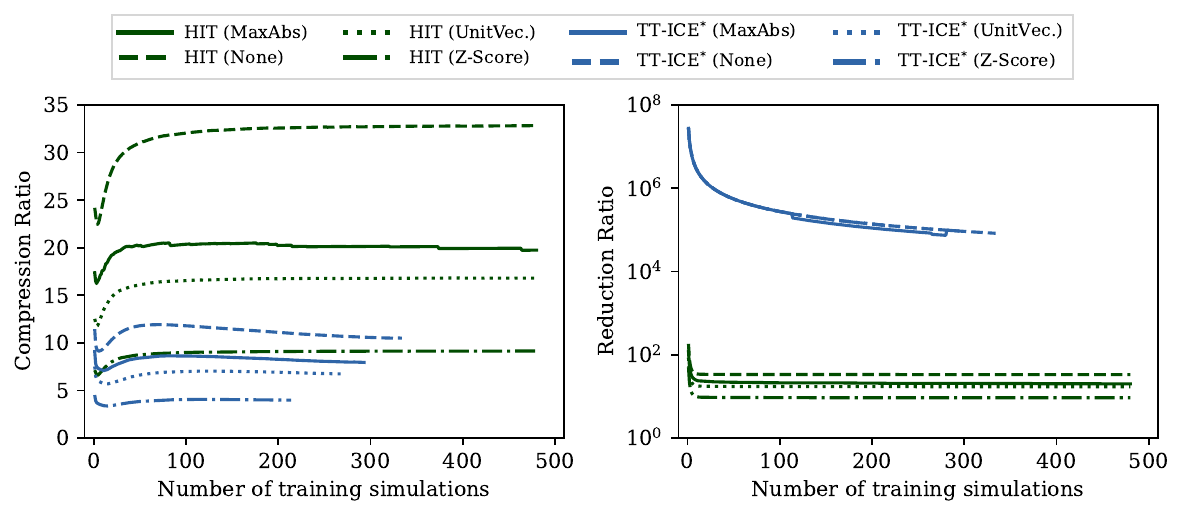}
                \caption{\small\emph{Compression time~(left) and Relative Test Error~(right) of the algorithms on the PDEBench 3D turbulent Navier-Stokes dataset with $\varepsilon_{rel}=0.10$ and using various normalization methods. \texttt{TT-ICE${}^{*}$} offers orders of magnitude higher reduction ratio while yielding less compression ratio compared to \ouralgorithm. All experiments with \texttt{TT-ICE${}^{*}$} terminates prematurely due to timeout. MaxAbs: Maximum absolute value normalization, None: No normalization, UnitVec: Unit vector normalization, ZScore: Z-score normalization. The results are averaged over 5 seeds.}}
                \label{fig:pdebench_compression_reduction_010}
            \end{figure}

            In parallel to the findings in \Cref{fig:pdebench_compression_reduction_simple}, \Cref{fig:pdebench_compression_reduction_010} shows that \ouralgorithm~results in higher compression ratio but multiple orders of magnitude lower reduction ratio against \texttt{TT-ICE${}^{*}$}. This discrepancy is caused by the fact that the maximum size of the latent space is upper bounded by the number of tensors in the accumulation for \texttt{TT-ICE${}^{*}$}. This limitation also results in the high RTE for \texttt{TT-ICE${}^{*}$} as it struggles to reduce the approximation error on the test set anywhere near the target $\varepsilon_{rel}$ in \Cref{fig:pdebench_time_validation_010}. On the other hand, \ouralgorithm~is able to reduce the RTE below the target $\varepsilon_{rel}$ error within a handful of training simulations across all normalization methods.
            As a result, \ouralgorithm~achieves faster compression time across all normalization methods due to the less updates required to reach the target $\varepsilon_{rel}$ error.

            \begin{figure}[htbp]
                \centering
                \includegraphics[width=\textwidth]{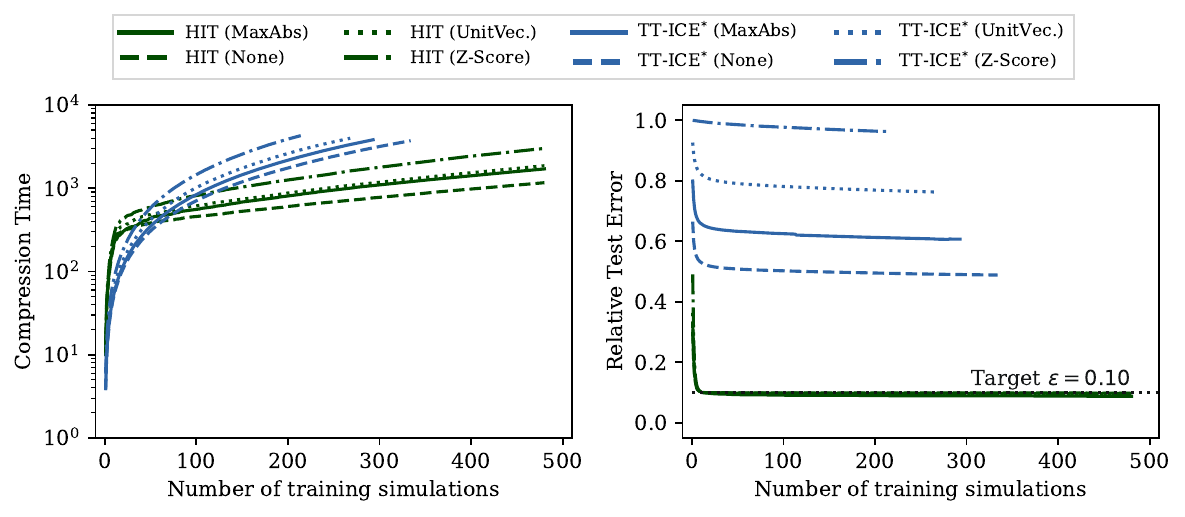}
                \caption{\small\emph{Comparison of compression time and mean validation error of algorithms using the PDEBench 3D Navier-Stokes dataset with $\varepsilon_{rel}=0.10$ and using various normalization methods. TT-ICE${}^{*}$ fails to reduce the approximation error on the test set whereas \ouralgorithm achieves below $\varepsilon_{rel}$ error within a handful of training simulations. This results in less updates to \ouralgorithm and therefore less compression time. All experiments with \texttt{TT-ICE${}^{*}$} terminates prematurely due to timeout. Please refer to the legend of \Cref{fig:pdebench_compression_reduction_010} for the normalization methods. The results are averaged over 5 seeds.}}
                \label{fig:pdebench_time_validation_010}
            \end{figure}

            A similar trend is observed in \Cref{fig:pdebench_compression_reduction_005} and \Cref{fig:pdebench_time_validation_005} for the experiments with $\varepsilon_{rel}=0.05$. \texttt{TT-ICE${}^{*}$} offers orders of magnitude higher reduction ratio while yielding less compression ratio compared to \ouralgorithm.
            Similarly, \texttt{TT-ICE${}^{*}$} fails to reduce the approximation error on the test set whereas \ouralgorithm~achieves below $\varepsilon_{rel}$ error within a handful of training simulations. This issue becomes a serious bottleneck as \texttt{TT-ICE${}^{*}$} hardly compresses 200 simulations out of 480 during the allocated 4-day maximum walltime. On the other hand, \ouralgorithm~only fails at z-score normalization due to insufficient memory.
            
            \begin{figure}[htbp]
                \centering
                \includegraphics[width=\textwidth]{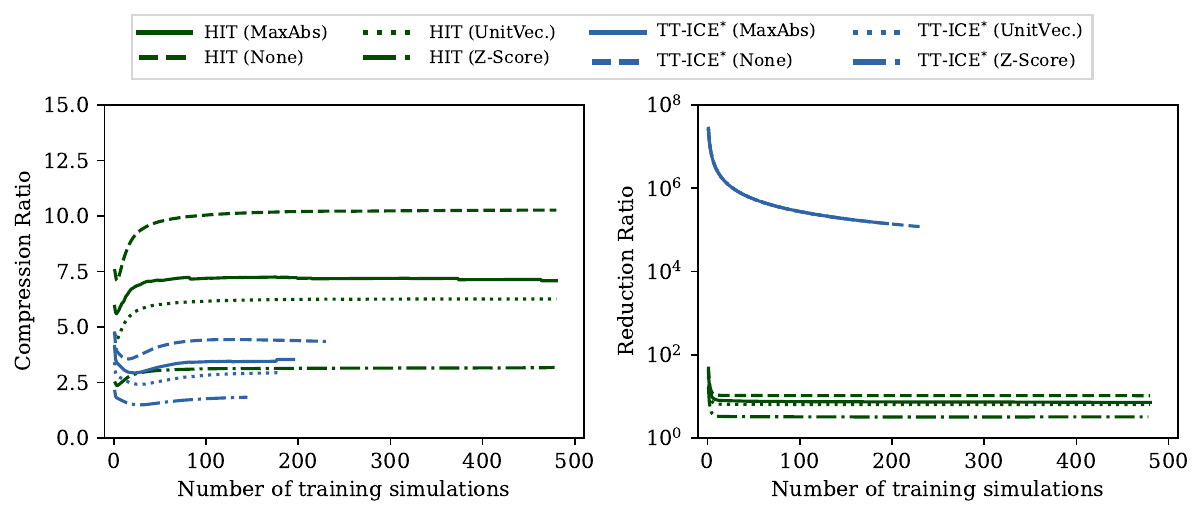}
                \caption{\small\emph{Compression time~(left) and Relative Test Error~(right) of the algorithms on the PDEBench 3D turbulent Navier-Stokes dataset with $\varepsilon_{rel}=0.05$ and using various normalization methods. \texttt{TT-ICE${}^{*}$} offers orders of magnitude higher reduction ratio while yielding less compression ratio compared to \ouralgorithm. All experiments with \texttt{TT-ICE${}^{*}$} terminates prematurely due to timeout. \ouralgorithm~only fails at z-score normalization due to insufficient memory. Please refer to the legend of \Cref{fig:pdebench_compression_reduction_010} for the normalization methods. The results are averaged over 5 seeds.}}
                \label{fig:pdebench_compression_reduction_005}
            \end{figure}
            \begin{figure}[htbp]
                \centering
                \includegraphics[width=\textwidth]{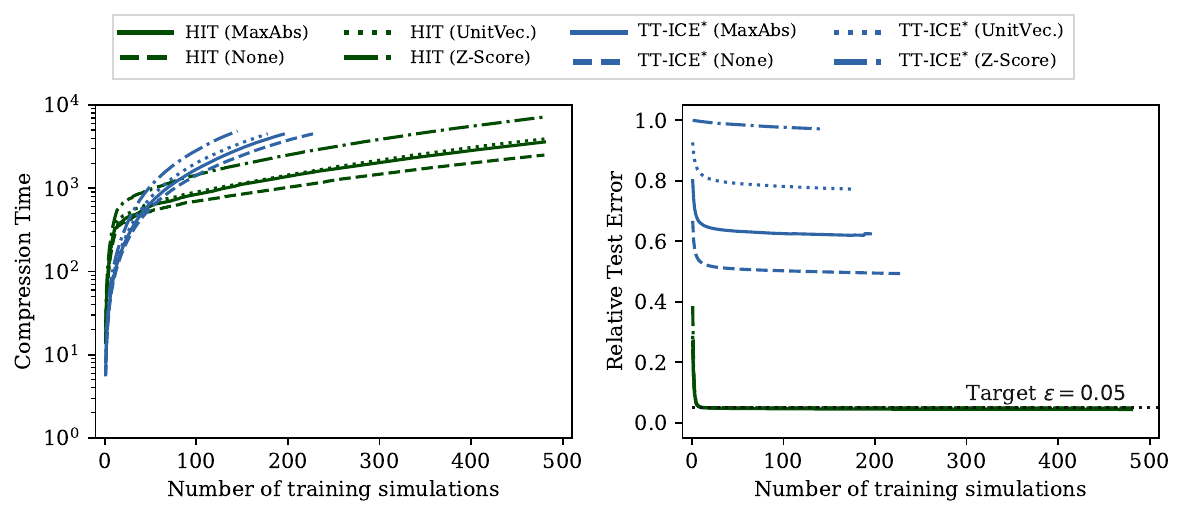}
                \caption{\small\emph{Comparison of compression time~(left) and Relative Test Error~(right) of algorithms using the PDEBench 3D Navier-Stokes dataset with $\varepsilon_{rel}=0.05$ and using various normalization methods. TT-ICE${}^{*}$ fails to reduce the approximation error on the test set whereas \ouralgorithm achieves below $\varepsilon_{rel}$ error within a handful of training simulations. This results in less updates to \ouralgorithm and therefore less compression time. All experiments with \texttt{TT-ICE${}^{*}$} terminates prematurely due to timeout. \ouralgorithm~only fails at z-score normalization due to insufficient memory. Please refer to the legend of \Cref{fig:pdebench_compression_reduction_010} for the normalization methods. The results are averaged over 5 seeds.}}
                \label{fig:pdebench_time_validation_005}
            \end{figure}

        \subsection{Additional self-oscillating-gel simulations results}\label{app:gels_results}
            This section presents the detailed results discussed in \Cref{sec:gel_experiments}. Similar to \Cref{tab:catgel_results}, we present experiments on the self oscillating gel simulations with $\varepsilon_{rel}=0.10$ and $\varepsilon_{rel}=0.01$. However, in this section we provide results for z-score normalization and unit vector normalization in addition to experiments without any normalization. \Cref{fig:catgel_compression_reduction_010,fig:catgel_time_validation_010} show the results for experiments with $\varepsilon_{rel}=0.10$ and \Cref{fig:catgel_compression_reduction_001,fig:catgel_time_validation_001} show the results for experiments with $\varepsilon_{rel}=0.01$.

            \Cref{fig:catgel_compression_reduction_010} shows the compression ratio and reduction ratio of the algorithms on the self-oscillating gel dataset with $\varepsilon_{rel}=0.10$ and using all investigated normalization methods. At this target relative error level, both algorithms successfully complete the compression task for all normalization methods. \texttt{TT-ICE${}^{*}$} offers $3.3-3.5\times$ the CR, and $4.6-6.3\times$ the RR of \ouralgorithm.

            \begin{figure}[htbp]
                \centering
                \includegraphics[width=\textwidth]{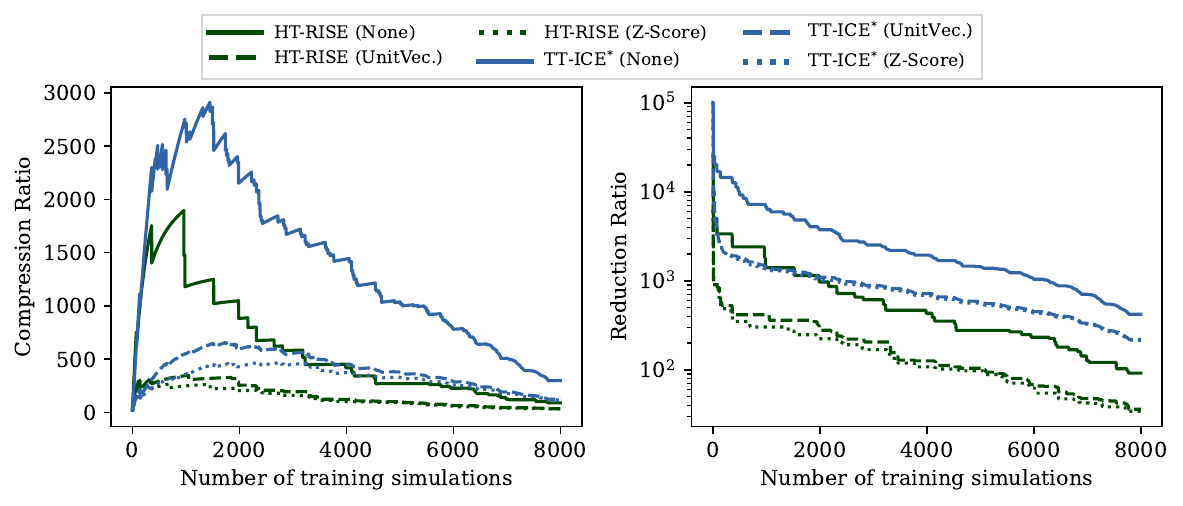}
                \caption{\small\emph{Compression ratio~(CR - left) and Reduction ratio~(RR - right) of the algorithms on the self-oscillating gel dataset with $\varepsilon_{rel}=0.10$ and using various normalization methods. \texttt{TT-ICE${}^{*}$} offers orders of magnitude higher reduction ratio while yielding less compression ratio compared to \ouralgorithm. All experiments with \texttt{TT-ICE${}^{*}$} terminates prematurely due to timeout. MaxAbs: Maximum absolute value normalization, None: No normalization, UnitVec: Unit vector normalization, ZScore: Z-score normalization. The results are averaged over 5 seeds.}}
                \label{fig:catgel_compression_reduction_010}
            \end{figure}

            \Cref{fig:catgel_time_validation_010} shows the compression time and relative test error of the algorithms on the self-oscillating gel dataset with $\varepsilon_{rel}=0.10$ with all investigated normalization methods. Both methods were able to reduce the approximation error below the target $\varepsilon_{rel}$ threshold for all normalziation methods within similar number of training simulations. Due to the low dimensionality of the problem, \texttt{TT-ICE${}^{*}$} resulted in a lower compression time. \ouralgorithm~takes $2.9-5.4\times$ the time it takes for \texttt{TT-ICE${}^{*}$} to compress the dataset.
            
            \begin{figure}[htbp]
                \centering
                \includegraphics[width=\textwidth]{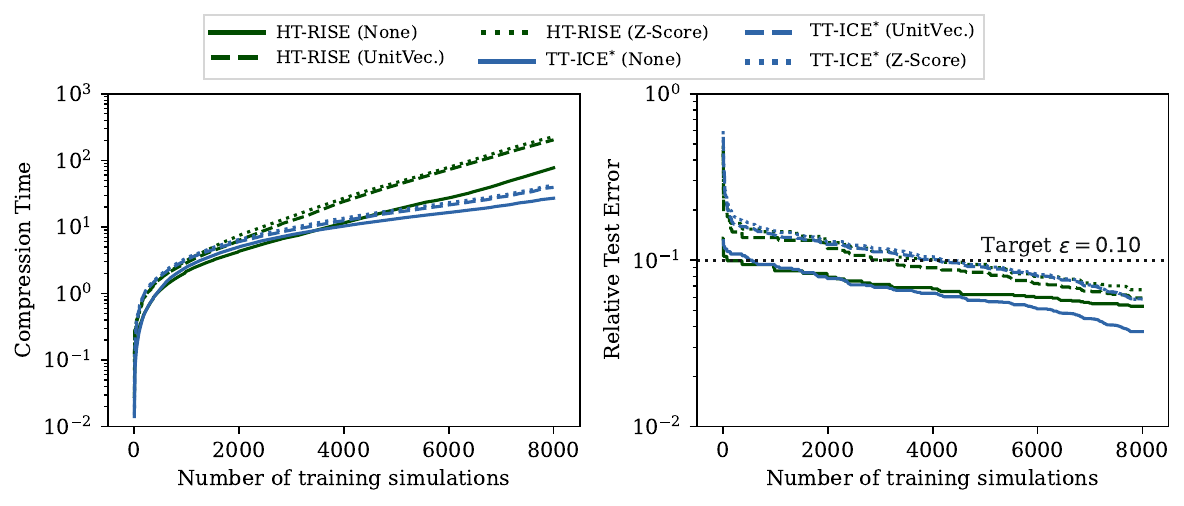}
                \caption{\small\emph{Compression time~(left) and Relative Test Error~(right) of the algorithms on the self-oscillating gel dataset with $\varepsilon_{rel}=0.10$ and using various normalization methods. TT-ICE${}^{*}$ offers orders of magnitude higher reduction ratio while yielding less compression ratio compared to \ouralgorithm. All experiments with \texttt{TT-ICE${}^{*}$} terminates prematurely due to timeout. Please refer to the caption of \Cref{fig:catgel_compression_reduction_010} for normalization methods. The results are averaged over 5 seeds.}}
                \label{fig:catgel_time_validation_010}
            \end{figure}

            \Cref{fig:catgel_compression_reduction_001} shows the compression ratio and reduction ratio of the algorithms on the self-oscillating gel dataset with $\varepsilon_{rel}=0.01$ and using all investigated normalization methods. As the target relative error level becomes tighter, the compression ratio and reduction ratio of the algorithms decrease significantly. Furthermore, \texttt{TT-ICE${}^{*}$} runs into timeout for unit vector and z-score normalizations. In contrast to that, \ouralgorithm~completes the compression task for all normalization methods without running into any issues. For all normalization methods \texttt{TT-ICE${}^{*}$} offers $2.5-6.6\times$ the CR of \ouralgorithm~and $9.5-21.8\times$ the RR of \ouralgorithm. 

            \begin{figure}[htbp]
                \centering
                \includegraphics[width=\textwidth]{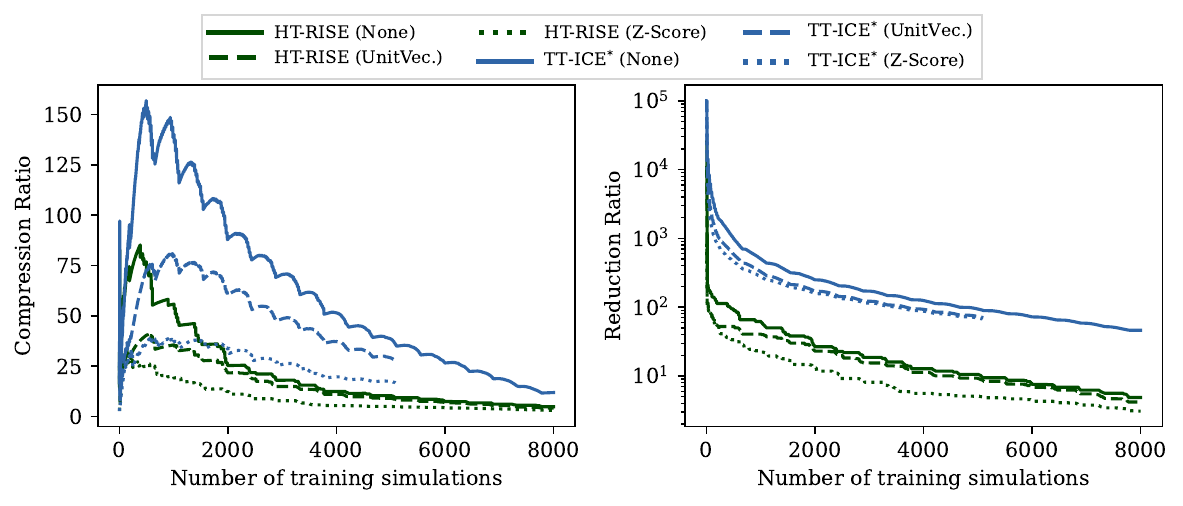}
                \caption{\small\emph{Compression ratio~(CR - left) and reduction ratio~(RR - right) of the algorithms on the self-oscillating gel dataset with $\varepsilon_{rel}=0.01$ and using various normalization methods. None: No normalization, UnitVec: Unit vector normalization, ZScore: Z-score normalization. TT-ICE${}^{*}$ offers an order of magnitude higher reduction ratio while yielding comparable compression ratio compared to \ouralgorithm.}}
                \label{fig:catgel_compression_reduction_001}
            \end{figure}

            \Cref{fig:catgel_time_validation_001} shows the compression time and relative test error of the algorithms on the self-oscillating gel dataset with $\varepsilon_{rel}=0.01$ and using all investigated normalization methods. Both methods were able to reduce the approximation error below the target $\varepsilon_{rel}$ threshold when no normalization is employed. In addition to that, \ouralgorithm~was able to cross the target $\varepsilon_{rel}$ line with z-score normalization as well. It is not fair to compare the methods in terms of total compression time as \texttt{TT-ICE${}^{*}$} fails to compress the entire dataset due to the maximum walltime limit at two of the normalization methods. However, it is worth noting that in \Cref{fig:catgel_time_validation_001} \texttt{TT-ICE${}^{*}$} runs faster than \ouralgorithm~for all normalization methods for same number of simulations. Note that only \ouralgorithm~with z-score normalization can cross the target $\varepsilon_{rel}$ line in addition to the cases with no normalization.

            \begin{figure}[htbp]
                \centering
                \includegraphics[width=\textwidth]{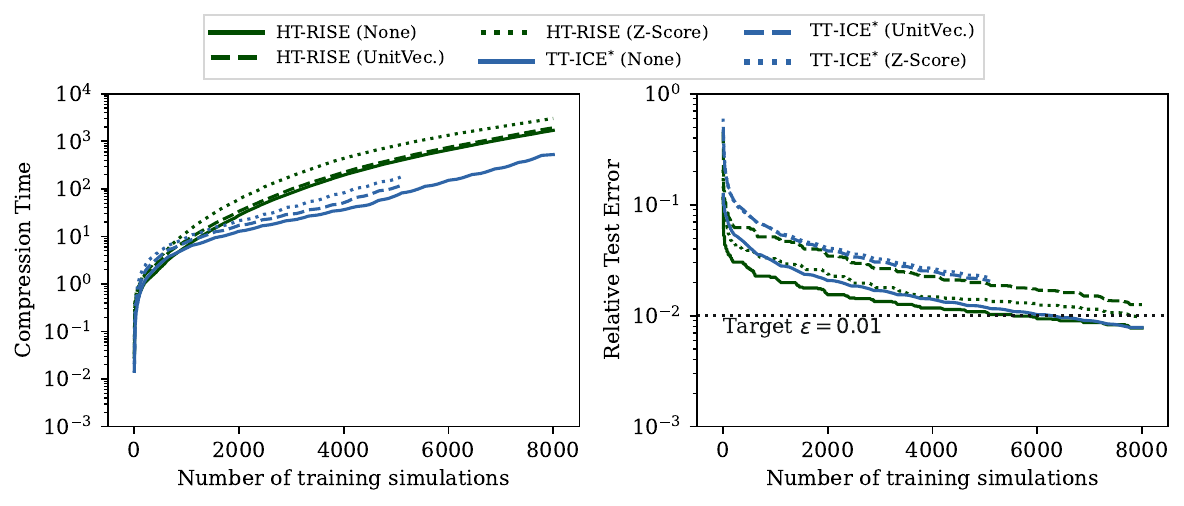}
                \caption{\small\emph{Comparison of compression time~(left) and Relative test error~(right) of algorithms using the Self-oscillating gel simulation dataset with $\varepsilon_{rel}=0.01$ and using various normalization methods. Both methods were able to reduce the approximation error below the target $\varepsilon_{rel}$ threshold when no normalization is employed. In addition to that, \ouralgorithm~was able to cross the target $\varepsilon_{rel}$ line with z-score normalization as well. Due to the low dimensionality of the problem, TT-ICE${}^{*}$ resulted in a lower compression time. Please refer to the legend of \Cref{fig:catgel_compression_reduction_001} for the normalization methods.}}
                \label{fig:catgel_time_validation_001}
            \end{figure}

        \subsection{Additional BigEarthNet results}\label{app:bigearthnet_results}
            This section presents the detailed results discussed in \Cref{sec:multispectral_frames}. Similar to \Cref{tab:bigearthnet_results}, we present experiments on the BigEarthNet dataset with $\varepsilon_{rel}=0.05-0.30$. \Cref{fig:bigearthnet_compression_reduction} presents the results of the experiments on compression ratio and reduction ratio of the algorithms whereas \Cref{fig:bigearthnet_time_validation} shows the results of the experiments considering compression time and relative test error of the algorithms.

            \begin{figure}[htbp]
                \centering
                \includegraphics[width=\textwidth]{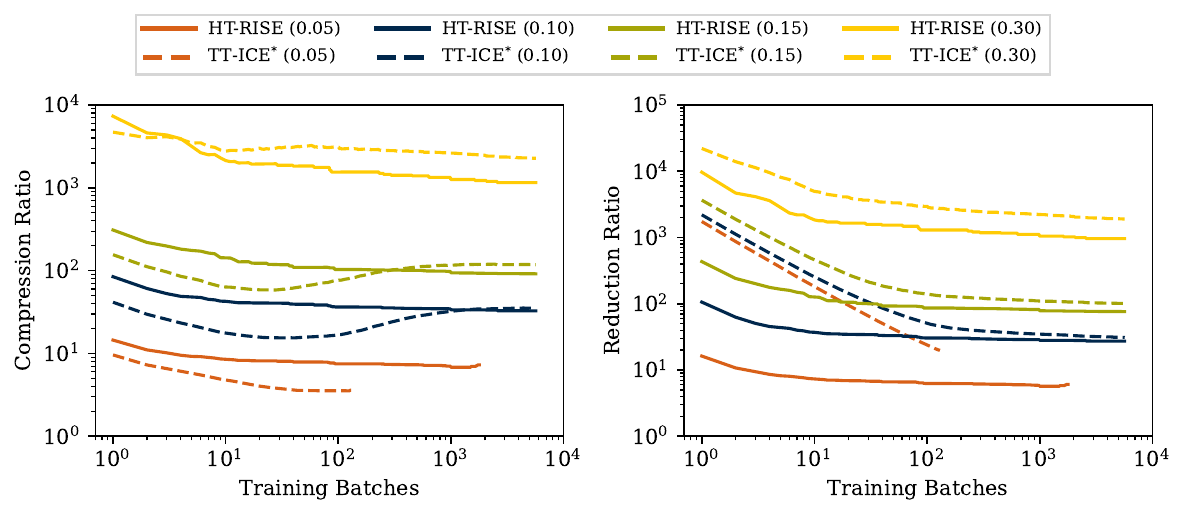}
                \caption{\small\emph{Compression ratio~(CR - left) and reduction ratio~(RR - right) of the algorithms on the BigEarthNet dataset with $\varepsilon_{rel}=0.05-0.30$. \texttt{TT-ICE${}^{*}$} offers higher CR~($1.1-2\times$ that of \ouralgorithm) for all target relative error levels but $\varepsilon_{rel}=0.05$. At $\varepsilon_{rel}=0.05$, \ouralgorithm~achieves $2\times$ the CR of \texttt{TT-ICE${}^{*}$} at the portion that it \texttt{TT-ICE${}^{*}$} is able to compress. \texttt{TT-ICE${}^{*}$} also offers higher RR~($1.1-2\times$ that of \ouralgorithm) for $\varepsilon_{rel}=0.10-0.30$. The results are averaged over 5 seeds.}}
                \label{fig:bigearthnet_compression_reduction}
            \end{figure}

            \Cref{fig:bigearthnet_compression_reduction} shows that at high target relative error levels, \texttt{TT-ICE${}^{*}$} offers higher CR and RR than \ouralgorithm. At $\varepsilon_{rel}=0.30$, \texttt{TT-ICE${}^{*}$} achieves $2274\times$ CR and $1901\times$ RR whereas \ouralgorithm~achieves $1154\times$ CR and $962\times$ RR. The discrepancy between two algorithms reduces as the target relative error level decreases. At $\varepsilon_{rel}=0.15$, \texttt{TT-ICE${}^{*}$} achieves $117\times$ CR and $100\times$ RR whereas \ouralgorithm~achieves $91\times$ CR and $76\times$ RR. At $\varepsilon_{rel}=0.10$, the performance of the two algorithms become even more comparable, where \texttt{TT-ICE${}^{*}$} achieves $35\times$ CR and $31\times$ RR whereas \ouralgorithm~achieves $32\times$ CR and $27\times$ RR. At these target relative error levels, neither algorithm stuggles to compress the dataset. However at $\varepsilon_{rel}=0.05$, \texttt{TT-ICE${}^{*}$} runs into maximum walltime timeout and therefore fails to compress the entire dataset and \ouralgorithm~runs into maximum memory limit and therefore fails to compress the entire dataset. The portion of the dataset that \texttt{TT-ICE${}^{*}$} is able to compress results in $3.59\times$ CR and $19.76\times$ RR whereas \ouralgorithm~results in $7.29\times$ CR and $6.07\times$ RR. However, we need to acknowledge that the portion of the dataset that \texttt{TT-ICE${}^{*}$} is able to compress is significantly smaller than the portion of the dataset that \ouralgorithm~is able to compress.

            \begin{figure}[htbp]
                \centering
                \includegraphics[width=\textwidth]{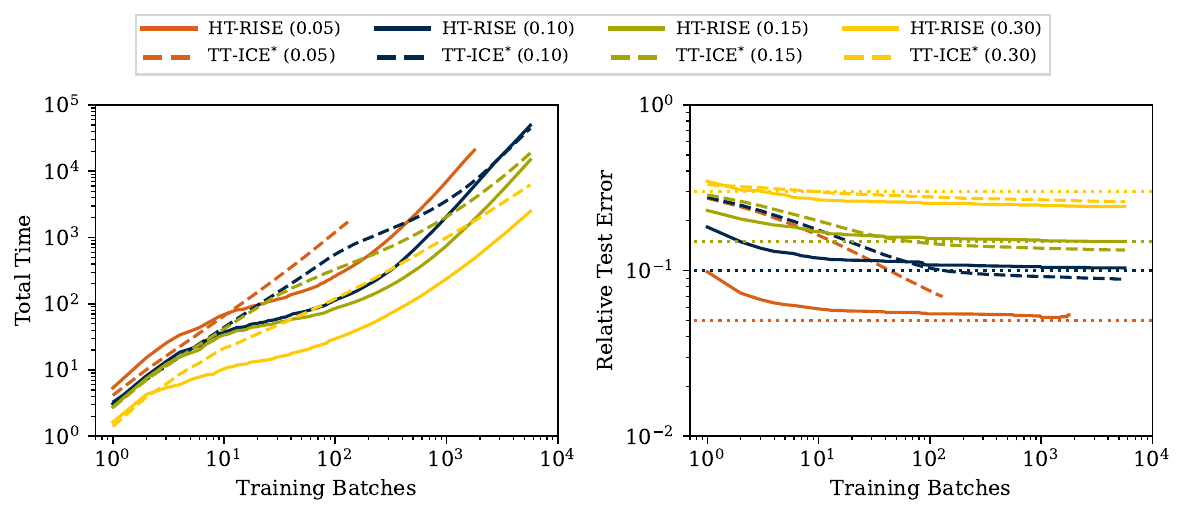}
                \caption{\small\emph{Compression time~(left) and relative test error~(RTE - right) of the algorithms on the BigEarthNet dataset with $\varepsilon_{rel}=0.05-0.30$. For all $\varepsilon_{rel}$ levels, \ouralgorithm~results in lower, if not comparable, compression time compared to \texttt{TT-ICE${}^{*}$}. For all $\varepsilon_{rel}$ levels, \ouralgorithm~is able to reduce the RTE faster than \texttt{TT-ICE${}^{*}$}. Each target relative error level is represented with horizontal dotted lines in their respective colors. The results are averaged over 5 seeds.}}
                \label{fig:bigearthnet_time_validation}
            \end{figure}
        
            \Cref{fig:bigearthnet_time_validation} shows that irrespective of the target relative error level, \ouralgorithm~results in lower, if not comparable, compression time compared to \texttt{TT-ICE${}^{*}$}. At $\varepsilon_{rel}=0.30$ \ouralgorithm~compresses the entire dataset in $2482$s whereas \texttt{TT-ICE${}^{*}$} compresses the entire dataset in $6239$s. At $\varepsilon_{rel}=0.15$ the difference becomes less pronounced as \ouralgorithm~compresses the entire dataset in $14,896$s whereas \texttt{TT-ICE${}^{*}$} compresses the entire dataset in $18,866$s. At $\varepsilon_{rel}=0.10$ \texttt{TT-ICE${}^{*}$} becomes slightly faster as \ouralgorithm~compresses the entire dataset in $49,408$s and \texttt{TT-ICE${}^{*}$} compresses the entire dataset in $44,687$s. As discussed above, at $\varepsilon_{rel}=0.05$ neither algorithm is able to compress the entire dataset. At $\varepsilon_{rel}=0.05$, \texttt{TT-ICE${}^{*}$} compresses 129 batches in $1720$s whereas in the same amount of time \ouralgorithm~compresses 443 batches. At the point of failure, \ouralgorithm~compresses 1781 batches in $20,880$s.

        \subsection{Qualitative MineRL results}\label{app:minerl_results}
            This section presents a qualitative comparison of the compressed frames using reconstructions. \Cref{fig:minerl_qualitative_results} shows the reconstructed video frames from the Basalt MineRL competition dataset using \texttt{TT-ICE${}^{*}$} and \ouralgorithm~algorithms with $\varepsilon_{rel}=0.10$ and $\varepsilon_{rel}=0.30$.
                \begin{figure}[htbp]
                    \centering
                    \begin{subfigure}{0.19\textwidth}\centering
                        \includegraphics[width=\columnwidth]{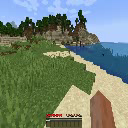}
                        \caption{\small\emph{Downsampled}}
                    \end{subfigure}
                    \begin{subfigure}{0.19\textwidth}\centering
                        \includegraphics[width=\columnwidth]{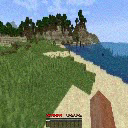}
                        \caption{\small\emph{$\varepsilon_{rel}=0.10$}}
                    \end{subfigure}
                    \begin{subfigure}{0.19\textwidth}\centering
                        \includegraphics[width=\columnwidth]{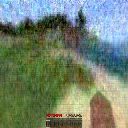}
                        \caption{\small\emph{TT $\varepsilon_{rel}=0.30$}}
                    \end{subfigure}
                    \begin{subfigure}{0.19\textwidth}\centering
                        \includegraphics[width=\columnwidth]{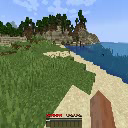}
                        \caption{\small\emph{HT $\varepsilon_{rel}=0.10$}}
                    \end{subfigure}
                    \begin{subfigure}{0.19\textwidth}\centering
                        \includegraphics[width=\columnwidth]{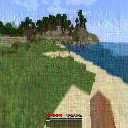}
                        \caption{\small\emph{HT $\varepsilon_{rel}=0.30$}}
                    \end{subfigure}
                    \caption{\small\emph{Reconstructed video frames from Basalt MineRL competition dataset using \texttt{TT-ICE${}^{*}$}~(TT) and \ouralgorithm~(HT) algorithms with $\varepsilon_{rel}=0.10$ and $\varepsilon_{rel}=0.30$. The downsampled frame is provided for baseline comparison. Visual quality of the reconstructed frames is comparable to the downsampled frames except for the case with \texttt{TT-ICE${}^{*}$} at $\varepsilon_{rel}=0.30$. Please refer to \Cref{tab:minecraft_results} and \Cref{fig:minerl_compression_reduction,fig:minerl_time_validation} for quantitative results.}}
                    \label{fig:minerl_qualitative_results}
                \end{figure}
            In \Cref{fig:minerl_qualitative_results}, we observe that \ouralgorithm~in fact results in a better visual quality compared to the frame of \texttt{TT-ICE${}^{*}$} of the same target relative error level. Despite the fact that \texttt{TT-ICE${}^{*}$} achieves almost $3\times$ the CR and the RR of \ouralgorithm, the visual quality of the frame reconstructed by \texttt{TT-ICE${}^{*}$} at $\varepsilon_{rel}=0.30$ is significantly worse than the downsampled frame. The reconstruction of \texttt{TT-ICE${}^{*}$} at $\varepsilon_{rel}=0.30$ is significantly blurrier compared to the reconstruction of \ouralgorithm~at $\varepsilon_{rel}=0.30$.
            In addition to that, we see that the visual quality of the reconstruction from \ouralgorithm~at $\varepsilon_{rel}=0.10$ is comparable, if not identical, to the downsampled frame even though \ouralgorithm~offers a CR and a RR around $1.85\times$. 
    \end{appendices}

\bibliography{references}
\end{document}